\documentclass[10pt, letterpaper]{amsart}
\usepackage{amsfonts, amscd, amsmath, mathrsfs, amssymb, amsthm, amsxtra, stmaryrd, bbding, epsfig, graphicx, latexsym, url,xcolor}
\definecolor{cite}{rgb}{0.00,0.60,1.00}
\definecolor{link}{rgb}{0.00,0.00,1.00}
\usepackage[colorlinks,linkcolor=link,urlcolor=url,citecolor=cite,breaklinks]{hyperref}
\usepackage{tikz-cd}
\usepackage{enumerate}

\newtheorem{theorem}[subsection]{Theorem}
\newtheorem{proposition}[subsection]{Proposition}
\newtheorem{corollary}[subsection]{Corollary}
\newtheorem{lemma}[subsection]{Lemma}
\newtheorem{definition}[subsection]{Definition}
\newtheorem{remark}[subsection]{Remark}

\numberwithin{equation}{subsection}

\theoremstyle{plain}

\renewcommand\b[1]{\mathbb{#1}}
\newcommand\f[1]{\mathfrak{#1}}
\newcommand\m[1]{\mathrm{#1}}
\renewcommand\c[1]{\mathcal{#1}}
\renewcommand\d{\dagger}
\newcommand{\Q}{\mathbb{Q}}
\newcommand{\D}{\mathbb{D}}

\begin{document}

\title{Arithmetic hypergeometric $\mathcal {D}$-modules and exponential sums on reductive groups}

\author[L. Fu]{Lei Fu}
\address{Yau Mathematical Sciences Center, Tsinghua University, Beijing 100084, P. R. China}
\email{leifu@tsinghua.edu.cn}

\author[X. Li]{Xuanyou Li}
\address{Qiuzhen College, Tsinghua University, Beijing 100084, P. R. China}
\email{lixuanyo21@mails.tsinghua.edu.cn}

\author[C. Liu]{Chenhan Liu}
\address{Department of Mathematics, Tsinghua University, Beijing 100084, P. R. China}
\email{liu-ch22@mails.tsinghua.edu.cn}

\date{}


\begin{abstract} For a finite family of representations of a reductive group, we define a Laurent polynomial on the group. 
The exponential sum associated this Laurent polynomial is called a hypergeometric exponential sum. We introduce an
arithmetic hypergeometric $\mathcal D$-module to study the hypergeometric exponential sum. It is an overholonomic arithmetic 
$\mathcal D$-module with a Frobenius structure so that the trace of the Frobenius at a rational point
is the exponential sum. We prove that over the locus where the Laurent polynomial is nondegenerate, 
the arithmetic hypergeometric $\mathcal {D}$-module defines an $F$-isocrystal overconvergent along the degenerate locus. 
As an application, we get an estimation of the hypergeometric exponential sum.
\end{abstract}

\maketitle

\section*{Introduction}

Let $p$ be a prime number, $q$ a power of $p$, $k$ a finite field with $q$ elements, $\psi:k\to\overline{\b{Q}}^*_p$ a nontrivial additive 
character, and $X$ a separated scheme of finite type over $k$. For any extension $k'$ of $k$, denote by $X(k')$ the set 
of $k'$-points in $X$. For any regular function $f\in\mathcal O_X(X)$, denote also by $f:X\to\mathbb A^1_k$ 
the morphism corresponding to the $k$-algebra homomorphism
$$k[t]\to \mathcal O_X(X),\quad t\mapsto f.$$ In number theory, many problems lead to the study of the exponential sum
$$S=\sum_{x\in X(k)} \psi(f(x)).$$ In this paper, we study the exponential sum when $X$ is a reductive algebraic group.

Let $G_k$ be a reductive group over the finite field $k$, and let $$\rho_{j,k}:G_k\to\m{GL}(V_{j,k})\quad
(j=1, \cdots, N)$$ a family of representations of $G_k$, where $V_{j,k}$ are finite dimensional vector spaces over $k$. For any $k$-point
$A=(A_1,\cdots,A_N)$ of $\prod_{j=1}^N\m{End}(V_{j,k})$, consider the morphism
$$f_A:G_k\to\b{A}_k^1,\quad g\mapsto\sum_{j=1}^N\m{Tr}(A_j\rho_{j,k}(g)).$$ Following \cite{l-adic}, 
we call $f_A$ a \emph{Laurent polynomial} 
on $G_k$, and define the \textit{hypergeometric exponential sum} associated with the representations $\rho_{j,k}$ $(j=1,\ldots, N)$ to be 
\begin{eqnarray}\label{eqn:hypexpsum}
\sum_{g\in G_k(k)}\psi\Big(\sum_{j=1}^N\m{Tr}(A_j\rho_{j,k}(g))\Big).
\end{eqnarray}
In \cite{l-adic}, we study this exponential sum using the theories of $\ell$-adic cohomology and algebraic $\mathcal {D}$-modules.
In \cite{padictorus}, we study the twisted GKZ hypergeometric exponential sum by the $p$-adic method, which corresponds to the
case where $G_k$ is the torus $\b{G}_{m,k}^d$. In this paper, we use the theory of arithmetic $\mathcal D$-modules to study the 
hypergeometric exponential sum on a reductive group $G_k$ which can be lifted to a split reductive group defined over a $p$-adic number field. 

Let $K$ be a $p$-adic number field containing an element $\pi$ satisfying 
$$\pi^{p-1}+p=0,$$
let $R$ be the discrete valuation ring in $K$, let $k$ be the residue field of $R$,
and let $G$ be a split reductive group $R$-scheme. Fix a Borel subgroup $R$-scheme $B^+$. Let $T$ be 
the maximal torus of $B^+$, let $\Lambda=\mathrm{Hom}_R(T, \mathbb G_m)$ be the group of characters of $T$,
let $$\rho_{j}: G\to\mathrm{GL}(V_j)\quad(j=1, \ldots, N)$$ be a family of representations with
$V_j$ being free $R$-modules of finite ranks, and let
$$V_j=\bigoplus_{\lambda\in \Lambda} V_j(\lambda)$$ be the weight decomposition of $V_j$ so that
$V_j(\lambda)$ is the component of weight $\lambda$ under the action of $T$.
We define the \emph{Newton
polytope at infinity} $\Delta_\infty$ 
for the family of representations $\rho_1, \ldots, \rho_N$ to be the convex hull in
$\Lambda_{\mathbb R}:=\Lambda\otimes \mathbb R$ of the zero weight and those weights $\lambda$ such 
that $V_j(\lambda)$ is nonzero 
for some $j\in\{1, \ldots, N\}$. Occasionally we also use the the \textit{Newton polytope} $\Delta$ which is defined to be
the convex hull of those weight $\lambda$ such that $V_j(\lambda)$ is nonzero 
for some $j\in\{1, \ldots, N\}$. Let $$\mathbb V=\prod_{j=1}^N\mathrm{End}(V_j)$$
For any face $\tau\prec \Delta_\infty$, define an $R$-point $e(\tau)=(e(\tau)_j)$ of $\mathbb V$ so that 
$e(\tau)_j$ is the block-diagonal linear transformation on $V_j$ defined by
$$e(\tau)_j\Big|_{ V_j(\lambda)}=\begin{cases}
\mathrm{id}_{V_j(\lambda)}&\hbox{if }\lambda\in \tau, \\
0&\hbox{otherwise}.
\end{cases}$$
Denote the $k$-point of $\mathbb V_k:=\mathbb V\otimes_R k$ corresponding to $e(\tau)$ by the same notation. 
Let $H=G\times G$. 
We have an action of
$$(G \times_R G)\times_R \mathbb V\to  \mathbb V, \quad
\Big((g, h), (A_1, \ldots, A_N)\Big)\mapsto  \Big(\rho_1(g) A_1 \rho_1(h^{-1}), \ldots, \rho_N(g) A_N\rho_N(h^{-1})\Big).$$

\begin{definition}\label{defn:nondeg} Let $\bar k$ be an algebraic closure of $k$ and let $\rho_{j,k}$ $(j=1,\ldots, N)$
be the base changes of $\rho_j$ to $k$. 
For any $\bar k$-point 
$A=(A_1, \ldots, A_N)$ in $\mathbb V_{\bar k}=\prod_{j=1}^N\mathrm{End}(V_{j,\bar k})$, let $\phi_A$ be the morphism
$$\phi_A:\mathbb V_{\bar{k}}\to
\mathbb A^1_{\bar{k}},\quad (B_1,\ldots, B_N)\mapsto \sum_{j=1}^N\mathrm{Tr}(A_jB_j).$$
The Laurent polynomial
\begin{equation}\label{generalLaurent}
f_A: G_{\bar{k}}\to\mathbb A^1_{\bar{k}}, \quad f(g)=\sum_{j=1}^N \mathrm{Tr}(A_j \rho_{j,k}(g))
\end{equation}
is called \emph{nondegenerate} if for any face $\tau$ of $\Delta_\infty$ not containing the origin, the restriction of 
$\phi_A$ to the $H_{\bar k}$-orbit of $G_{\bar k}e(\tau)G_{\bar k}$ has no critical point, that is, 
the function
\begin{equation}\label{ftauA}
f_{\tau, A}:G_{\bar{k}}\times_{\bar{K}} G_{\bar{k}}
\to\mathbb A^1_{\bar{k}}, \quad f_{\tau, A}(g, h)=\sum_{j=1}^N \mathrm{Tr}(A_j \rho_j(g)e(\tau)_j\rho_j(h^{-1}))
\end{equation}
has no critical point, which means that $\mathrm df_{\tau, A}=0$ has no solution
on $G_{\bar{k}}\times_{\bar{k}} G_{\bar{k}}$. 
\end{definition}

Let $\mathbb V_{k}^{\mathrm{gen}}$ be the set
consisting of those $A$ so that the Laurent polynomial (\ref{generalLaurent}) is nondegenerate.
Then $\mathbb V_{k}^{\mathrm{gen}}$ is a Zariski open subset of $\mathbb V_{k}$ parametrizing nondegenerate Laurent polynomials.
The main result of this paper is the following. 

\begin{theorem}\label{thm:expsum} Notation as above.
Suppose the morphism $$\iota: G\to\prod_{j=1}^N \mathrm{End}(V_j),\quad g\mapsto (\rho_1(g),
\ldots, \rho_N(g))$$ 
is quasi-finite. 
For any finite extension $k'$ of $k$ with $q'$ element, and any $k'$-point $A=(A_1, \ldots, A_N)$
in $\prod_{j=1}^N \mathrm{End}(V_{j,k})$ such that the Laurent polynomial 
$f_A(g)=\sum_{j=1}^N \mathrm{Tr}(A_j\rho_j(g))$ is nondegenerate,
we have
$$\Big\vert \sum_{g\in G(k')} \psi \Big(\mathrm{Tr}_{k'/k}\Big(\sum_{j=1}^N 
\mathrm{Tr}(A_{j}\rho_j(g))\Big)\Big)\Big\vert\leq  q'^{\frac{d}{2}}d!\int_{\Delta_\infty\cap\mathfrak C} 
\prod_{\alpha \in R^+} 
\frac{\lambda(H_\alpha)^2}{\rho(H_\alpha)^2}\mathrm d\lambda,$$ where 
$d=\mathrm{dim}\,G$, $\mathfrak C$ is the dominant Weyl chamber in 
$\Lambda_{\mathbb R}$, $R^+$ is the set of positive roots, 
$\{H_\alpha:\alpha\in R\}$ is the set of co-roots, 
$\rho=\frac{1}{2}\sum_{\alpha\in R^+}\alpha$, and $\mathrm d\lambda$ is the 
Lebesgue measure on $\Lambda_{\mathbb R}$
normalized by $\mathrm{vol}(\Lambda_{\mathbb R}/\Lambda)=1$. 
\end{theorem}

If $G=\mathbb G_m^d$, we recover the estimate of the exponential sum of Adolphson-Sperber (\cite{AS}) obtained by Dwork's theory
and Denef-Loeser (\cite{DL}) obtained by $\ell$-adic cohomology theory. Our tool is the arithmetic $\mathcal D$-module
theory developed by Berthelot and Caro which is a vast generalization of Dwork's theory. 

In Section 1, after briefly recalling the six-functor formalism on the derived categories of overholonomic arithmetic 
$\mathcal D$-modules, we define the hypergeometric $\mathcal D$-module $\mathrm{Hyp}_{\pi,!}$ with a Frobenius structure
so that the trace of Frobenius at $A$ is the hypergeometric exponential sum (\ref{eqn:hypexpsum}). In section 2, 
we show $\mathrm{Hyp}_{\pi,!}$ is a direct summand of an explicitly described arithmetic $\mathcal D$-module, 
which we call the modified hypergeometric $\mathcal D$-module. After describing invariant differential operators in Section 3, we 
construct in Section 4 a formal model for the modified hypergeometric $\mathcal D$-module, and prove it is an 
$F$-isocrystal on $\mathbb V^{\mathrm{gen}}_k$ overconvergent along the degenerate locus. Theorem \ref{thm:expsum} then follows 
from the trace formula and the $p$-adic version of Deligne's theorem on weights proved by Abe-Caro. In the appendix, we prove 
several technical results used in the paper.

\subsection{Acknowledgements} We thank Shang Li for discussion on integral models of toroidal equivariant embeddings. 
This research is supported by the 
National Key R\&D Program of China 2023YFA1009703.

\section{Arithmetic hypergeometric $\mathcal D$-modules}

\subsection{Arithmetic $\mathcal D$-modules} 

We briefly recall the definition of the sheaf of rings of 
arithmetic differential operators and refer to \cite{Berthelot1} for details.
Let $\mathfrak X$ be a smooth formal scheme over $\mathrm{Spf}\,R$, $\mathfrak m$ the maximal ideal of
$R$, $X_i$ 
the reduction of $\mathfrak X$ modulo $\mathfrak m^{i+1}$, 
and $T$ a divisor on $X_0$. Let $\mathfrak U\subset \mathfrak X$ 
be an affine open subset so that the divisor $T$ of $X_0$ is given 
by $h\equiv 0 \mod \mathfrak m$ for some $h\in\mathcal O_{\mathfrak X}(\mathfrak U)$, 
and let $U_i$ and $h_i\in\mathcal O_{X_i}(U_i)$ 
be the reduction of $\mathfrak U$ and $h$ modulo $\mathfrak m^{i+1}$, respectively.  For any $m\geq 0$, define
\begin{eqnarray*}
&&\mathcal B_{X_i}^{(m)}(T)|_{U_i}:=\mathcal O_{U_i}[t]/(h_i^{p^{m+1}}t-p),\\
&& \widehat{\mathcal B}_{\mathfrak X}^{(m)}(T)|_{\mathfrak U}
:=\varprojlim_i\mathcal B_{X_i}^{(m)}(T)|_{U_i}\cong\mathcal O_{\mathfrak U}\{t\}/(h^{p^{m+1}}t-p).
\end{eqnarray*}
$\widehat{\mathcal B}_{\mathfrak X}^{(m)}(T)$ is a $p$-adically complete $\mathcal O_{\mathfrak X}$-algebra 
depending only on $\mathfrak X$ and $T$. Define the sheaf of functions on $\mathfrak X$ with overconvergent singularities along $T$ by
$$\mathcal O_{\mathfrak X, \mathbb Q}(^\dagger T)=
\varinjlim_{m}\widehat{\mathcal B}_{\mathfrak X}^{(m)}(T)\otimes_{\mathbb Z}\mathbb Q.$$
We have 
$$\mathcal O_{\mathfrak X, \mathbb Q}(^\dagger T)(\mathfrak U)
=\bigcup_{s>1} \Big\{\sum_{j\in\mathbb Z_{\geq 0}}\frac{a_{j}}{h^{j+1}}:\, a_{j}\in 
\mathcal O_{\mathfrak X, \mathbb Q}(\mathfrak U), 
\,  \lim_{j\to\infty}\Vert a_{j}\Vert s^{j} =0\Big\},$$
where $\mathcal O_{\mathfrak X, \mathbb Q}=\mathcal O_{\mathfrak X}\otimes_{\mathbb Z}\mathbb Q$, and $\Vert\cdot\Vert$ is the quotient norm
on the Tate algebra $\mathcal O_{\mathfrak X, \mathbb Q}(\mathfrak U)$.
Denote by $\mathcal D_{X_i}^{(m)}$ the sheaf differential operators of level $m$ on $X_i$. If $(x_1,\ldots, x_n)$ is a 
local coordinate on $U_i$, we have
$$\mathcal D_{X_i}^{(m)}(U_i)=\Big\{\sum_{\mathbf k\in \mathbb Z^n_{\geq 0}}\mathbf{q}_{\mathbf k}^{(m)}! a_{\mathbf k}
\partial^{[\mathbf k]}:\, a_{\mathbf k}\in \mathcal O_{X_i}(U_i), \, a_{\mathbf k}\not =0 \text{ for only finitely many } \mathbf k\Big\},$$
where $\partial^{[\mathbf k]}=\frac{1}{k_1!\cdots k_n!}\partial_{x_1}^{k_1}\ldots\partial_{x_n}^{k_n}$
for any $\mathbf k=(k_1,\ldots,k_n)$, and $\mathbf{q}_{\mathbf k}^{(m)}$ is determined by the 
expression $\mathbf k=p^m\mathbf{q}_{\mathbf k}^{(m)}
+\mathbf{r}_{\mathbf k}^{(m)}$ with $0\le \mathbf r_{\mathbf k, j}^{(m)}<p^m$ for all $j$.
The sheaf of differential operators of level $m$ on $\mathfrak X$ is defined by
$$\widehat{\mathcal D}_{\mathfrak X}^{(m)}=\varprojlim_i
\mathcal D_{X_i}^{(m)}, \quad \widehat{\mathcal D}_{\mathfrak X, \mathbb Q}^{(m)}=
\widehat{\mathcal D}_{\mathfrak X}^{(m)}\otimes_{\mathbb Z} \mathbb Q.$$ We have
$$\widehat{\mathcal D}_{\mathfrak X, \mathbb Q}^{(m)}(\mathfrak U)
=\Big\{\sum_{\mathbf k\in\mathbb Z^n_{\geq 0}}\mathbf{q}_{\mathbf k}^{(m)}!a_{\mathbf k}\mathbf{\partial}^{[\mathbf k]}: \,
a_{\mathbf k}\in \mathcal O_{\mathfrak X, \mathbb Q}(\mathfrak U), \,
\lim_{\vert {\mathbf k}\vert\to\infty }a_{\mathbf k}=0\Big\}.$$ 
Let 
$$\mathcal D^\dagger_{\mathfrak X, \mathbb Q}
=\varinjlim_{m}\widehat{\mathcal D}_{\mathfrak X,\mathbb Q}^{(m)}.$$
We have 
$$\mathcal D^\dagger_{\mathfrak X, \mathbb Q}(\mathfrak U)
=\bigcup_{s>1} \Big\{\sum_{\mathbf k\in\mathbb Z^n_{\geq 0}}a_{\mathbf k}\mathbf{\partial}^{[\mathbf k]}: \, a_{\mathbf k}\in 
\mathcal O_{\mathfrak X, \mathbb Q}(\mathfrak U), \, \lim_{\vert \mathbf k\vert\to\infty}\Vert a_{\mathbf k}\Vert s^k=0\Big\}.$$
Define the sheaf of differential operators overconvergent along $T$ as
$$\mathcal D_{\mathfrak X}^\dagger(^\dagger T):=
\varinjlim_m\widehat{\mathcal D}_{\mathfrak X}^{(m)}(T)
:=\varinjlim_m\widehat{\mathcal B}_{\mathfrak X}^{(m)}(T)\widehat{\otimes}_{\mathcal O_{\mathfrak X}}
\widehat{\mathcal D}_{\mathfrak X}^{(m)},\quad \mathcal D_{\mathfrak X, \mathbb Q}^\dagger(^\dagger T):=
\mathcal D_{\mathfrak X}^\dagger(^\dagger T)\otimes_{\mathbb Z}\mathbb Q.$$ We have
$$\mathcal D_{\mathfrak X, \mathbb Q}^\dagger(^\dagger T)(\mathfrak U)
=\bigcup_{s>1} \Big\{\sum_{j, \mathbf k}\frac{a_{j, \mathbf k}}{h^{j+1}}\partial^{[\mathbf k]}:\,
a_{j, \mathbf k}\in \mathcal O_{\mathfrak X, \mathbb Q}(\mathfrak U), 
\,  \lim_{j+\vert\mathbf k\vert\to \infty}\Vert a_{j, \mathbf k}\Vert s^{j+|\mathbf k|}=0\Big\}.$$
Denote by $F\hbox{-}D_{\text{coh}}^b(\mathcal D_{\mathfrak X, \mathbb Q}^\dagger(^\dagger T))$ the derived category 
of $\mathcal D_{\mathfrak X, \mathbb Q}^\dagger(^\dagger T)$-modules with coherent cohomologies and with 
Frobenius structures. For any object $\mathcal E$ in 
$F\hbox{-}D_{\text{coh}}^b(\mathcal D_{\mathfrak X, \mathbb Q}^\dagger(^\dagger T))$,
define its Verdier dual to be 
$$\mathbb D_T(\mathcal E)=R\mathcal Hom_{\mathcal D_{\mathfrak X, \mathbb Q}^\dagger(^\dagger T)}
(\mathcal G, \mathcal D_{\mathfrak X, \mathbb Q}^\dagger(^\dagger T)\otimes_{\mathcal O_{\mathfrak X, \mathbb Q}}\omega^{-1}
_{\mathfrak X, \mathbb Q})[\mathrm{dim}\, \mathfrak X]),$$ where $\omega_{\mathfrak X, \mathbb Q}$ is the right 
$\mathcal D^\dagger_{\mathfrak X,\mathbb Q}$-module of top differential forms. 
Define the functor $(^\dagger T)$ to be 
\begin{eqnarray*}
(^\dagger T): 
D_{\text{coh}}^b(\mathcal D_{\mathfrak X, \mathbb Q}^\dagger)\to D_{\text{coh}}^b
(\mathcal D_{\mathfrak X, \mathbb Q}^\dagger(^\dagger T)),\quad
(^\dagger T)(\mathcal E):=\mathcal D_{\mathfrak X, \mathbb Q}^\dagger
(^\dagger T)\otimes_{\mathcal D_{\mathfrak X, \mathbb Q}^\dagger}\mathcal E.
\end{eqnarray*}
Let $f: \mathfrak X'\to \mathfrak X$ be a morphism of formal $R$-schemes, let $f_0: X'_0\to X_0$ be $f\mod \mathfrak m$, 
and let $T'$ be a divisor on $X'_0$ such that $f_0(X'_0-T')\subset X_0-T$.  Define
\begin{align*}
\widehat{\mathcal D}_{\mathfrak X'\to\mathfrak X}^{(m)}(T', T)&:=
\widehat{\mathcal B}_{\mathfrak X'}^{(m)}(T')\widehat{\otimes}_{f^{-1}\mathcal O_{\mathfrak X}}f^{-1}
\widehat{\mathcal D}_{\mathfrak X}^{(m)},\\
\mathcal D_{\mathfrak X'\to\mathfrak X, \mathbb Q}^{\dagger}(^\dagger T', T)
&:= \varinjlim_m\widehat{\mathcal D}_{\mathfrak X'\to\mathfrak X}^{(m)}(T', T)\otimes_{\mathbb Z}\mathbb Q.
\end{align*}
Then $\mathcal D_{\mathfrak X'\to\mathfrak X, \mathbb Q}^{\dagger}(^\dagger T', T)$  
is a $(\mathcal D_{\mathfrak X', \mathbb Q}^{\dagger}(^\dagger T'), f^{-1}\mathcal D_{\mathfrak X, \mathbb Q}^{\dagger}(^\dagger T))$-bimodule.
We define $\mathcal D_{\mathfrak X\leftarrow \mathfrak X, \mathbb Q}^{\dagger}(^\dagger T, T')$ to be the
$(f^{-1}\mathcal D_{\mathfrak X, \mathbb Q}^{\dagger}(^\dagger T), \mathcal D_{\mathfrak X', \mathbb Q}^{\dagger}(^\dagger T'))$-bimodule
obtained from $\mathcal D_{\mathfrak X'\to\mathfrak X, \mathbb Q}^{\dagger}(^\dagger T', T)$ by side-change. 
For any objects 
$\mathcal E\in \mathrm{ob}\,F\hbox{-}D_{\mathrm{coh}}^b(\mathcal D_{\mathfrak X, \mathbb Q}^\dagger(^\dagger T))$ and 
$\mathcal E'\in \mathrm{ob}\,F\hbox{-}D_{\mathrm{coh}}^b(\mathcal D_{\mathfrak X', \mathbb Q}^\dagger(^\dagger T'))$
define
\begin{eqnarray*}
f_{(T,T'),+}(\mathcal E'):&=&Rf_*(\mathcal D_{\mathfrak X\leftarrow \mathfrak X', \mathbb Q}^{\dagger}(^\dagger T, T')
\otimes_{\mathcal D_{\mathfrak X', \mathbb Q}^\dagger(^\dagger T')}^L\mathcal E'),\\
f_{(T',T)}^!(\mathcal E):&=&\mathcal D_{\mathfrak X'\to\mathfrak X, \mathbb Q}^{\dagger}(^\dagger T', T)
\otimes_{f^{-1}\mathcal D_{\mathfrak X, \mathbb Q}^\dagger(^\dagger T)}^Lf^{-1}\mathcal E
[\mathrm{dim}\, {\mathfrak X'}-\mathrm{dim}\,{\mathfrak X}].
\end{eqnarray*}
If $T$ and $T'$ are empty, we write $f_+$ and $f^!$ for $f_{(T,T'),+}$ and $f_{(T',T)}^!$. 
Following \cite[1.16 and 3.1]{Carosurhol}, we say $\mathcal E$ is $0$-overholonomic if for any smooth
morphism $f: \mathfrak X'\to \mathfrak X$ and any divisor $T'$ of $\mathfrak X'$, 
$(^\dagger T')(f^! \mathcal E)$ lies in $F\hbox{-} D^b_{\mathrm{coh}}(\mathcal D^\dagger_{\mathfrak X'}(^\dagger f_0^{-1}(T)))$. 
For any positive integer $r$, we say $\mathcal E$ is $r$-overholonomic if 
$\mathcal E$ is $(r-1)$-overholonomic and for any smooth morphism $f:\mathfrak X'\to \mathfrak X$ and any divisor 
$T'$ of $X'$, $\mathbb D(^\dagger T')f^!\mathcal E$ is $(r-1)$-overholonomic. We say $\mathcal E$ is \emph{overholonomic}
if it is $r$-overholonomic for all $r$.
Denote by 
$F\hbox{-}D^b_{\mathrm{ovhol}}(\mathcal D^\dagger_{\mathfrak X, \mathbb Q}(^\dagger T))$ the full subcategory 
consisting of overholonomic objects. 

Assume furthermore that $\mathfrak X$ is proper over $R$. Let $U=X_0-T$. We define  
$$F\hbox{-}D^b_{\mathrm{ovhol}}(U/K):= F\hbox{-}D^b_{\mathrm{ovhol}}(\mathcal D^\dagger_{\mathfrak X, \mathbb Q}(^\dagger T)).$$
This definition only depends on $U$ and is independent of the choice of the 
formal compactification $\mathfrak X$ of $U$ (\cite[4.14]{Carosurhol}). 
We define the Verdier duality functor $\mathbb D$ on $F\hbox{-}D^b_{\mathrm{ovhol}}(U/K)$ to be $\mathbb D_T$. 
Let $\bar f:\mathfrak X'\to \mathfrak X$ be a morphism of formal schemes such that 
both $\mathfrak X$ and $\mathfrak X'$ are proper over $R$, $T$ and $T'$ divisors of $X_0$ and $X'_0$ such that 
$\bar f_0(U')\subset U$, where $U=X_0-T$, $U'=X'_0-T'$. Let 
$f_0: U'\to U$ the morphism induced by $\bar f$. We define the functor
$$f_{0+}: F\hbox{-}D^b_{\mathrm{ovhol}}(U'/K)\to F\hbox{-}D^b_{\mathrm{ovhol}}(U/K)\quad (\hbox{resp. }
 f_0^!: F\hbox{-}D^b_{\mathrm{ovhol}}(U/K)\to F\hbox{-}D^b_{\mathrm{ovhol}}(U/K))$$
 to be the functor $\bar f_{(T,T')+}$ (resp. $\bar f^!_{(T',T)}$), and we define
 $$f_{0}^+= \mathbb D \circ f^!_0\circ \mathbb D, \quad f_{0!}= \mathbb D\circ f_{0+}\circ \mathbb D.$$
These definitions depend only on the morphism $f_0$ and are independent of the choice of the compactification $\bar f$ of $f_0$. 
We refer to \cite[1.3.14]{AC18padicweight} for details of the six-functor formalism for overholonomic arithmetic $\mathcal D$-modules. 
The $\mathcal D^\dagger_{\mathfrak X, \mathbb Q}(^\dagger T)$-module 
$\mathcal O_{\mathfrak X, \mathbb Q}(^\dagger T)$ defines an object in $F\hbox{-}D^b_{\mathrm{ovhol}}(U/K)$ which we denote by 
$\mathcal O^\dagger_U$. Denote by $\mathcal O^\dagger_U(d)$ the 
$\mathcal D^\dagger_{\mathfrak X, \mathbb Q}(^\dagger T)$-module 
$\mathcal O_{\mathfrak X, \mathbb Q}(^\dagger T)$ with the Frobenius structure 
$$F^*\mathcal O_{\mathfrak X, \mathbb Q}(^\dagger T)\to \mathcal O_{\mathfrak X, \mathbb Q}(^\dagger T), \quad s\mapsto q^{-d}s.$$
For any object $\mathcal E$ in $F\hbox{-}D^b_{\mathrm{ovhol}}(U/K)$, we define its $d$-th Tate twist $\mathcal E(d)$ to be 
$$\mathcal E(d):=\mathcal E\otimes_{\mathcal O_U^\dagger}\mathcal O_U^\dagger(d).$$
Denote by $\omega^\dagger_U$ the object in $F\hbox{-}D^b_{\mathrm{ovhol}}(U/K)$
defined by the right overholonomic $\mathcal D^\dagger_{\mathfrak X, \mathbb Q}(^\dagger T)$-module 
$\omega_{\frak X}\otimes_{\mathcal O_{\mathfrak X}}\mathcal O_{\mathfrak X, \mathbb Q}(^\dagger T)$. 

\subsection{Arithmetic $\mathcal D$-modules on the affine space and the Fourier transform}

Let $\mathbb P^n$ be the projective space over $R$ of relative dimension $n$, $[x_0:\ldots, :x_n]$ a homogeneous coordinate 
on $\mathbb P^n$, and $\infty$ the divisor $x_0=0$. We have an open immersion
$$\mathbb A^n\hookrightarrow\mathbb P^n,\quad (x_1,\ldots, x_n)\mapsto [1:x_1:\ldots:x_n]$$ 
and $\mathbb A^n=\mathbb P^n-\infty$. Let $\widehat{\mathbb P}^n$ and $\widehat{\mathbb A}^n$ 
the formal schemes obtained by taking completions of $\mathbb P^n$ and $\mathbb A^n$. Then the global section of 
$\mathcal D_{\widehat{\b{P}}^n,\b{Q}}^\dagger(^\dagger\infty)$ is
$$\mathcal D_{\widehat{\b{P}}^n,\b{Q}}^\dagger(^\dagger\infty)(\widehat{\mathbb P}^n)
\cong\bigcup_{s>1}\Big\{\sum_{\mathbf i,\mathbf k\in\mathbb Z^n_{\geq 0}}
a_{\mathbf i\mathbf k} \mathbf x^{\mathbf i}\partial^{[\mathbf k]}: \,a_{\mathbf i\mathbf j}\in K,\,
\lim_{|\mathbf i|+|\mathbf k|\to\infty}|a_{\mathbf i\mathbf k}|s^{|\mathbf i|+|\mathbf k|}=0\Big\},$$
where $\mathbf x^{\mathbf i}=x_1^{i_1}\cdots x_n^{i_n}$ and ${\partial}^{[\mathbf k]}=\frac{1}{k_1!\cdots k_n!}\partial^{k_1}_{x_1}
\cdots \partial^{k_n}_{x_n}$.
Denote this ring by $D_{\widehat{\b{P}}^n,\b{Q}}^\dagger(^\dagger\infty)$. We have the following 
theorem of Noot-Huyghe \cite[5.3.3]{Ddaggeraffine}.

\begin{proposition}\label{prop:Huyghe}
The functor $\Gamma(\widehat{\b{P}}^n,\hbox{-})$ defines an equivalence from the category of 
coherent $\mathcal {D}_{\widehat{\b{P}}^n,\b{Q}}^\dagger(^\dagger\infty)$-modules to the category of coherent 
$D_{\widehat{\b{P}}^n,\b{Q}}^\dagger(^\dagger\infty)$-modules.
\end{proposition}

Let $\widehat {\mathbb P}^{n*}$ be the dual formal projective space of $\widehat {\mathbb P}^{n}$, and let $[x'_0:\ldots:x'_n]$
be the dual homogeneous coordinate on ${\mathbb P}^{n*}$. We have an isomorphism of rings
\begin{eqnarray}\label{eqn:Fourieriso}
F_\pi:D_{\widehat{\b{P}}^{n*},\b{Q}}^\dagger(^\dagger\infty)\to D_{\widehat{\b{P}}^{n},\b{Q}}^\dagger(^\dagger\infty),
\quad x'_i\mapsto -\partial_{x_i}/\pi, \quad \partial_{x'_i}\mapsto\pi x_i \quad (i=1,\ldots, n).
\end{eqnarray}
It transform a $D_{\widehat{\b{P}}^n,\b{Q}}^\dagger(^\dagger\infty)$-module $E$ to a 
$D_{\widehat{\b{P}}^{n*},\b{Q}}^\dagger(^\dagger\infty)$-module 
$\mathfrak F_\pi(E)$, which we call the \emph{Fourier transform} of $E$. As a vector space over $K$, 
$\mathfrak F_\pi(E)$ coincides with $E$. The left multiplication by $x'_i$ (resp. $\partial_{x'_i}$) on
$\mathfrak F_\pi(E)$ coincides with the left multiplication by $ -\partial_{x_i}/\pi$ (resp. $\pi x_i$) 
on $E$. For any coherent $\mathcal {D}_{\widehat{\b{P}}^n,\b{Q}}^\dagger(^\dagger\infty)$-module $\c{E}$, let
$E=\Gamma(\widehat{\b{P}},\c{E})$. We define the \emph{Fourier transform} $\mathfrak F_\pi(\mathcal E)$ of 
$\mathcal E$ to be the coherent $\mathcal {D}_{\widehat{\b{P}}^{n*},\b{Q}}^\dagger(^\dagger\infty)$-module 
corresponding to the $D_{\widehat{\b{P}}^{n*},\b{Q}}^\dagger(^\dagger\infty)$-module $\mathfrak F_\pi(E)$. 
It can be extended to a functor on the derived category
$$\mathfrak F_\pi: D^b_{\mathrm{coh}}(\mathcal {D}_{\widehat{\b{P}}^n,\b{Q}}^\dagger(^\dagger\infty)) \to 
D^b_{\mathrm{coh}}(\mathcal {D}_{\widehat{\b{P}}^{n*},\b{Q}}^\dagger(^\dagger\infty))$$ which is exact with
respect to the $t$-structure defined in \cite[1.2.9]{AC18padicweight}.

In the case $n=1$, consider the $D_{\widehat{\b{P}}^1,\b{Q}}^\dagger(^\dagger\infty)$-module
$$L_\pi=\bigcup_{s>1}\Big\{\sum_{i\in\mathbb Z_{\geq 0}}
a_i x^i: a_i\in K,\,
\lim_{i\to \infty} |a_i|s^i=0\Big\}$$ 
so that for any $g\in L_\pi$, we have 
$$\partial_x \cdot g=\Big(\exp(\pi x)\circ\frac{\mathrm d}{\mathrm dx}\circ \exp(-\pi x)\Big)(g)=\frac{\mathrm dg}{\mathrm dx} -\pi g.$$
Denote by $\mathcal L_\pi$ the coherent $\mathcal {D}_{\widehat{\b{P}}^1,\b{Q}}^\dagger(^\dagger\infty)$-module
corresponding to $L_\pi$. Let $F:\widehat {\mathbb A}^1\to \widehat{\mathbb A}^1$ be the 
Frobenius morphism corresponding to the $R$-homomorphism 
$$R\{x\}\to R\{x\},\quad x\mapsto x^q.$$
The $D_{\widehat{\b{P}}^1,\b{Q}}^\dagger(^\dagger\infty)$-module 
corresponding to $F^*\mathcal L_\pi$ is given by 
\begin{eqnarray*}
&& F^*L_\pi=\bigcup_{s>1}\Big\{\sum_{i\in\mathbb Z_{\geq 0}}
a_i  x^i: \, a_i\in K,\,
\lim_{i\to\infty}|a_i|s^i=0\Big\},\\
&& \partial_x \cdot g=\Big(\exp(\pi x^q) \circ\frac{\mathrm d}{\mathrm dx}\circ \exp(-\pi x^q)\Big)(g)
=\frac{\mathrm dg}{\mathrm dx} -q\pi x^{q-1}g.
\end{eqnarray*}
Recall that $$\theta(x)=\exp(\pi x-\pi x^q)$$ is an overconvergent power series (\cite[Theorem 4.1]{Monsky1970PadicAA}). 
We define the Frobenius structure $F^*\mathcal L_\pi\to\mathcal L_\psi$ on $\mathcal L_\pi$ to be 
the morphism of $\mathcal D_{\widehat{\b{P}}^1,\b{Q}}^\dagger(^\dagger\infty)$-modules corresponding to the homomorphism 
of $D_{\widehat{\b{P}}^1,\b{Q}}^\dagger(^\dagger\infty)$-module 
$$F^*L_\pi \to L_\pi , \quad g\mapsto \exp(\pi x-\pi x^q) g.$$
$\mathcal L_\pi$ is the $\mathcal D_{\widehat{\b{P}}^1,\b{Q}}^\dagger(^\dagger\infty)$-module corresponding to the Dwork overconvergent
$F$-isocrystal. It is overholonomic by \cite[2.3.16]{CT12}. Note that the iteration of the Frobenius structure
$$F^m: F^{m*}\mathcal L_\pi \to \cdots\to  F^* \mathcal L_\pi \to \mathcal L_\pi$$ 
corresponds to the homomorphism of $D_{\widehat{\b{P}}^1,\b{Q}}^\dagger(^\dagger\infty)$-module 
$$F^{m*}L_\pi \to L_\pi , \quad g\mapsto \exp(\pi x-\pi x^{q^m}) g.$$

Let $\mathbb A_k^n$ be the affine space, 
$\mathbb A_k^{n*}$ the dual affine space, 
$$\langle\,,\,\rangle: \mathbb A_k^n\times_k \mathbb A_k^{n*}\to \mathbb A^1_k,\quad ((x_1,\ldots, x_n),(x'_1,\ldots, x'_n))\mapsto \sum_i{x_ix'_i}$$ 
the duality pairing, $p_1: \mathbb A_k^n\times_k \mathbb A_k^{n*}\to \mathbb A^n_k$ and
$p_2:\mathbb  A_k^n\times_k \mathbb A_k^{n*}\to \mathbb A^{n*}_k$ the projections. 
Noot-Huyghe (\cite[5.3.1 (iii)]{Huyghe2004Trans}) shows that the Fourier transform can be defined by
$$\mathfrak F_\pi: D^b_{\mathrm{ovhol}}(\mathbb A_k^n)\to D^b_{\mathrm{ovhol}}(\mathbb A_k^{n*}),
\quad \mathfrak F_\pi(\mathcal E)=p_{2,+}(p_1^!\mathcal E\widetilde\otimes_{\mathcal O^\dagger_{\mathbb A_k^n\times\mathbb A_k^n}} 
\langle\,,\,\rangle^! \mathcal L_\pi)[1-n].$$ Note that our definition of the Fourier transform differs from 
\cite[3.2.1]{Huyghe2004Trans} by a shifting. 

\subsection{Hypergeometric $\mathcal D$-modules}

Let $G$ be a split reductive group scheme over $R$, $\rho_{j}: G\to\mathrm{GL}(V_j)$ $(j=1, \ldots, N)$ 
representations of $G$, and $\mathbb V=\prod_{j=1}^N \mathrm{End}(V_j)$. We denote the base changes to $k$ and $K$ 
of objects over $R$ by the same notation with a subscript $k$ and $K$, respectively. 
Let $f$ be the $k$-morphism 
$$f: G_k\times_k \mathbb V_k\to \mathbb A_k^1, \quad (g, (A_1,\ldots, A_N))\mapsto \sum_{j=1}^N \mathrm{Tr}(A_j\rho_{j,k}(g)),$$
and let $\pi_2:  G_k\times_k \mathbb V_k\to \mathbb V_k$ be the projection. We define the \emph{hypergeometric arithmetic 
$\mathcal D$-modules} to be the objects 
$$\mathrm{Hyp}_{\pi,+}:=\pi_{2+}f^!\mathcal L_\pi[1-d-n], \quad \mathrm{Hyp}_{\pi,!} :=\pi_{2!}f^+\mathcal L_\pi[d+n-1],$$
in $F\hbox{-}D^b_{\mathrm{ovhol}}(\mathbb V_k)$, where $d=\mathrm{dim}\,G$ and $n=\mathrm{dim}\,\mathbb V$. Using 
\cite[3.12 Corollary]{Abe14calculation}, one can show 
$$\mathbb D(\mathrm{Hyp}_{\pi, +} )\cong \mathrm{Hyp}_{-\pi,!}(-1).$$
Let $\mathbb P:=\mathbb P(\mathbb A^1\times \mathbb V)$ be the projective space
containing $\mathbb V$, $\infty=\mathbb P_k-\mathbb V_k$, and $\widehat{\mathbb P}$ the completion of $\mathbb P$. 
Then $\mathrm{Hyp}_{\pi,+}$ and $\mathrm{Hyp}_{\pi,!}$
are objects in $F\hbox{-}D^b_{\mathrm{ovhol}}(\mathcal D^\dagger_{\widehat{\mathbb P},\mathbb Q}(^\dagger\infty))$. 
Let $$\mathrm{sp}: \widehat{\mathbb P}\otimes_R K\to \widehat{\mathbb P}$$ be the specialization map,
$X$ an open subset of $\mathbb P(\mathbb A^1_k\times _k\mathbb V_k)$ such that its complement $T$ is a divisor,
$F\text{-isoc}^\dagger(X)$ the category of 
$F$-isocrystals on $X$ overconvergent along $T$ as defined in \cite[D\'efinition 2.3.2]{B3}, and 
$F\text{-Coh} (\mathcal D^\dagger_{\widehat{\mathbb P}, \mathbb Q}(^\dagger T))$ the category of coherent 
$\mathcal D^\dagger_{\widehat{\mathbb P}, \mathbb Q}(^\dagger T)$-modules with Frobenius structures. 
We have a fully faithful functor induced by the specialization morphism
$$\text{sp}_\ast: F\text{-isoc}^\dagger(X) \to F\text{-Coh} (\mathcal D^\dagger_{\widehat{\mathbb P}, \mathbb Q}(^\dagger T)).$$
By \cite[Th\'eor\`eme 2.2.12]{Caro-L}
the essential image consists of coherent $\mathcal D^\dagger_{\widehat{\mathbb P}, \mathbb Q}(^\dagger T)$-modules with Frobenius structures 
whose restriction to $\widehat{\mathbb P}-T$ is
$(\mathcal O_{\widehat{\mathbb P} ,\mathbb Q})|_{\widehat{\mathbb P}-T}$-coherent. We also call an object in 
$F\text{-Coh} (\mathcal D^\dagger_{\widehat{\mathbb P}, \mathbb Q}(^\dagger T))$ which lies in the essential image of 
$\mathrm{sp}_*$ a convergent $F$-isocrystal on $X$ overconvergent 
along $T$, or simply an \emph{overconvergent $F$-isocrystal} on $X$.
Our main result about the hypergeometic $\mathcal D$-modules is the following.

\begin{theorem}\label{thm:hyp} Suppose the morphism $$\iota: G\to \mathbb V,\quad g\mapsto (\rho_1(g), \ldots, \rho_N(g))$$  
is quasi-finite. Then $\mathrm{Hyp}_{\pi, +}$ and $\mathrm{Hyp}_{\pi, !}$ are
overholonomic arithmetic $\mathcal D$-modules in the sense of \cite[1.2.9]{AC18padicweight}. In particular, we have 
$$\mathcal H^j(\mathrm{Hyp}_{\pi,+})=0 \quad 
(\hbox{resp. }\mathcal H^j(\mathrm{Hyp}_{\pi,!})=0)$$  for $j\not=0$.
Restricting to the open subset $\mathbb V_k^{\mathrm{gen}}$ of $\mathbb V_k$ 
parametrizing nondegenerate Laurent polynomials, 
$\mathcal H^0(\mathrm{Hyp}_{\pi,+})|_{\mathbb V_k^{\mathrm{gen}}}$ 
(resp. $\mathcal H^0(\mathrm{Hyp}_{\pi,!})|_{\mathbb V_k^{\mathrm{gen}}}$) is locally free over
$\mathcal O_{\widehat {\mathbb V},\mathbb Q}|_{\mathbb V_k^{\mathrm{gen}}}$, 
where $\widehat {\mathbb V}$ is the completion of the affine space 
$\mathbb V$. 
For any divisor $T$ of $\mathbb P_k$ containing the complement of $\mathbb V^{\mathrm{gen}}_k$, 
$\mathrm{Hyp}_{\pi,+}|_{\mathbb P_k-T}$ 
(resp. $\mathrm{Hyp}_{\pi,!}|_{\mathbb P_k-T}$) 
is an $F$-isocrystal on $\mathbb P_k-T$ overconvergent along $T$ with rank at most
$d!\int_{\Delta_\infty\cap\mathfrak C} 
\prod_{\alpha \in R^+} 
\frac{\lambda(H_\alpha)^2}{\rho(H_\alpha)^2}\mathrm d\lambda.$
\end{theorem}

We will prove the above theorem in forthcoming sections. At the end of this section, we deduce Theorem 
\ref{thm:expsum} from Theorem \ref{thm:hyp}. 

\subsection{Relation with the Fourier transform} 

The affine space $\mathbb V=\prod_{j=1}^N\mathrm{End}(V_j)$ is self dual via the pairing 
$$\langle\,,\,\rangle:
\mathbb V\times \mathbb V\to \mathbb A^1, \quad ((A_1, \ldots, A_N), (A'_1, \ldots, A'_N))\mapsto \sum_{j=1}^N\mathrm{Tr}(A_jA'_j).$$
We thus have the Fourier transformation 
$$\mathfrak F_\pi: D^b_{\mathrm{ovhol}}(\mathbb V_k)\to D^b_{\mathrm{ovhol}}(\mathbb V_k).$$

\begin{proposition}\label{prop:Fourier}
We have $$\mathrm{Hyp}_{\pi, +}\cong \mathfrak F_\pi(\iota_{k+} \mathcal {O}^\dagger_{G_k}),\quad 
\mathrm{Hyp}_{\pi, !}\cong \mathfrak F_\pi(\iota_{k!} \mathcal {O}^\dagger_{G_k}).$$
Suppose $\iota$ is quasi-finite. Then $\mathrm{Hyp}_{\pi, +}$ and $\mathrm{Hyp}_{\pi, !}$ are
overholonomic arithmetic $\mathcal D$-modules. 
\end{proposition}

\begin{proof}
Fix notation by the following commutative diagram, where all squares
are Cartesian:
$$\begin{tikzcd}
G_k\times_k\mathbb V_k\arrow[r, "\iota_k\times \mathrm
{id}"]\arrow[d,"\pi_1"] &\mathbb V_k\times_k\mathbb  V_k\arrow[r, "p_2"]\arrow[d, "p_1"]& \mathbb V_k\arrow[d]\\
G_k\arrow[r, "\iota_k"]&\mathbb V_k\arrow[r]&\mathrm{Spec}\,k.
\end{tikzcd}$$
We have 
\begin{eqnarray*}
&&\mathfrak F_{\pi}(\iota_{k+}\mathcal {O}^\dagger_{G_k})\\
&\cong& p_{2+} \Big(p_1^! \iota_{k+} \mathcal {O}^\dagger_{G_k}\widetilde\otimes
_{\mathcal O^\dagger_{\mathbb V\times_k\mathbb V}} \langle\,,\,\rangle^! \mathcal L_\pi\Big)[1-n]\\
&\cong& p_{2+} \Big((\iota_k\times\mathrm{id})_+ \pi_1^! \mathcal {O}^\dagger_{G_k}
\widetilde\otimes_{\mathcal O^\dagger_{\mathbb V\times_k \mathbb V}} \langle\,,\,\rangle^! \mathcal L\Big)[1-n]\quad 
\hbox{(the base change theorem \cite[1.3.10]{AC18padicweight})}\\
&\cong& (p_{2} \circ (\iota_k\times \mathrm{id}))_+ \Big(\pi_1^! \mathcal {O}^\dagger_{G_k}\widetilde\otimes
_{\mathcal O^\dagger_{G\times_k\mathbb V}} 
(\langle\,,\,\rangle\circ (\iota_k\times \mathrm{id}))^! \mathcal L\Big)[1-n]\; \hbox{(the projection formula \cite[A.6]{AC18padicweight})}\\
&\cong& \pi_{2+} (\pi_1^!  \mathcal {O}^\dagger_{G_k}
\widetilde\otimes_{\mathcal O^\dagger_{G\times_k\mathbb V}} f^! \mathcal L)[1-n]
\cong \pi_{2+} f^! \mathcal L[1-d-n]
\cong\mathrm{Hyp}_{\pi, +}.
\end{eqnarray*}
Suppose furthermore that $\iota$ is quasi-finite. Note that $\iota$ is affine. By \cite[1.3.13]{AC18padicweight}, 
$\iota_+$ is an exact functor with respect to the $t$-structure defined in \cite[1.2.9]{AC18padicweight}. So 
$\iota_{k+}\mathcal {O}^\dagger_{G_k}$ is an overholonomic arithmetic $\mathcal D$-module. 
The Fourier transform $\mathfrak F_\pi$ is an exact functor. 
So $\mathrm{Hyp}_{\pi, +}\cong \mathfrak F_{\pi}(\iota_{k+}\mathcal {O}^\dagger_{G_k})$ 
is an overholonomic arithmetic $\mathcal D$-module. The assertion for 
$\mathrm{Hyp}_{\pi, !}$ follows by duality.
\end{proof}

\begin{remark} We also need the hypergeometric arithmetic right $\mathcal D$-modules which are obtained from left ones 
by side change. They are defined to be
$$\mathrm{Hyp}_{\pi,+}:=\pi_{2+}f^!(\omega^\dagger_{\mathbb A^1}\otimes_{\mathcal O^\dagger_{\mathbb A^1}}\mathcal L_\pi)[1-d-n], 
\quad \mathrm{Hyp}_{\pi,!} :=\pi_{2!}f^+(\omega^\dagger_{\mathbb A^1}\otimes_{\mathcal O^\dagger_{\mathbb A^1}}\mathcal L_\pi)[d+n-1].$$
We have 
$$\mathrm{Hyp}_{\pi, +}\cong \mathfrak F_\pi(\iota_{k+} \omega^\dagger_{G_k}),\quad 
\mathrm{Hyp}_{\pi, !}\cong \mathfrak F_\pi(\iota_{k!} \omega^\dagger_{G_k}),$$
\end{remark}

\subsection{Relation with the hypergeometric exponential sum}

For any positive integer $m\ge 1$, let 
$$\theta_1(x)=\exp \left(\pi x-\pi x^p\right), \quad \theta_m(x)=\exp \left(\pi x-\pi x^{p^m}\right)=\prod_{i=0}^{m-1} \theta_1\left(x^{{p^i}}\right).$$
They are overconvergent power series, and $\theta_1(1)=\theta_1(x)|_{x=1}$ is a primitive $p$-th root of unity in $K$ 
(\cite[Theorems 4.1-4.3]{Monsky1970PadicAA}). Let  $\psi_0: \mathbb F_p\to \overline{\mathbb Q}_p$ be the additive character
$\psi(a)=\theta_1(1)^a$, and let $\psi: k\to \overline{\mathbb Q}_p$ be the additive character $\psi=\psi_0\circ\mathrm{Tr}_{k/\mathbb F_p}$. 

\begin{proposition}\label{prop:hyp-exp}
Let $k'$ be a finite extension of $k$ of degree $m$, let $A=(A_1, \ldots, A_N)$ be a $k'$-point of $\mathbb V_k=\prod_{j=1}^N
\mathrm{End}(V_{j, k})$, and let $i_A: \mathrm{Spec}\, k'\to \mathbb V_k$ be the morphism corresponding to $A$. We have 
$$\mathrm{Tr}(F^m, i_A^+ \mathrm{Hyp}_{\pi,!}) =(-1)^{d+n}q^{-m}
\sum_{g\in G_k(k')} \psi\Big(\mathrm{Tr}_{k'/k}\Big(\sum_{j=1}^N A_j\rho_j(g)\Big)\Big).$$
\end{proposition}

\begin{proof}  Let 
$p_A: \mathrm{Spec}\, k'\to \mathrm {Spec}\, k$ be the structure morphism. Fix notation by the following diagram:
$$\begin{tikzcd}
G_k\otimes_kk'\arrow[r,"\mathrm{id}_{G_k}\times i_A"]\arrow[d,"p_2"]&
G_k\times_k \mathbb V_k\arrow[r,"f"]\arrow[d,"\pi_2"]& \mathbb A^1_k\\ 
\mathrm{Spec}\, k'\arrow[r," i_A"]\arrow[d,"p_A"]&
 \mathbb V_k&\\
 \mathrm{Spec}\,k.
\end{tikzcd}$$
By the dual version of the base change theorem \cite[1.3.10]{AC18padicweight}, we have
\begin{eqnarray}\label{1}
i_A^+\mathrm{Hyp}_{\pi,!}\cong i_A^+ \pi_{2!}f^+\mathcal L_\pi[d+n-1]
\cong  p_{2!}(\mathrm{id}_{G_k}\times i_A)^+ f^+\mathcal L_\pi[d+n-1].
\end{eqnarray}
Let $|G_k\otimes_kk'|$ be the set of Zariski closed points in $G_k\otimes_kk'$. For any $g\in |G_k\otimes_kk'|$, let
$k(g)$ be the residue field at $g$, $i_g:\mathrm{Spec}\,k(g)\to G_k\otimes_kk'$ the closed immersion, 
$p_g: \mathrm{Spec}\,k(g)\to \mathrm{Spec}\,k'$ the structure morphism, and $\mathrm{deg}(g)=[k(g):k']$.  
By \cite[6.5]{Carosurhol}, we have 
\begin{eqnarray*}
&&\prod_{j}\mathrm{det}_K\Big(1-tF, \mathcal H^j\Big(p_{A!} p_{2!}(\mathrm{id}_{G_k}\times i_A)^+ f^+\mathcal L_\pi\Big)\Big)^{(-1)^{j+1}}\\
&=&\prod_{g\in |G_k\otimes_kk'|}\prod_{j}\mathrm{det}_K
\Big(1-tF, \mathcal H^j\Big(p_{A!}p_{g!}i_g^+ (\mathrm{id}_{G_k}\times i_A)^+ f^+\mathcal L_\pi\Big)\Big)^{(-1)^{j+1}}.
\end{eqnarray*}
The equality is exactly
\begin{eqnarray*}
&&\prod_{j}\mathrm{det}_K\Big(1-t^mF^m, \mathcal H^j
\Big(p_{2!}(\mathrm{id}_{G_k}\times i_A)^+ f^+\mathcal L_\pi\Big)\Big)^{(-1)^{j+1}}\\
&=&\prod_{g\in |G_k\otimes_kk'|}\prod_{j}\mathrm{det}_K
\Big(1-t^{m\mathrm{deg}(g)}F^{m\mathrm{deg}(g)}, 
\mathcal H^j\Big(i_g^+ (\mathrm{id}_{G_k}\times i_A)^+ f^+\mathcal L_\pi\Big)\Big)^{(-1)^{j+1}},
\end{eqnarray*}
Applying the operator $t\frac{d}{dt} \log$ to the above equality and comparing the coefficient of $t^m$, we get
\begin{eqnarray}\label{2}
\mathrm{Tr}(F^m, p_{2!}(\mathrm{id}_{G_k}\times i_A)^+ f^+\mathcal L_\pi)
=\sum_{g\in G(k')} \mathrm{Tr}(F^m, (f\circ (\mathrm{id}_{G_k}\times i_A)\circ i_g)^+ \mathcal L_\pi).
\end{eqnarray}
For any $g\in G(k')$,  $f\circ (\mathrm{id}_{G_k}\times i_A)\circ i_g$ is the $k'$-point 
$$y:=\sum_{j=1}^N \mathrm{Tr}(A_j\rho_j(g))\in k'\cong \mathbb A^1(k').$$
Let $\tilde y \in\bar K$ be the Techm\"uller lifting of $y$ with the property ${\tilde y}^{q^m}=\tilde y$, 
and let $[k:\mathbb F_p]=m_0$ so that $[k':\mathbb F_p]=mm_0$.
By \cite[Theorem 4.4]{Monsky1970PadicAA}), we have 
$$\theta_{mm_0}(x)|_{x=\tilde y }=\theta(1)^{\mathrm{Tr}_{k'/\mathbb F_p}(y)}
=\psi\Big(\mathrm{Tr}_{k'/k}\Big(\sum_{j=1}^N \mathrm{Tr}(A_j\rho_j(g))\Big)\Big).$$  
On the other hand, we have 
$$\theta_{mm_0}(x)=\exp(\pi x- \pi x^{p^{mm_0}})=\exp(\pi x-\pi x^{q^m}).$$ 
By the purity theorem \cite[5.6]{Abe14calculation} and the definition of the Frobenius structure on $\mathcal L_\pi$, we have 
\begin{eqnarray*}
\mathrm{Tr}(F^m, (f\circ (\mathrm{id}_{G_k}\times i_A)\circ i_g)^+ \mathcal L_\pi)
&=&\mathrm{Tr}(F^m, (f\circ (\mathrm{id}_{G_k}\times i_A)\circ i_g)^!\mathcal L_\pi(1)[2] )\\
&=&-q^{-m} \exp(\pi x-\pi x^{q^m})|_{x=\tilde y}.
\end{eqnarray*}
We thus have
\begin{eqnarray}\label{3}
\mathrm{Tr}(F^m, (f\circ (\mathrm{id}_{G_k}\times i_A)\circ i_g)^+ \mathcal L_\pi)
=-q^{-m} \psi\Big(\mathrm{Tr}_{k'/k}\Big(\sum_{j=1}^N \mathrm{Tr}(A_j\rho_j(g))\Big)\Big).
\end{eqnarray}
The proposition follows from the equations (\ref{1})-(\ref{3}).
\end{proof}

\subsection{Proof of Theorem \ref{thm:expsum}}

Let $y$ be a closed point in $\mathbb A_k^1$ of degree $r$, 
let $i_y:\mathrm{Spec}\,k(y)\to \mathbb A^1_k$ be the closed immersion, and let $p_y: \mathrm{Spec}\,k(y)\to\mathrm{Spec}\,k$
be the structure morphism. By the purity theorem \cite[5.6]{Abe14calculation}, 
we have $$i_y^+\mathcal L_\pi=i_y^!\mathcal L_\pi(1)[2]=i_y^*\mathcal L_\pi(1)[1].$$ 
Let $\tilde y$ be the Techm\"uller lifting of $y$. 
The Frobenius acts on $p_{y+}i_y^*\mathcal L_\pi$ via multiplication by 
$$\exp(\pi x-\pi x^{q^r})|_{x=\tilde y}$$ which is a $p$-th root of unity and has weight $0$. 
By the definition in \cite[2.1.3]{AC18padicweight},
$\mathcal L_\pi$ is pure of weight $-1$. 
By the main theorem in \cite{AC18padicweight}, 
$\mathrm{Hyp}_{\pi,!}=\pi_{2!}f^+\mathcal L_\pi[d+n-1]$ is mixed of weight $\leq d+n-2$.  
Admitting Theorem \ref{thm:hyp}, $\mathrm{Hyp}_{\pi,!}$ is an overconvergent
$F$-isocrystal on $\mathbb V_k^{\mathrm{gen}}$ of rank $\leq d!\int_{\Delta_\infty\cap\mathfrak C} 
\prod_{\alpha \in R^+} 
\frac{\lambda(H_\alpha)^2}{\rho(H_\alpha)^2}\mathrm d\lambda.$
By the purity theorem \cite[5.6]{Abe14calculation}, we have 
$$i_A^+\mathrm{Hyp}_{\pi,!}=i_A^!\mathrm{Hyp}_{\pi,!}(n)[2n]=i_A^*\mathrm{Hyp}_{\pi,!}(n)[n],$$
and hence $\mathcal H^j(i_A^+\mathrm{Hyp}_{\pi,!})=0$ for $j\not=-n$. 
Let $m=[k':k]$ so that $q'=q^m$. We have
\begin{eqnarray*}
|\mathrm{Tr}(F^m, i_A^+ \mathrm{Hyp}_{\pi,!})|&=&|(-1)^{-n} \mathrm{Tr}(F^m, \mathcal H^{-n}(i_A^+ \mathrm{Hyp}_{\pi,!}))|\\
&\leq& q^{m\frac{((d+n-2)-n)}{2}} \mathrm{rank} (\mathrm{Hyp}_{\pi,!})\\
&\leq& q^{m\frac{d-2}{2}}  d! \int_{\Delta_\infty\cap\mathfrak C} 
\prod_{\alpha \in R^+} 
\frac{\lambda(H_\alpha)^2}{\rho(H_\alpha)^2}\mathrm d\lambda.
\end{eqnarray*}
Theorem \ref{thm:expsum} follows from this inequality and Proposition \ref{prop:hyp-exp}. \qed

\section{Calculation on arithmetic $\mathcal D$-modules}

We assume the morphism 
$$\iota: G\to \mathbb V=\prod_{j=1}^N \mathrm{End}(V_j),\quad g\mapsto (\rho_1(g), \ldots, \rho_N(g))$$ 
is quasi-finite. 
Let $\mathbb P:=\mathbb P(\mathbb A^1\times \mathbb V)$ be the projective space. We regard $\mathbb V$ as an open subscheme of 
$\mathbb P$ via the open immersion
$$\mathbb V\hookrightarrow \mathbb P, \quad v\mapsto [1:v].$$
Let $X$ (resp. $\overline X$) be the closure of $\iota(G)$ in $\mathbb V$ (resp. $\mathbb P$) with the reduced closed subscheme 
structure. We have 
$X=\overline X\cap \mathbb V$. Let $Y$ (resp. $\overline Y$) be the integral closure of $X$ (resp. $\overline X$) in 
the function field of $G$. The composite 
$$Y\to X\to \mathbb V\quad (\hbox{resp. }\overline Y\to \overline X\to \mathbb P)$$
is a finite morphism by the Zariski Main theorem (\cite[8.12.6]{EGAIV}). Let $H=G\times G$. We have an action 
$$H\times \mathbb V\to \mathbb V, \quad ((g_1, g_2), (A_1, \ldots, A_N))\mapsto (\rho_1(g_1)A_1\rho_1(g_2^{-1}), 
\ldots, \rho_N(g_1)A_N\rho_N(g_2^{-1})).$$
It can be extended to an action of $H$ on $\mathbb P=\mathbb P(\mathbb A^1\times \mathbb V)$ by letting $H$ act trivially 
on the homogeneous coordinate corresponding to the component $\mathbb A^1$. 
It induces actions of $H$ on $X$, $Y$, $\overline X$, and $\overline Y$. Moreover, $H$ acts on $G$
via the action 
$$(G\times G) \times G\to G, \quad ((g_1, g_2), x)\to g_1 xg_2^{-1}.$$ 
The morphisms $$G\hookrightarrow Y\to X\to \mathbb V,\quad 
G\hookrightarrow\overline Y\to \overline X\to \mathbb P$$ are equivariant.
Following \cite[6.2.1]{Frob-split}, we call a geometrically integral and geometrically normal algebraic $K$-variety with an $H_K$-action 
containing $G_K$ as an open dense orbit an \emph{equivariant embedding of $G_K$}. 
Then $Y_K$ and $\overline Y_K$ are equivariant embedding of $G_K$.

\begin{proposition}\label{ass1} Let $B^+$ be a Borel subsgroup of $G$, $B^-$ the opposite Borel subgroup, 
$U^+$ (resp. $U^-$) the unipotent
radical of $B^+$ (resp. $B^-$), and $T=B^+\cap B^-$ the maximal torus. 
There exists a morphism $\sigma:\tilde Y\to \overline Y$ with the following properties:
\begin{enumerate}[(1)]
\item $\tilde Y$ is a smooth proper $R$-scheme with a $(G\times G)$-action containing $G$ an open subscheme, 
and $\tilde Y\to \mathrm{Spec}\,R$ has geometrically connected fibers. 
\item $\sigma:\tilde Y\to \overline Y$ is proper equivariant, and 
induces identity on $G$, where we regard $G$ as an open subscheme of both
$\tilde Y$ and $\overline Y$. 
\item There exists a $(B^+\times B^-)$-invariant open subscheme $\tilde{Y}^\circ$ of $\tilde{Y}$ and a $(T\times T)$-invariant 
closed subscheme $S$ of $\tilde Y^\circ$ 
such that $S$ is a toric scheme containing $T$ as an open dense torus,  
$$(G\times G)\tilde Y^\circ=\tilde Y, \quad (U^+\times U^-)T= G\cap \tilde Y^\circ,$$ and we have an isomorphism
\begin{eqnarray}\label{eqn:equiv.map}
(U^+\times U^-)\times S\stackrel\cong\to \tilde Y^\circ, \quad (g, x)\mapsto gx.
\end{eqnarray}
\end{enumerate}
\end{proposition}

\begin{proof} By \cite[6.2.5]{Frob-split}, we may choose an equivariant proper morphism 
$\sigma_K:\tilde Y_K\to \overline Y_K$ such that
$\tilde Y_K$ is a smooth toroidal equivariant embedding of $G_K$ in the sense of \cite[6.2.2]{Frob-split}.
Let $D_i=\overline{B^+_Ks_iB^-_K}$ be the closure of $B^+_Ks_iB_K^-$ in $\tilde Y_K$, where $s_i$ $(i=1,\ldots, r)$ 
are those elements in the Weyl group $W:=N_G(T)/T$ corresponding to the reflections determined 
by the simple roots. They are the $(B^+_K\times_K B^-_K)$-stable but not $(G_K\times_K G_K)$-stable prime divisors 
of $\tilde Y_K$. Let $\tilde Y^\circ_{K}=\tilde Y_K-D_1\cup\cdots\cup D_r.$ By \cite[6.2.3]{Frob-split}, 
there exists a closed subvariety $S_K$ of $\tilde Y^\circ_{K}$
stable under the action of $T_K\times_K T_K$ such that $S_K$ is a toric variety containing $T_K$ as an open dense torus,
$$(G_K\times_K G_K) \tilde Y^\circ_{K}=\tilde Y_K, \quad U^+_KT_KU^-_K=G_K\cap \tilde Y^\circ_{K},$$ and we have an
isomorphism $$(U^+_K\times_K U^-_K) \times_K S_K \stackrel\cong \to \tilde Y^\circ_{K}, \quad (g,y)\mapsto gy.$$
By the work of Li \cite{Lishang}, the above data can be defined over $R$. 
More precisely,  let $S$ be the toric $R$-scheme with the same fan $\Sigma$ as $S_K$. Take 
$\tilde Y^\circ:=(U^+\times U^-)\times S$.  By the Bruhat decomposition, $U^+TU^-$ is an open subscheme of $G$,
and is isomorphic to an open subset of $\tilde Y^\circ$. The morphism 
$$(G\times G)\times G\to G,\quad ((g_1, g_2), g)\mapsto g_1gg_1^{-1}$$ 
defines a rational action 
of $G\times G$ on $\tilde Y^\circ$: 
$$(G\times G)\times \tilde Y^\circ \dashrightarrow \tilde Y^\circ.$$
Let $\tilde Y$ be the fppf quotient sheaf of $(G\times G)\times \tilde Y^\circ$ with respect to the equivalence 
relation defined in \cite[3.1]{Lishang}.  Li  (\cite[5.4, 5.5]{Lishang}) shows that $\tilde Y$ is an $R$-scheme 
equipped with a $(G\times G)$-action and containing $G$ as an open dense subscheme such that 
the generic fiber of $\tilde Y\to \mathrm{Spec}\,R$ is $\tilde Y_K$, and the special fiber $\tilde Y_k$
is the toroidal embedding of $G_k$ associated to the fan $\Sigma$ (\cite[6.2.4]{Frob-split}). 
The conditions (1) and (3) hold by the construction. 
Since $\tilde Y_K$ is smooth, $\Sigma$ is a regular fan. So $S$ is smooth over $R$. This 
implies that $\tilde Y^\circ$ is smooth. Since $\tilde Y$ is a translate of $\tilde Y^\circ$ by $G\times G$, $\tilde Y$ is also smooth.
Since $\overline Y$ is normal, the morphism $\sigma_K: \tilde Y_K\to \overline Y_K$ can be uniquely extended to an $R$-morphism 
$\sigma: \tilde Y\to \overline Y$, which is necessarily $(G\times G)$-equivariant. Since $\tilde Y_K$ is proper, the fan 
$W\Sigma$ is a complete fan, where $W$ is the Weyl group. By \cite[5.6]{Lishang}, $\tilde Y$ is proper over $R$. 
So $\sigma: \tilde Y\to \overline Y$ is proper. 
\end{proof} 

\begin{proposition}\label{prop:inv0inf} Let $\bar\iota$ be the composite $$\bar \iota: \tilde Y\to \overline Y\to \overline X\to \mathbb P.$$ 
Then $\bar\iota^{-1}(\prod_{j=1}^N\mathrm{GL}(V_j))=G$.
\end{proposition}

\begin{proof} Let $U=\bar\iota^{-1}(\prod_{j=1}^N\mathrm{GL}(V_j))$. Then $U$ is an open subscheme 
of $\tilde Y$, $G\subset U$, and $U$ is $H$-invariant. For any prime ideal $\mathfrak p$ of $R$, let $\bar k(\mathfrak p)$ be an algebraic
closure of the residue field $k(\mathfrak p)$. By the condition Proposition \ref{ass1} (1), $\tilde Y_{\bar k(\mathfrak p)}$ 
is irreducible. So $G_{\bar k(\mathfrak p)}$ is dense in $\tilde Y_{\bar k(\mathfrak p)}$
and hence 
\begin{eqnarray}\label{eq:1}
\dim (U_{\bar k(\mathfrak p)}-G_{\bar k(\mathfrak p)})< \mathrm{dim}\, G_{\bar k(\mathfrak p)}.
\end{eqnarray}
Thus $\bar\iota(U_{\bar k(\mathfrak p)}-G_{\bar k(\mathfrak p)})$ is an $H_{\bar k(\mathfrak p)}$-invariant 
subset of $\prod_{j=1}^N\mathrm{GL}(V_{j,\bar k(\mathfrak p)})$ of dimension $<\mathrm{dim} G_{\bar k(\mathfrak p)}$. 
Suppose this set is not empty, say contains a $\bar k(\mathfrak p)$-point $x$. We have 
$$G_{\bar k(\mathfrak p)} x\subset G_{\bar k(\mathfrak p)} xG_{\bar k(\mathfrak p)}= H_{\bar k(\mathfrak p)}x\subset 
\bar\iota(U_{\bar k(\mathfrak p)}-G_{\bar k(\mathfrak p)}).$$
So we have
\begin{eqnarray}\label{eq:2}
\mathrm{dim}(G_{\bar k(\mathfrak p)} x)\leq \mathrm{dim}\, \bar \iota(U_{\bar k(\mathfrak p)}-G_{\bar k(\mathfrak p)})
\leq \mathrm{dim} (U_{\bar k(\mathfrak p)}-G_{\bar k(\mathfrak p)}).
\end{eqnarray}
Since the components of $x\in \prod_{j=1}^N\mathrm{GL}(V_{j,\bar k(\mathfrak p)})$ are invertible matrices, we have 
\begin{eqnarray}\label{eq:3}
\mathrm{dim}(G_{\bar k(\mathfrak p)} x)= \mathrm{dim}\, \bar \iota(G_{\bar k(\mathfrak p)})=
\mathrm{dim}\,G_{\bar k(\mathfrak p)}.
\end{eqnarray}
The three conditions (\ref{eq:1})-(\ref{eq:3}) lead to a contradiction. So $U_{\bar k(\mathfrak p)}-G_{\bar k(\mathfrak p)}$
is empty for every $\mathfrak p\in\mathrm{Spec}\,R$. Hence $U\subset G$. 
\end{proof}

Let 
$\widehat{{\tilde Y}^\circ}$ and $\widehat {\tilde Y}$ be the completion of the $\tilde Y^\circ$ and $\tilde Y$, respectively.
Fix notation by the following commutative diagram
$$\begin{tikzcd}
(U_k^+\times_k U_k^-)\times_k T_k\arrow[r, hook]\arrow[d, hook]&  (U_k^+\times_k U_k^-) \times_k S_k \cong \tilde Y^\circ_k\arrow[r]
\arrow[d,hook]& \widehat{{\tilde Y}^\circ}\arrow[d,hook]\\
G_k\arrow[d,"\iota_k"]\arrow[r,hook]& \tilde Y_k\arrow[d,"\bar\iota_k"]\arrow[r]
&\widehat {\tilde Y}\arrow[d,"\widehat{\bar\iota}"]\\
\mathbb V_k=\prod_{j=1}^N \mathrm{End}(V_{j, k})\arrow[r, hook]&\mathbb P_k\arrow[r]&\widehat{\mathbb P}.
\end{tikzcd}$$
Let $\f{L}(H_K)$ be the Lie algebra of $H_K=G_K\times_K G_K$. For any smooth $K$-variety $V$ with a left $H_K$-action 
and any $\xi\in \f{L}(H_K)$, let $L^V_{\xi}$be the vector field on $V$ defined by
$$L^V_{\xi}(x)=\frac{d}{dt}\Big|_{t=0}(\exp(t\xi)x)$$ 
for any point $x$ in $V$. We omit the superscript $V$ from $L^V_\xi$ if this causes 
no confusion. We regard them as differential operators.
See section 3 for the precise definition of $L_\xi$ and more general invariant differential operators. 

\begin{lemma}\label{toroidalDmodiso} Let $D$ be the divisor $\tilde Y-G$ of $\tilde Y$. 
We have an isomorphism of $\mathcal {D}^\d_{\widehat{\tilde Y},\Q}$-modules
$$\mathbb D\circ (^\dagger D_k)\circ \mathbb D (\mathcal O_{\widehat{\tilde Y},\Q})\cong \mathcal {D}^\d_{\widehat{\tilde Y},\mathbb Q}
/\sum_{\xi\in \f{L}(H_K)}\mathcal {D}^\d_{\widehat{\tilde Y},\Q}L_{\xi},$$
where $\b{D}$ is the Verdier dual on category $D_{\mathrm{coh}}^b(\mathcal {D}^\d_{\widehat{\tilde Y},\b{Q}})$.
\end{lemma}

\begin{proof}
The morphism 
\begin{eqnarray*}
\mathcal {D}^\d_{\widehat{\tilde Y},\mathbb Q}(^\dagger D_k)
\Big /\sum_{\xi\in \f{L}(H_K)}\mathcal {D}^\d_{\widehat{\tilde Y},\mathbb Q}(^\dagger D_k)
L_{\xi}\to \mathcal {O}_{\widehat{\tilde Y},\mathbb Q}(^\dagger D_k), \quad P\mapsto P\cdot 1
\end{eqnarray*}
is an isomorphism when restricted to $\widehat G$ by \cite[3.2.2]{Berthelot1990}, where $\widehat G$ is the completion of $G$ 
regarded as an open formal subscheme of $\widehat{\tilde Y}$. By \cite[4.3.12 (ii)]{Berthelot1}, it must be an isomorphism. 
We thus have an isomorphism
\begin{eqnarray}\label{eqn:0iso}
\mathcal {O}_{\widehat{\tilde Y},\mathbb Q}(^\dagger D_k)\cong \mathcal {D}^\d_{\widehat{\tilde Y},\mathbb Q}(^\dagger D_k)
\Big /\sum_{\xi\in \f{L}(H_K)}\mathcal {D}^\d_{\widehat{\tilde Y},\mathbb Q}(^\dagger D_k)
L_{\xi}.
\end{eqnarray}
Let $\mathbb D_{D_k}$ be the Verdier dual on the category 
$D^b_{\mathrm{coh}}(\mathcal {D}^\d_{\widehat{\tilde Y},\b{Q}}(^\dagger D_k))$. We have an isomorphism 
\begin{eqnarray}\label{eqn:pregamma}
\mathbb D_{D_k}\Big(\mathcal {D}^\d_{\widehat{\tilde Y},\mathbb Q}(^\dagger D_k)
\Big /\sum_{\xi\in \f{L}(H_K)}\mathcal {D}^\d_{\widehat{\tilde Y},\mathbb Q}(^\dagger D_k)
L_{\xi}\Big)\cong \mathbb D_{D_k}(\mathcal {O}_{\widehat{\tilde Y},\mathbb Q}(^\dagger D_k)).
\end{eqnarray}
By \cite[I.4.4]{V} and \cite[4.3.11]{Berthelot1}, this isomorphism can be identified with
$$(^\dagger D_k)\circ \mathbb D 
\Big(\mathcal {D}^\d_{\widehat{\tilde Y},\mathbb Q}\Big /\sum_{\xi\in \f{L}(H_K)}\mathcal {D}^\d_{\widehat{\tilde Y},\mathbb Q}
L_{\xi}\Big)\cong(^\dagger D_k)\circ \mathbb D (\mathcal {O}_{\widehat{\tilde Y},\mathbb Q}).$$ 
Composed with the canonical morphism 
$$\mathbb D\Big(\mathcal {D}^\d_{\widehat{\tilde Y},\mathbb Q}\Big /\sum_{\xi\in \f{L}(H_K)}\mathcal {D}^\d_{\widehat{\tilde Y},\mathbb Q}
L_{\xi}\Big)\to(^\dagger D_k)\circ \mathbb D 
\Big(\mathcal {D}^\d_{\widehat{\tilde Y},\mathbb Q}\Big /\sum_{\xi\in \f{L}(H_K)}\mathcal {D}^\d_{\widehat{\tilde Y},\mathbb Q}
L_{\xi}\Big),$$
we get a morphism
\begin{eqnarray}\label{eqn:addlater}
\mathbb D 
\Big(\mathcal {D}^\d_{\widehat{\tilde Y},\mathbb Q}\Big /\sum_{\xi\in \f{L}(H_K)}\mathcal {D}^\d_{\widehat{\tilde Y},\mathbb Q}
L_{\xi}\Big)\to (^\dagger D_k)\circ\mathbb D(\mathcal {O}_{\widehat{\tilde Y},\mathbb Q}).
\end{eqnarray}
Applying $\D$ to this morphism, we get a morphism 
\begin{eqnarray*}
\gamma: \mathbb D\circ (^\dagger D_k) \circ \mathbb D(\mathcal O_{\widehat{\tilde Y},\Q})\to \mathcal {D}^\d_{\widehat{\tilde Y},\mathbb Q}
/\sum_{\xi\in \f{L}(H_K)}\mathcal {D}^\d_{\widehat{\tilde Y},\Q}L_{\xi}.
\end{eqnarray*}
Let's prove the last morphism $\gamma$ is an isomorphism.

This morphism is $H_{k}$-equivariant in the following sense: For any $\bar h\in H(\bar k)$, we can lift $\bar h$ to an $R_{\mathrm{ur}}$-point
$h$ of $H$, where $R_{\mathrm{ur}}$ is the maximal unramified extension of $R$. The pullback of the above morphism by the $h$-action
$h: \widehat {\tilde{Y}}_{R_{\mathrm ur}}\to\widehat{\tilde{Y}}_{R_{\mathrm ur}}$ can be identified with itself. By Proposition \ref{ass1}, 
we have $H\tilde{Y}^\circ=\tilde{Y}$. So it suffices to show $\gamma|_{\widehat{{\tilde Y}^\circ}}$ is an isomorphism.
We have 
\begin{eqnarray*}
\mathbb D\circ (^\dagger D_k)\circ \mathbb D(\mathcal {O}_{\widehat{\tilde Y},\mathbb Q})|_{\widehat{{\tilde Y}^\circ}}&\cong& 
\mathcal O_{(U^+\times U^-)^\wedge,\mathbb Q}
\boxtimes\b{D}\circ (^\dagger D_k\cap S_k)\circ\mathbb D (\mathcal {O}_{\widehat S,\mathbb{Q}}),\\
\Big(\mathcal {D}^\d_{\widehat{\tilde Y},\mathbb Q}
\Big/\sum_{\xi\in \mathfrak {L}(H_K)}\mathcal {D}^\d_{\widehat{\tilde Y},\Q}L_{\xi}\Big)\Big|_{\widehat{\tilde Y^\circ}}
&\cong&\Big( \mathcal D^\dagger_{(U^+\times U^-)^\wedge,\mathbb Q}\Big/\sum_{\xi\in \mathfrak L(U_K^+\times_K U_K^-)}
\mathcal D^\dagger_{(U^+\times U^-)^\wedge,\mathbb Q}L_\xi\Big)\\
&&\qquad\qquad\boxtimes\Big( \mathcal D^\dagger_{\widehat S,\mathbb Q}
\Big/\sum_{\xi\in \mathfrak L(T_K\times_K T_K)}
\mathcal D^\dagger_{\widehat S,\mathbb Q}L_\xi\Big),
\end{eqnarray*}
where $(U^+\times U^-)^\wedge$ and $\widehat S$ are the completions of $U^+\times U^-$ and $S$, respectively. 
By \cite[3.2.2]{Berthelot1990}, we have 
$$
\mathcal O_{(U^+\times U^-)^\wedge,\mathbb Q}\cong 
\mathcal D^\dagger_{(U^+\times U^-)^\wedge,\mathbb Q}\Big/\sum_{\xi\in \mathfrak L(U_K^+\times_K U_K^-)}
\mathcal D^\dagger_{(U^+\times U^-)^\wedge,\mathbb Q}L_\xi.$$
Since $S$ is a smooth toric variety, locally we may assume the pair $(S,D_k\cap S_k)$ is isomorphic to 
$(\mathbb{G}_m^s\times\mathbb A^t, \bigcup_{i=s+1}^{s+t} (x_i=0))$, 
where we take the canonical coordinate system $(x_1, \ldots,x_{s+t})$ 
on $\mathbb G_m\times \mathbb A^t$. Again by \cite[3.2.2]{Berthelot1990}, we have 
$$
\mathcal O_{\widehat{\mathbb G}^s_m,\mathbb Q}\cong 
\mathcal D^\dagger_{\widehat{\mathbb G}^s_m,\mathbb Q}/\sum_{\xi\in \mathfrak L(\mathbb G^s_{m,K})}
\mathcal D^\dagger_{\widehat{\mathbb G}^s_m,\mathbb Q}L_\xi.$$
By \cite[4.3.2]{Berthelot1990}, we have an isomorphism
\begin{eqnarray}\label{eqn:aftergamma}
\mathcal D^\dagger_{\widehat{\mathbb A}^1,\mathbb Q}/
\mathcal D^\dagger_{\widehat{\mathbb A}^1,\mathbb Q}\partial_{x}x 
\stackrel\cong \to \mathcal O_{\widehat{\mathbb A}^1,\mathbb Q}(^\dagger 0),\quad P\mapsto P\cdot (1/x).
\end{eqnarray}
It follows that we have an isomorphism
\begin{eqnarray}\label{eqn:reaftergamma}
\mathcal D^\dagger_{\widehat{\mathbb A}^1,\mathbb Q}(^\dagger 0)/
\mathcal D^\dagger_{\widehat{\mathbb A}^1,\mathbb Q}(^\dagger 0)\partial_{x}x 
\stackrel\cong \to \mathcal O_{\widehat{\mathbb A}^1,\mathbb Q}(^\dagger 0),\quad P\mapsto P\cdot (1/x).
\end{eqnarray}
The isomorphism (\ref{eqn:0iso}) shows that 
$$0\to \mathcal D^\dagger_{\widehat{\mathbb A}^1,\mathbb Q}(^\dagger 0)\stackrel{\cdot\partial_x}\to
\mathcal D^\dagger_{\widehat{\mathbb A}^1,\mathbb Q}(^\dagger 0)\to 0$$ is a free resolution of 
$\mathcal O_{\widehat{\mathbb A}^1,\mathbb Q}(^\dagger 0)$. Using this resolution to calculate 
$\mathbb D_{0}(\mathcal O_{\widehat{\mathbb A}^1,\mathbb Q}(^\dagger 0))$, we get 
\begin{eqnarray}\label{eq:repregamma}
\mathrm{coker}(\mathcal D^\dagger_{\widehat{\mathbb A}^1,\mathbb Q}(^\dagger 0)\stackrel{\cdot\partial_x}\to
\mathcal D^\dagger_{\widehat{\mathbb A}^1,\mathbb Q}(^\dagger 0))\cong 
\mathbb D_{0}(\mathcal O_{\widehat{\mathbb A}^1,\mathbb Q}(^\dagger 0)).
\end{eqnarray} 
This is the isomorphism (\ref{eqn:pregamma}).
We have a commutative diagram
$$\begin{tikzcd}
0\arrow[r]&\mathcal D^\dagger_{\widehat{\mathbb A}^1,\mathbb Q}(^\dagger 0)
\arrow[r, "\cdot\partial_x"]\arrow[d,"\mathrm{id}"]&\mathcal D^\dagger_{\widehat{\mathbb A}^1,\mathbb Q}(^\dagger 0)
\arrow[d,"\cdot x"] \arrow[r, "P\mapsto P\cdot 1"]& \mathcal  O_{\widehat{\mathbb A}^1,\mathbb Q}(^\dagger 0)\arrow[d
,"\mathrm{id}"]\arrow [r]&0\\
0\arrow[r]&\mathcal D^\dagger_{\widehat{\mathbb A}^1,\mathbb Q}(^\dagger 0)
\arrow[r, "\cdot\partial_x x"]&\mathcal D^\dagger_{\widehat{\mathbb A}^1,\mathbb Q}(^\dagger 0)
\arrow[r, "P\mapsto P\cdot (1/x)"]& \mathcal  O_{\widehat{\mathbb A}^1,\mathbb Q}(^\dagger 0)\arrow [r]&0.
\end{tikzcd}
$$
So the isomorphism (\ref{eq:repregamma}) can be identified with the isomorphism (\ref{eqn:reaftergamma}). 
The composite of the isomorphism (\ref{eqn:reaftergamma}) with the canonical morphism 
$$\mathcal D^\dagger_{\widehat{\mathbb A}^1,\mathbb Q}/
\mathcal D^\dagger_{\widehat{\mathbb A}^1,\mathbb Q}\partial_{x}x \to
\mathcal D^\dagger_{\widehat{\mathbb A}^1,\mathbb Q}(^\dagger 0)/
\mathcal D^\dagger_{\widehat{\mathbb A}^1,\mathbb Q}(^\dagger 0)\partial_{x}x $$
is exactly the isomorphism (\ref{eqn:aftergamma}). This shows that the morphism 
(\ref{eqn:addlater}) is an isomorphism. So $\gamma$ is an isomorphism. 
\end{proof}

We have the following two right  $\mathcal {D}^\dagger_{\widehat {\tilde Y},\mathbb Q}$-module structures 
on $\omega_{\widehat {\tilde Y},\Q} \otimes_{\mathcal{O}_{\widehat {\tilde Y},\Q}} \mathcal {D}^\dagger_{\widehat {\tilde Y},\mathbb Q}$. 
Let $\omega\otimes P$ be a section of 
$\omega_{\widehat {\tilde Y},\Q} \otimes_{\mathcal{O}_{\widehat {\tilde Y},\Q}} \mathcal {D}^\dagger_{\widehat {\tilde Y},\mathbb Q}$, 
and let $\theta$ be a vector field viewed as a section of $\mathcal {D}^\dagger_{\widehat {\tilde Y},\mathbb Q}$.

(1) The naive right action is given by
$$(\omega \otimes P) \cdot \theta=\omega \otimes P\theta.$$

(2) The right action is given by the Leibniz rule 
$$(\omega \otimes P) \cdot \theta=\omega \theta\otimes P-\omega \otimes \theta P.$$ 
This structure is obtained by side change from the left $\mathcal {D}^\dagger_{\widehat {\tilde Y},\mathbb Q}$-module structure on 
$\mathcal {D}^\dagger_{\widehat {\tilde Y},\mathbb Q}$.

By \cite[1.3.3]{Berthelot2}, there exists an isomorphism 
on $\omega_{\widehat {\tilde Y},\Q} \otimes_{\mathcal{O}_{\widehat {\tilde Y}},\Q} \mathcal {D}^\dagger_{\widehat {\tilde Y},\mathbb Q}$
which exchanges the two right $\mathcal {D}^\dagger_{\widehat {\tilde Y},\mathbb Q}$-module structures. We have 
the following right $\mathcal D$-module version of Lemma \ref{toroidalDmodiso}.

\begin{corollary}\label{toroidalDmodiso'} ${}$
\begin{enumerate}[(i)]
\item We have an isomorphism of right $\mathcal D^\dagger_{\widehat {\tilde Y},\mathbb Q}$-modules
$$\D(^\d {D}_k)\omega_{\widehat {\tilde Y},\Q}\cong (\omega_{\widehat {\tilde Y},\Q}\otimes_{\c{O}_{\widehat {\tilde Y},\Q}}
\mathcal {D}^\d_{\widehat {\tilde Y},\Q})\Big/\sum_{\xi\in \f{L}(H_K)} L'_{\xi}(\omega_{\widehat {\tilde Y},\Q}
\otimes_{\c{O}_{\widehat {\tilde Y},\Q}}\mathcal {D}^\d_{\widehat {\tilde Y},\Q})$$
where the right $\mathcal {D}^\d_{\widehat {\tilde Y},\Q}$-module structure is induced by the right 
$\mathcal {D}^\d_{\widehat {\tilde Y},\Q}$-module structure on 
$\omega_{\widehat {\tilde Y},\Q}\otimes_{\c{O}_{\widehat {\tilde Y},\Q}}\mathcal {D}^\d_{\widehat {\tilde Y},\Q}$ 
given by the Leibniz rule, and
$L'_{\xi}$ acts on $\omega_{\widehat {\tilde Y},\Q}
\otimes_{\c{O}_{\widehat {\tilde Y},\Q}}\mathcal {D}^\d_{\widehat {\tilde Y},\Q}$ by the naive right multiplication of $L_\xi$
on $\omega_{\widehat {\tilde Y},\Q}\otimes_{\c{O}_{\widehat {\tilde Y},\Q}}\mathcal {D}^\d_{\widehat {\tilde Y},\Q}$.
\item We have an isomorphism of right $\mathcal D^\dagger_{\widehat {\tilde Y},\mathbb Q}$-modules
$$\D(^\d {D}_k)\omega_{\widehat {\tilde Y},\Q}\cong (\omega_{\widehat {\tilde Y},\Q}\otimes_{\c{O}_{\widehat {\tilde Y},\Q}}
\mathcal {D}^\d_{\widehat {\tilde Y},\Q})\Big/\sum_{\xi\in \f{L}(H_K)} L_{\xi}(\omega_{\widehat {\tilde Y},\Q}
\otimes_{\c{O}_{\widehat {\tilde Y},\Q}}\mathcal {D}^\d_{\widehat {\tilde Y},\Q})$$
where the right $\mathcal {D}^\d_{\widehat {\tilde Y},\Q}$-module structure is induced by the naive right 
$\mathcal {D}^\d_{\widehat {\tilde Y},\Q}$-module 
structure on $\omega_{\widehat {\tilde Y},\Q}\otimes_{\c{O}_{\widehat {\tilde Y},\Q}}\mathcal {D}^\d_{\widehat {\tilde Y},\Q}$, 
and $L_{\xi}$ acts on $\omega_{\widehat {\tilde Y},\Q}\otimes_{\c{O}_{\widehat {\tilde Y},\Q}}\mathcal {D}^\d_{\widehat {\tilde Y},\Q}$ by
the Leibniz rule
$$L_{\xi}(\omega\otimes P)=\omega L_\xi\otimes P-\omega \otimes L_{\xi}P$$ 
for any sections $\omega$ of $\omega_{\widehat {\tilde Y},\Q}$ and $P$ of $\mathcal {D}^\d_{\widehat {\tilde Y},\Q}$.
\end{enumerate}
\end{corollary}

\begin{proposition}\label{prop:directfactor} In the category of right 
$\mathcal {D}^\d_{\b{\widehat P}, \Q}$-modules, 
we have $$\mathbb D \bar \iota_{k+}(^\d {D}_k)\omega_{\widehat {\tilde Y},\Q}\cong
\mathcal H^0(\mathbb D \bar \iota_{k+}(^\d {D}_k)\omega_{\widehat {\tilde Y},\Q}),$$ and  
$\mathbb D \bar \iota_{k+}(^\d {D}_k)\omega_{\widehat {\tilde Y},\Q}$ is a direct summand of 
$$\mathcal N':=((\bar\iota_*\omega_{\tilde Y})^\wedge_{\mathbb Q}
\otimes_{\c{O}_{\widehat{\b{P}}, \Q}}\mathcal {D}^\d_{\b{\widehat P}, \Q})\Big/\sum_{\xi\in \f{L}(H_K)}
L_{\xi}(\bar\iota_*\omega_{\tilde Y})^\wedge_{\mathbb Q}\otimes_{\c{O}_{\widehat{\b{P}}, \Q}}\mathcal {D}^\d_{\b{\widehat P}, \Q}),$$
where the right $\mathcal {D}^\d_{\b{\widehat P}, \Q}$-module structure on $\mathcal N'$ is induced by the naive right 
$\mathcal {D}^\d_{\b{\widehat P}, \Q}$-module structure on $(\bar\iota_*\omega_{\tilde Y})^\wedge_{\mathbb Q}
\otimes_{\c{O}_{\widehat{\b{P}}, \Q}}\mathcal {D}^\d_{\b{\widehat P}, \Q}$, and 
the action of $L_{\xi}$ on $(\bar\iota_*\omega_{\tilde Y})^\wedge_{\mathbb Q}
\otimes_{\c{O}_{\widehat{\b{P}}, \Q}}\mathcal {D}^\d_{\b{\widehat P}, \Q}$ is given by 
$$L_{\xi}(\omega\otimes P)=\omega L_{\xi}\otimes P-\omega \otimes L_{\xi} P$$ for any
sections $\omega$ and $P$ of $(\bar\iota_*\omega_{\tilde Y})^\wedge_{\mathbb Q}$ and $\mathcal {D}^\d_{\b{\widehat P}, \Q}$, respectively.
\end{proposition}

\begin{proof}
Let $\bar\iota_i: \tilde Y_i\to \mathbb P_i$ be obtained from $\bar \iota: \tilde Y\to \mathbb P$ by
the base extension $R\to R/\mathfrak m^{i+1}$. Denote by $K\mapsto K_{\mathbb Q}$ the localization functor 
of the derived category of $\mathcal {D}^{(m)}$-modules with respect to the isogenies. 
Then we have
\begin{eqnarray*}
&&\bar\iota_{k+}(\omega_{\widehat{\tilde Y},\Q}\otimes_{\c{O}_{\widehat{\tilde Y},\Q}}\mathcal {D}^\d_{\widehat{\tilde Y},\Q})
\cong \varinjlim_m \Big(R\varprojlim_i R\bar\iota_{i*}(\omega_{\tilde Y_i}\otimes_{\mathcal O_{\tilde Y_i}}
\mathcal {D}^{(m)}_{\tilde Y_i}
\otimes^L_{\mathcal {D}^{(m)}_{\tilde Y_i}} \mathcal D^{(m)}_{\tilde Y_i\to \mathbb P_i})\Big)_{\mathbb Q}\\
&\cong& \varinjlim_m \Big(R\varprojlim_i R\bar\iota_{i*}(\omega_{\tilde Y_i}\otimes_{\mathcal O_{\tilde Y_i}}
\mathcal {D}^{(m)}_{\tilde Y_i}
\otimes^L_{\mathcal {D}^{(m)}_{\tilde Y_i}}  \mathcal O_{\tilde Y_i} 
\otimes^L_{\bar\iota_i^{-1}\mathcal O_{\mathbb P_i}}\bar\iota_i^{-1}
\mathcal D^{(m)}_{\mathbb P_i})\Big)_{\mathbb Q}\\
&\cong& \varinjlim_m \Big(R\varprojlim_i (R\bar\iota_{i*}\omega_{\tilde Y_i}
\otimes_{\mathcal O_{\mathbb P_i}}\mathcal D^{(m)}_{\mathbb P_i})\Big)_{\mathbb Q}.
\end{eqnarray*}
By Corollary \ref{cor:usedin2} in the appendix, we have 
\begin{eqnarray*}
R\bar\iota_{i*}\omega_{\tilde Y_i} \cong
\iota_{i*}\omega_{\tilde Y_i},\quad
\mathcal O_{\tilde Y_i} \otimes_{\mathcal O_{\tilde Y_{i+1}}}  \iota_{i*}\omega_{\tilde Y_{i+1}} \cong \iota_{i*}\omega_{\tilde Y_i}.
\end{eqnarray*}
We thus have  
\begin{eqnarray*}
\bar\iota_{k+}(\omega_{\widehat{\tilde Y},\Q}\otimes_{\c{O}_{\widehat{\tilde Y},\Q}}\mathcal {D}^\d_{\widehat{\tilde Y},\Q})
&\cong &\varinjlim_m \Big(R\varprojlim_i (R\bar\iota_{i*}\omega_{\tilde Y_i}
\otimes_{\mathcal O_{\mathbb P_i}}\mathcal D^{(m)}_{\mathbb P_i})\Big)_{\mathbb Q}\\
&\cong& \varinjlim_m \Big(\varprojlim_i (\bar\iota_{i*}\omega_{\tilde Y_i}
\otimes_{\mathcal O_{\mathbb P_i}}\mathcal D^{(m)}_{\mathbb P_i})\Big)_{\mathbb Q}\\
&\cong& (\bar\iota_*\omega_{\tilde Y})^\wedge_{\mathbb Q}\otimes_{\c{O}_{\widehat{\b{P}}, \Q}}\mathcal {D}^\d_{\b{\widehat P}, \Q}
\quad(\hbox{\cite[4.1.5]{EGAIII}}).
\end{eqnarray*}
By Corollary \ref{toroidalDmodiso'}, we have an exact sequence 
\begin{equation}\label{resolution}
\f{L}(H_K)\otimes_{K}\omega_{\widehat{\tilde Y},\Q}\otimes_{\c{O}_{\widehat{\tilde Y},\Q}}\mathcal {D}^\d_{\widehat{\tilde Y},\Q}
\to \omega_{\widehat{\tilde Y},\Q}\otimes_{\c{O}_{\widehat{\tilde Y},\Q}} \mathcal {D}^\d_{\widehat{\tilde Y},\Q}\to \D(^\d {D}_k)
\omega_{\widehat{\tilde Y},\Q}\to 0 
\end{equation}
of right $\mathcal {D}^\d_{\widehat{\tilde Y},\Q}$-modules.
Applying $\bar\iota_{k+}$ to the sequence (\ref{resolution}), we get morphisms 
$$\f{L}(H_K)\otimes_{K} (\bar\iota_*\omega_{\tilde Y})^\wedge_{\mathbb Q}
\otimes_{\c{O}_{\widehat{\b{P}}, \Q}}\mathcal {D}^\d_{\b{\widehat P}, \Q}
\to  (\bar\iota_*\omega_{\tilde Y})^\wedge_{\mathbb Q}\otimes_{\c{O}_{\widehat{\b{P}}, \Q}}\mathcal {D}^\d_{\b{\widehat P}, \Q}
\to \mathbb D \bar \iota_{k+}(^\d {D}_k)\omega_{\widehat{\tilde Y},\Q}$$ 
in the derived category of right $\mathcal {D}_{\widehat{\b{P}},\b{Q}}^\dagger$-modules whose composite vanishes,
where for the last term, we use the fact that $\mathbb D$ commutes with $\bar \iota_{k+}$ since $\bar\iota$ is a proper morphism. 
It induces a morphism
$$\psi: \mathcal N'\to \mathbb D \bar \iota_{k+}(^\d {D}_k)\omega_{\widehat{\tilde Y},\Q}.$$
Let $E=\mathbb P-\prod_{j=1}^N\mathrm{GL}(V_j)$. 
Then $E$ is a divisor of $\mathbb P$, $\bar\iota\,^{-1}(E)=D$ by Proposition \ref{prop:inv0inf}, and 
$\bar\iota$ induces an affine morphism $$G=\tilde Y-D \to \mathbb P-E=\prod_{j=1}^N\mathrm{GL}(V_{j}).$$
In particular, $\widehat{\bar\iota}_*|_{\mathbb P_k-E_k}$ is an exact functor.
We have $$\bar\iota_{k+}(\hbox{-})\cong R\widehat{\bar\iota}_*(\hbox{-}
\otimes^L_{\mathcal {D}^\d_{\widehat{\tilde Y},\Q}} \mathcal D_{\widehat {\tilde Y}\to \widehat P} ).$$ 
So $\bar\iota_{k+}|_{\mathbb P_k-E_k}$ is right exact. From the right exact sequence (\ref{resolution}), we get a right exact sequence 
$$\f{L}(H_K)\otimes_{K} (\bar\iota_*\omega_{\tilde Y})^\wedge_{\mathbb Q}\otimes_{\c{O}_{\widehat{\b{P}}, \Q}}
\mathcal {D}^\d_{\b{\widehat P}, \Q} |_{\mathbb P_k-E_k}
\to  (\bar\iota_*\omega_{\tilde Y})^\wedge_{\mathbb Q}\otimes_{\c{O}_{\widehat{\b{P}}, \Q}}\mathcal {D}^\d_{\b{\widehat P}, \Q}|_{\mathbb P_k-E_k}
\to \mathbb D \bar \iota_{k+}(^\d {D}_k)\omega_{\widehat{\tilde Y},\Q}|_{\mathbb P_k-E_k}\to 0.$$
Hence $\psi|_{\mathbb P_k-E_k}$ is an isomorphism. 
Since $\D\mathcal N'\in D^{\m b}_{\mathrm{coh}}(\mathcal {D}^\d_{\b{\widehat P}, \Q})$,
we have $(^\d {E}_k)\D\mathcal N'\in \mathrm{ob}\,D^{\m b}_{\mathrm{coh}}(\mathcal {D}^\d_{\b{\widehat P}, \Q}(^\d {E}_k))$. 
We have
\begin{eqnarray*}
\bar \iota_{k+}(^\dagger D_k)\omega_{\widehat{\tilde Y},\mathbb Q}
&\cong&\widehat{\bar \iota}_{+}(^\dagger D_k)\omega_{\widehat{\tilde Y},\mathbb Q}
\cong \widehat{\bar\iota}_{(E_k, D_k),+}  (^\dagger {D}_k)\omega_{\widehat{\tilde Y},\mathbb Q} \quad \hbox{(\cite[1.1.9]{Caro-L})}
\\
(^\dagger {E}_k)\bar \iota_{k+}(^\dagger D_k)\omega_{\widehat{\tilde Y},\mathbb Q}
&\cong& (^\dagger {E}_k)\widehat{\bar \iota}_{+}(^\dagger D_k)\omega_{\widehat{\tilde Y},\mathbb Q}
\cong\widehat{\bar\iota}_{(E_k, D_k),+}  (^\dagger {D}_k)(^\dagger D_k)\omega_{\widehat{\tilde Y},\mathbb Q}
 \quad \hbox{(\cite[1.1.10]{Caro-L})}
\\
&\cong& \widehat{\bar\iota}_{(E_k, D_k),+}  (^\dagger {D}_k)\omega_{\widehat{\tilde Y},\mathbb Q}.\quad \hbox{(\cite[1.1.8]{Caro-L})}
\end{eqnarray*} 
We thus have $$\bar \iota_{k+}(^\dagger D_k)\omega_{\widehat{\tilde Y},\mathbb Q}
\cong (^\dagger {E}_k)\bar \iota_{k+}(^\dagger D_k)\omega_{\widehat{\tilde Y},\mathbb Q}
\in \mathrm{ob}\,D^{\m b}_{\mathrm{coh}}(\mathcal {D}^\d_{\b{\widehat P}, \Q}(^\d {E}_k)).$$
Since $\bar\iota_k$ induces an affine quasi-finite morphism $\iota_k: \tilde Y_k-D_k\to \mathbb P_k-E_k$, 
we have $$\widehat{\bar\iota}_{(E_k, D_k),+}  (^\dagger {D}_k)\omega_{\widehat{\tilde Y},\mathbb Q}
\cong \mathcal H^0(\widehat{\bar\iota}_{(E_k, D_k),+}  (^\dagger {D}_k)\omega_{\widehat{\tilde Y},\mathbb Q}).$$
by \cite[1.3.13]{AC18padicweight}. 
Combined with \cite[1.3.1]{AC18padicweight}, we have
$$\mathbb D \bar \iota_{k+}(^\d {D}_k)\omega_{\widehat {\tilde Y},\Q}\cong
\mathcal H^0(\mathbb D \bar \iota_{k+}(^\d {D}_k)\omega_{\widehat {\tilde Y},\Q}).$$
Since $\psi|_{\mathbb P_k-E_k}$ is an isomorphism and 
$(^\dagger {E}_k) \bar \iota_{k+}(^\dagger D_k)\omega_{\widehat{\tilde Y},\mathbb Q}, (^\d {E}_k)\D\mathcal N'\in
\mathrm{ob}\, D^b_{\mathrm{coh}}(\mathcal D^\dagger_{\mathbb P, \mathbb Q}(^\dagger E_K))$, 
by \cite[4.3.12 (ii)]{Berthelot1}, the Verdier dual 
$$\D (\psi): \bar \iota_{k+}(^\dagger D_k)\omega_{\widehat{\tilde Y},\mathbb Q}\to \mathbb D\mathcal N'$$
of $\psi$ induces an isomorphism
$$(^\dagger {E}_k)\circ \D (\psi):(^\dagger {E}_k) \bar \iota_{k+}(^\dagger D_k)\omega_{\widehat{\tilde Y},\mathbb Q}\xrightarrow{\sim} 
(^\d {E}_k)\D\mathcal N'.$$
We have a commutative diagram 
$$\begin{tikzcd}
\bar \iota_{k+}(^\dagger D_k)\omega_{\widehat{\tilde Y},\mathbb Q}\arrow[r,"\mathbb D(\psi)"]\arrow[d,"\cong"]& \mathbb D\mathcal N'\arrow[d]\\
(^\dagger {E}_k) \bar \iota_{k+}(^\dagger D_k)\omega_{\widehat{\tilde Y},\mathbb Q}\arrow[r, "\cong"]&(^\d {E}_k)\mathbb D\mathcal N'.
\end{tikzcd}$$
Since the left and the bottom arrows are isomorphisms, $\mathbb D(\psi)$ is left invertible. So $\psi$ is right invertible, and hence
$\mathbb D \bar \iota_{k+}(^\d {D}_k)\omega_{\widehat {\tilde Y},\Q}$ is a direct summand of $\mathcal N'$.
\end{proof}

Let $\infty=\mathbb P_k-\mathbb V_k$. Applying the functor $(^\d\infty)$ to the result in
Proposition \ref{prop:directfactor}, we get the following.

\begin{corollary}\label{directsummand} In the category of right $\mathcal {D}^\d_{\widehat{\b P},\b{Q}}(^\dagger\infty)$-modules, we have 
$$(^\dagger\infty)\D\bar\iota_{k+}(^\d {D}_k)\omega_{\widehat{\tilde Y}, \Q}\cong 
\c{H}^0((^\dagger\infty)\D\bar\iota_{k+}(^\d {D}_k)\omega_{\widehat{\tilde Y}, \Q}),$$ 
and $(^\dagger\infty)\D\bar\iota_{k+}(^\d {D}_k)\omega_{\widehat{\tilde Y}, \Q}$ is a direct summand of 
$$\mathcal N
:=((\bar\iota_*\omega_{\tilde Y})^\wedge_{\mathbb Q}\otimes_{\c{O}_{\widehat{\b{P}}, \Q}}\mathcal {D}^\d_{\b{\widehat P}, \Q}(^\dagger\infty))
/\sum_{\xi\in \f{L}(H_K)}L_{\xi}((\bar\iota_*\omega_{\tilde Y})^\wedge_{\mathbb Q}
\otimes_{\c{O_{\widehat{\b{P}}, \Q}}}\mathcal {D}^\d_{\b{\widehat P}, \Q}(^\dagger\infty)).$$
\end{corollary}

We do not know whether $\mathcal N$ is overholonomic and whether there exists a Frobenius structure on ${\mathcal N}$. 
It is a coherent $\mathcal {D}^\d_{\widehat{\mathbb P}, \Q}(^\dagger\infty)$-module.
We call the Fourier transform $\f{F}_\pi(\c{N})$ of $\mathcal N$ the 
\emph{modified hypergeometric $\mathcal {D}^\d_{\b{\widehat P}, \Q}(^\dagger\infty)$-module}. 

\begin{proposition}\label{prop:mhyp} The hypergeometric right $\mathcal {D}^\d_{\b{\widehat P}, \Q}(^\dagger\infty)$-module 
$\mathrm{Hyp}_{\pi, !}$
is a direct summand of the modified hypergeometric $\mathcal {D}^\d_{\b{\widehat P}, \Q}(^\dagger\infty)$-module $\f{F}_\pi(\c{N})$.
\end{proposition} 

\begin{proof} By Proposition \ref{prop:Fourier}, we have $\mathrm{Hyp}_{\pi, !}\cong \mathfrak F_\pi(\iota_{k!}\omega^\dagger_{G_k})$.
It suffices to show $$\iota_{k!}\omega^\dagger_{G_k}\cong (^\dagger\infty)\D\bar\iota_{k+}(^\d {D}_k)\omega_{\widehat{\tilde Y}, \Q}.$$
 Denote by $\b{D}_\infty$ the Verdier dual on the category 
$D_\m{coh}^b(\mathcal {D}^\d_{\widehat{\b P},\b{Q}}(^\dagger\infty))$. We have 
\begin{eqnarray*}
\iota_{k!}\omega^\dagger_{G_k}&\cong& \mathbb D_\infty \widehat{\bar\iota}_{(\infty,D_k),+}(^\d {D}_k)\omega_{\widehat{\tilde Y}, \Q},\\
 (^\dagger\infty)\D\bar \iota_{k+}(^\dagger D_k)\omega_{\widehat{\tilde Y},\mathbb Q}
&\cong& (^\dagger\infty)\D\widehat{\bar \iota}_{+}(^\dagger D_k)\omega_{\widehat{\tilde Y},\mathbb Q}
\cong (^\dagger\infty)\D \widehat{\bar\iota}_{(\infty, D_k),+}  (^\dagger {D}_k)\omega_{\widehat{\tilde Y},\mathbb Q} \quad \hbox{(\cite[1.1.9]{Caro-L})}\\
&\cong& \mathbb D_\infty \widehat{\bar\iota}_{(\infty,D_k),+}(^\d {D}_k)\omega_{\widehat{\tilde Y}, \Q}\quad
\hbox{(\cite[I.4.4]{V}, \cite[4.3.11]{Berthelot1})}.
\end{eqnarray*}
Our assertion follows. 
\end{proof}

\section{Invariant differential operators}

To study the modified hypergeometric right arithmetic $\mathcal D$-module 
$$\mathfrak F_\pi\Big(((\bar\iota_*\omega_{\tilde Y})^\wedge_{\mathbb Q}\otimes_{\c{O}_{\widehat{\b{P}}, \Q}}
\mathcal {D}^\d_{\b{\widehat P}, \Q}(^\dagger\infty))
/\sum_{\xi\in \f{L}(H)}L_{\xi}((\bar\iota_*\omega_{\tilde Y})^\wedge_{\mathbb Q}
\otimes_{\c{O_{\widehat{\b{P}}, \Q}}}\mathcal {D}^\d_{\b{\widehat P}, \Q}(^\dagger\infty))\Big),$$
we need to lift it to a right $\mathcal D^{(m)}$-module, construct a good filtration on the lifting, and study the characteristic cycle. 
But $L_\xi$ is nilpotent in $\mathcal D^{(m)}\mod\mathfrak m$ and has not contribution to the characteristic cycle.
We need other invariant differential operators in $\mathcal D^{(m)}$ to describe the lifting and the characteristic cycle. 

Let $H$ be a smooth affine group $R$-scheme, $R[H]$ the affine coordinate ring of $H$, and $I$ the ideal 
of $R[H]$ corresponding to the closed immersion $1_H: \mathrm{Spec}\,R\to H$ of the unit section. We define the Lie algebra $H$
to be the $R$-module 
$$\mathfrak L(H):=\mathrm{Hom}_R(I/I^2, R).$$ 
By \cite[1.4.2]{Berthelot1}, for the pair $(R[H],I)$, we have the $m$-PD envelope $P_{(m)}(I)$  
for each nonnegative integer $m$, and the $m$-PD quotient $P^n_{(m)}(I)$ for each nonnegative integer $n$. We define 
the \emph{universal enveloping algebra $U^{(m)}(\f{L}(H))$ of $\mathfrak L(H)$ of level $m$} by  
$$F^{n}U^{(m)}(\f{L}(H))=\m{Hom}_R (P^n_{(m)}(I), R),\quad U^{(m)}(\f{L}(H))=\varinjlim_n\, F^{n}U^{(m)}(\f{L}(H)).$$
The coproduct $R[H]\to R[H]\otimes_R R[H]$ induces a homomorphism 
$$P^{n_1+n_2}_{(m)}(I)\to P^{n_1}_{(m)}(I)\otimes_R P^{n_2}_{(m)}(I).$$ 
Given $\eta_1\in F^{n_1}U^{(m)}(\f{L}(H))$ and $\eta_2\in F^{n_2}U^{(m)}(\f{L}(H))$,  we define the product 
$\eta_1\eta_2\in F^{n_1+n_2}U^{(m)}(\f{L}(H))$ to be the composite 
$$P^{n_1+n_2}_{(m)}(I)\to P^{n_1}_{(m)}(I)\otimes_R P^{n_2}_{(m)}(I)\xrightarrow{\mathrm{id}\otimes \eta_2} 
P^{n_1}_{(m)}(I)\xrightarrow{\eta_1} R.$$
We have $R[H]/I\cong R$. The structure morphism $R\to R[H]/I^{n+1}$ provides a section of $R[H]/I^{n+1}\to R[H]/I$. 
By \cite[1.5.4]{Berthelot1}, we have 
$$P^1_{(m)}(I)\cong R[H]/I^2\cong R\oplus I/I^2.$$ So we have
$$F^1U^{(m)}(\mathfrak L(H))\cong R\oplus \mathfrak L(H).$$

Let $\Delta: H\to H\times H$ be the diagonal morphism, and let $\mathcal I_\Delta$ be its ideal sheaf. We have 
the $m$-PD envelope $\mathcal P_{(m)}(\mathcal I_\Delta)$ and the quotient sheaves $\c P^n_{(m)}(\mathcal I_{\Delta})$
for the pair $(\mathcal O_{H\times H},\mathcal I_\Delta)$. 
The sheaf  of differential operators $\mathcal D_H^{(m)}$ on $H$ of level $m$ and
the subsheaf $F^n \mathcal D_H^{(m)}$ of differential operators of order $\leq n$ are defined to be
$$F^n \mathcal D_H^{(m)}= \mathcal Hom_{\mathcal O_H}(\c P^n_{(m)}(\mathcal I_{\Delta}), \mathcal O_H), 
\quad  \mathcal D_H^{(m)}=\varinjlim_n F^n \mathcal D_H^{(m)}.$$

\begin{proposition}\label{bij-of-inv} 
Let $\Gamma_{\m{inv}}(H, \c D^{(m)}_H)\subset \Gamma(H, \c D^{(m)}_H)$ be the subspace of right invariant differential operators.
We have a canonical isomorphism of $R$-modules $$U^{(m)}(\f{L}(H))\cong \Gamma_{\m{inv}}(H, \c D^{(m)}_H).$$
\end{proposition}

\begin{proof} 
Let $A$ be the isomorphism $$H\times H \to H\times H, \quad (g,h)\mapsto (g, gh^{-1}),$$ 
We have a commutative diagram
$$\begin{tikzcd}
H \arrow[rd, "{(\m{id}, 1_H)}"'] \arrow[r, "\Delta"] & H\times H \arrow[d, "{A}"] \\
                                                  & H\times H,
\end{tikzcd}$$
So $A$ induces an isomorphism 
$$A^*: \c P^n_{(m)}(\mathcal I_{(\m{id}, 1_H)})\stackrel\cong\to \c P^n_{(m)}(\mathcal I_{\Delta}),$$ where $\mathcal I_{(\m{id}, 1_H)}$ 
is the ideal sheaf for the closed immersion $(\m{id}, 1_H)$. 
By \cite[1.4.6]{Berthelot1}, we have  $$\c{O}_{H}\otimes_R P^n_{(m)}(I)\cong \c P^n_{(m)}(\mathcal I_{(\m{id}, 1_H)}).$$
Since $A(g_1h,g_2h)=(g_1h,g_1g_2^{-1})$, under the identification 
$$A^*:\c{O}_{H}\otimes_R P^n_{(m)}(I)\stackrel\cong\to \c P^n_{(m)}(\mathcal I_{\Delta}),$$
the right multiplication by $H$ on $\c P^n_{(m)}(I_{\Delta})$ is identified with the right multiplication  of $H$
on $\c{O}_{H}\otimes_R P^n_{(m)}(I)$ acting trivially on the factor $P^n_{(m)}(I)$ and acting as usual on $\c{O}_{H}$.
We have 
\begin{align*}
F^n\c D^{(m)}_{H}&= \mathcal Hom_{\mathcal O_H}(\c P^n_{(m)}(\mathcal I_{\Delta}), \mathcal O_H)
\cong \mathcal Hom_{\mathcal O_H}(\c{O}_{H}\otimes_R P^n_{(m)}(I), \mathcal O_H)\\
&\cong \c{O}_{H}\otimes_R \m{Hom}_R (P^n_{(m)}(I), R)
\cong \c{O}_{H}\otimes_R F^{n}U^{(m)}(\f{L}(H)).
\end{align*} 
A global section of $\c{O}_{H}\otimes_R  F^{n}U^{(m)}(\f{L}(H))$
is right invariant if and only if it lies in $1\otimes F^{n}U^{(m)}(\f{L}(H))$.
Moreover, we have 
\[U^{(m)}(\f{L}(H))= \Gamma_{\mathrm{inv}}\Big(H, \c{O}_{H}\otimes_RU^{(m)}(\f{L}(H))\Big)
\cong \Gamma_{\mathrm{inv}}(H, \c D^{(m)}_{H}).\qedhere\]
\end{proof}

\begin{definition}\label{defn:rightinvoper} ${}$
\begin{enumerate}[(i)]
\item For any $\xi\in \mathfrak L(H)$, regard $\xi$ as an element in $U^{(1)}(\mathfrak L(H))$. 
We denote the right invariant differential operator in $\Gamma_{\mathrm{inv}}(H, \mathcal D^{(1)}_H)$ corresponding to $\xi$ 
by $L_\xi$. 
\item Note that $I/I^2$ is a free module over $R[H]/I\cong R$. Let $s$ be its rank. Choose $t_1, \ldots, t_s\in I$ so that 
their images in $I/I^2$ form a basis. By \cite[1.5.3 (ii)]{Berthelot1}, 
$P^n_{(m)}(I)$ is a free $R$-module with the basis $t_1^{\{k_1\}_{(m)}}\cdots t_s^{\{k_s\}_{(m)}}$ $(k_i\geq 0, \; k_1+\cdots+k_s\leq n)$. 
Let $\xi_1^{\langle k_1\rangle_{(m)}}\cdots \xi_s^{\langle k_s\rangle_{(m)}}$ $(k_i\geq 0, \; k_1+\cdots+k_s\leq n)$ be the dual basis
for $F^{n}U^{(m)}(\f{L}(H))=\m{Hom}_R (P^n_{(m)}(I), R)$. In particular, $\{\xi_1, \ldots, \xi_s\}$ is a basis of the Lie algebra $\mathfrak L(H)$.
Each element $\xi_1^{\langle k_1\rangle_{(m)}}\cdots \xi_s^{\langle k_s\rangle_{(m)}}$
defines a right invariant differential operator in $\Gamma_{\m{inv}}(H, \c D^{(m)}_H)$
which we denote by $L_{\xi_1^{\langle k_1\rangle_{(m)}}\cdots \xi_s^{\langle k_s\rangle_{(m)}}}$. It depends on 
the choice of $t_1, \ldots, t_s$. 
\end{enumerate}
\end{definition}

For any differential operator $P$ in $F^r\mathcal D^{(m)}_H$, its image $\sigma_r(P)$ in
$\mathrm{Gr}^r(\mathcal D^{(m)}_H)$ is called the symbol of $P$. 
By abuse of notation, for any vector field $L$ on $H$, we denote its symbol also by $L$, and we denote by $L^r$ the 
$r$-th power of $L$ either in $\mathcal D^{(m)}_H$ or in $\mathrm{Gr}(\mathcal D^{(m)}_H)$. 

\begin{lemma}\label{lm:pdpower} Notation as above. For any $1\leq r\leq p^m$ and $1\leq i\leq s$, we have
$$\sigma_r((-1)^r r! L_{\xi_i^{\langle r\rangle_{(m)}}})=L^r_{\xi_i}.$$ 
\end{lemma}

\begin{proof} We first make preparations for the proof of the lemma. Let $\mathrm{ev}_1: R[H]\to R$ be 
the homomorphism corresponding to $1_H: \mathrm{Spec}\,R\to H$. 
Since the structure morphism $R\to R[H]$ is a section of $\mathrm{ev}_1$, we have 
$$R[H]=R\oplus I.$$ 
By the axiom $g\cdot 1= 1\cdot g=g$, the composites
$$R[H]\stackrel{c}\to R[H]\otimes_R R[H]\stackrel{\mathrm{id}_{R[H]}\otimes\mathrm{ev}_1}
{\underset{\mathrm{ev}_1\otimes \mathrm{id}_{R[H]}}\rightrightarrows} R[H]$$
are identity, where $c$ is the coproduct. 
So for any $t\in I$, we have 
\begin{align*}
c(t)-t\otimes 1&\in \mathrm{ker}(\mathrm{id}_{R[H]}\otimes\mathrm{ev}_1: R[H]\otimes_R R[H]\to R[H])=R[H]\otimes_R I,\nonumber\\
c(t)-1\otimes t&\in \mathrm{ker}(\mathrm{ev}_1\otimes \mathrm{id}_{R[H]}: R[H]\otimes_R R[H]\to R[H])=I\otimes_R R[H],\nonumber\\
c(t)-t\otimes 1-1\otimes t&\in(R[H]\otimes_R I)\cap (I\otimes_R R[H])=I\otimes_R I.
\end{align*} 
We can thus write 
\begin{align}\label{aln:c(t)}
c(t)=t\otimes 1+1\otimes t+\sum_{\lambda} t'_\lambda \otimes t''_\lambda
\end{align}
for some $t'_\lambda, t''_\lambda\in I$. 
Let $d: R[H]\otimes_R R[H]\to R[H]$ be the homomorphism corresponding to the morphism 
$$H\to H\times H, \quad g\mapsto (g, g^{-1}).$$ 
By the axiom $gg^{-1}=1$, the composite 
$$R[H]\stackrel{c}\to R[H]\otimes_R R[H]\stackrel{d}\to R[H]$$ factors through $\mathrm{ev}_1:R[H]\to R$. 
So for any $t\in I$, we have $dc(t)=0$. Substituting the formula (\ref{aln:c(t)}), we get
$$0=t+ i^*(t)+\sum_{\lambda} t'_\lambda i^*(t''_\lambda),$$
where $i^*: R[H]\to R[H]$ is the homomorphism corresponding to the inversion on $H$. 
Note that $i^*$ preserves the ideal $I$. The above equation implies 
$$i^*(t)\equiv -t\mod I^2.$$
Since  $(p_2\circ A)(g, h)=gh^{-1}$, we 
have 
\begin{align*}
A^*(1\otimes t)&= ((\mathrm{id}_{R[H]}\otimes i^*) \circ c)(t)\\
&= t\otimes 1+1\otimes i^*(t)+\sum_{\lambda} t'_\lambda \otimes i^*(t''_\lambda)\\
&\equiv t\otimes 1-1\otimes t \mod (I\otimes_R I + R[H]\otimes_R I^2).
\end{align*}
So we have 
$$A^*(1\otimes t)\equiv -1\otimes t\mod (I\otimes_R R[H]+ R[H]\otimes_R I^2)$$
for any $t\in I$. This implies that 
\begin{align}\label{aln:finalex}
A^*(1\otimes t_1^{k_1}\cdots t_s^{k_s})\equiv (-1)^{k_1+\cdots +k_s}
1\otimes t_1^{k_1}\cdots t_s^{k_s}\mod (I\otimes_R R[H]+ R[H]\otimes_R I^{k_1+\cdots +k_s+1}).
\end{align}
Recall that any differential operator $P$ in 
$$\mathcal D^{(m)}_H=\varinjlim_n \mathcal Hom_{\mathcal O_H}(\mathcal P_{(m)}^n(\mathcal I_\Delta),\mathcal O_H)$$
defines a section of $\mathcal End_{R}(\mathcal O_H)$ so that
$$P(f'\otimes f'')=f'P(f'')$$ for any sections $f'$ and $f''$ of $\mathcal O_H$, 
where on the lefthand (resp. righthand) side, $P$ is regarded as a section of 
$\mathcal Hom_{\mathcal O_H}(\mathcal P_{(m)}^n(\mathcal I_\Delta),\mathcal O_H)$
(resp. $\mathcal End_{R}(\mathcal O_H)$).
Let $L$ be a vector field regarded as a differential operator in $\mathcal D^{(m)}_H$. If $f'\in I$, we have 
$$\mathrm{ev}_1(L^r(f'\otimes f''))=\mathrm{ev}_1(f'L^r(f''))=0$$ 
since $\mathrm{ev}_1(f')=0$. If $f''_1,\ldots, f''_{r+1}\in I$, then applying the Leibniz rule 
to $L^r(f''_1\cdots f''_{r+1})$, we get $\mathrm{ev}_1(L^r( f''_1\cdots f''_{r+1}))=0$ and hence 
$$\mathrm{ev}_1(L^r(f'\otimes f''_1\cdots f''_{r+1}))=\mathrm{ev}_1(f'L^r( f''_1\cdots f''_{r+1}))=0.$$ 
Combined with the equation (\ref{aln:finalex}), we get 
\begin{align}\label{aln:lastex}
\mathrm{ev}_1( L^r(A^*(1\otimes t_1^{k_1}\cdots t_s^{k_s})))=
(-1)^{k_1+\cdots+k_s}\mathrm{ev}_1( L^r(t_1^{k_1}\cdots t_s^{k_s}))
\end{align}
if $r\leq k_1+\cdots+k_s$. 

We are now ready to prove the lemma. It suffices to show 
$$(-1)^rr! L_{\xi_i^{\langle r\rangle_{(m)}}}-L^r_{\xi_i}\in \Gamma(H, F^{r-1} \c D^{(m)}_H).$$
By Proposition \ref{bij-of-inv}, the right invariant differential operator $L^r_{\xi_i}$ corresponds to 
an element $U$ in $U^{(m)}(\f{L}(H))$ such that 
$$L^r_{\xi_i}(A^*(1\otimes t_1^{\{k_1\}_{(m)}}\cdots t_s^{\{k_s\}_{(m)}}))=U(t_1^{\{k_1\}_{(m)}}\cdots t_s^{\{k_s\}_{(m)}})
\quad (k_i\geq 0, k_1+\cdots+k_s\leq r).$$
So $L^r_{\xi_i}(A^*(1\otimes t_1^{\{k_1\}_{(m)}}\cdots t_s^{\{k_s\}_{(m)}}))$ is a constant function and we 
have
$$U(t_1^{\{k_1\}_{(m)}}\cdots t_s^{\{k_s\}_{(m)}})=\mathrm{ev}_1(
L^r_{\xi_i}(A^*(1\otimes t_1^{\{k_1\}_{(m)}}\cdots t_s^{\{k_s\}_{(m)}}))).$$
By the definition of $L_{\xi}$, we have $$L_{\xi_i}(A^*(1\otimes t_j))=\delta_{ij}.$$
On other hand, since $k_j\leq r\leq p^m$, we have $t_j^{\{k_j\}_{(m)}}=t_j^{k_j}$. By the equation (\ref{aln:lastex})
and the Leibnitz rule, we have 
\begin{align*}
&U(t_1^{\{k_1\}_{(m)}}\cdots t_s^{\{k_s\}_{(m)}})=\mathrm{ev}_1(L^r_{\xi_i}(A^*(1\otimes t_1^{\{k_1\}_{(m)}}\cdots t_s^{\{k_s\}_{(m)}})))
=\mathrm{ev}_1(L^r_{\xi_i}(A^*(1\otimes t_1^{k_1}\cdots t_s^{k_s})))
\\=&(-1)^{k_1+\cdots+k_s}\mathrm{ev}_1(L^r_{\xi_i}(t_1^{k_1}\cdots t_s^{k_s}))
=\begin{cases} 0 &\hbox{if } r<k_1+\cdots+k_s,\\
0 &\hbox{if }r=k_1+\cdots+k_s \hbox{ but }k_i\not=r\\
(-1)^r r!&\hbox{if }r=k_1+\cdots+k_s \hbox{ and }k_i=r.
\end{cases}
\end{align*}
So $(-1)^r r!\xi_i^{\langle r_i\rangle}-U\in F^{r-1}U^{(m)}(\mathfrak L(H))$. Hence 
$(-1)^rr! L_{\xi_i^{\langle r\rangle_{(m)}}}-L^r_{\xi_i}\in \Gamma(H, F^{r-1} \c D^{(m)}_H)$.
\end{proof}

For right invariant differential operators in $\mathcal D^{(m)}_{H, \mathbb Q}= \mathcal D^{(m)}_H\otimes_{\mathbb Z}\mathbb Q$, 
we have a more direct description.

\begin{proposition}\label{inv-generated-by-L} The operators $L_{\xi_1}^{i_1}\cdots L_{\xi_s}^{i_s}$ $(i_1, \ldots, i_s\geq 0)$
form a basis of $\Gamma_{\mathrm{inv}}(H, \mathcal D^{(m)}_{H, \mathbb Q})$ over $K$.
\end{proposition}

\begin{proof}  We have $$\mathcal D^{(0)}_{H,\mathbb Q}\cong \mathcal D^{(m)}_{H,\mathbb Q},\quad
\mathrm{Gr}(\mathcal D^{(0)}_{H,\mathbb Q})\cong \mathrm{Sym}_{\mathcal O_H}\mathcal T_H\otimes_{\mathbb Z}\mathbb Q,$$
where $\mathcal T_H$ is the sheaf of tangent vectors on $H$. Since $L_{\xi_i}$ $(i=1,\ldots,s)$ form a basis of $\mathcal T_H$ over 
$\mathcal O_H$, the operators $L_{\xi_1}^{i_1}\cdots L_{\xi_s}^{i_s}$ $(i_1, \ldots, i_s\geq 0)$ form a basis of 
$\mathcal D^{(m)}_{H,\mathbb Q}$ over $\mathcal O_H$. So $L_{\xi_1}^{i_1}\cdots L_{\xi_s}^{i_s}$ $(i_1, \ldots, i_s\geq 0)$ form a basis 
of the space of right invariant differential operators over $K$.
\end{proof}

Let $V$ be a smooth $R$-scheme with a left $H$-action $a: H\times V\to V$. The composite
$$V\xrightarrow{(1_H,\m{id})}H\times V\xrightarrow{a}V$$
is the identity. Let 
$$\gamma: \c D_{H\times V}^{(m)}\to a^* \c D_{ V}^{(m)}$$
be canonical morphism \cite[(2.2.2.2)]{Berthelot1}. It is a morphism of left $\c D_{H\times V}^{(m)}$-modules 
by the description of the left $\c D_{H\times V}^{(m)}$-module structure on $a^* \c D_{ V}^{(m)}$
in \cite[2.1.1]{Berthelot2}.
We then have the morphism of  left $\c D_V^{(m)}$-modules 
$$(1_H,\m{id})^*(\gamma): (1_H,\m{id})^*\c D_{H\times V}^{(m)}\to (1_H,\m{id})^*a^* \c D_{ V}^{(m)}\cong  \c D_{ V}^{(m)}.$$
For any right invariant differential operator 
$P\in\Gamma_{\mathrm{inv}}(H, {\c D}_{H}^{(m)})$, we define $P^V\in\Gamma(V, \c D_{ V}^{(m)})$
by $$P^V=(1_H,\m{id})^*(\gamma)\Big(1\otimes(P\boxtimes 1)\Big).$$
Here we use the fact that $\c D_{H\times V}^{(m)}\cong \c D_H^{(m)}\boxtimes \c D_V^{(m)}$. For simplicity, we often write 
$P^V$ as $P$. 

\begin{lemma}\label{description-alpha}${}$ 

\begin{enumerate}[(i)] 
\item For any $\eta_1,\cdots,\eta_k\in\f{L}(H)$, we have 
$$(L_{\eta_1}\cdots L_{\eta_k})^V=L^V_{\eta_1}\cdots L^V_{\eta_k}.$$ 
\item For any $1\leq r\leq p^m$ and $1\leq i\leq s$, we have
$$\sigma_r((-1)^rr! L_{\xi_i^{\langle r\rangle_{(m)}}}^V)=(L_{\xi_i}^V)^r.$$ 
\end{enumerate}
\end{lemma}

\begin{proof} (i) Using the fact that $L_\eta$ is right invariant, one can show 
that the canonical morphism $a_*: \c{T}_{H\times V} \to a^*\c{T}_V$ maps the vector field $(L_{\eta},0)$ to 
$1\otimes L^V_{\eta}$, and 
$$(L_{\eta}\boxtimes 1)\cdot (1\otimes Q)=1\otimes L_{\eta}^VQ$$ in $a^*\mathcal D_V^{(m)}$
for any $Q\in\Gamma(V, \c D_{ V}^{(m)})$. Since $\gamma$ is a left $\c D_{H\times V}^{(m)}$-module homomorphism, 
we have 
$$\gamma(L_{\eta_1}\cdots L_{\eta_k}\boxtimes 1)=(L_{\eta_1}\boxtimes 1)\cdots (L_{\eta_k}\boxtimes 1) \gamma (1)=
(L_{\eta_1}\boxtimes 1)\cdots (L_{\eta_k}\boxtimes 1) (1\boxtimes 1)=
1\otimes L^V_{\eta_1}\cdots L^V_{\eta_k}.$$
Our assertion follows.

(ii) By Lemma \ref{lm:pdpower} and the fact that $\gamma$ preserves the filtrations by orders of differential 
operators, $(-1)^rr! L_{\xi_i^{\langle r\rangle_{(m)}}}^V$ and $(L_{\xi_i}^r)^V$ differ by a
differential operator of order $\leq r-1$. Our assertion then follows from (i).
\end{proof}

Let $Z$ be a smooth scheme over $k$. Fix notation by the commutative diagram 
$$\begin{tikzcd}
Z\arrow[r,"\mathrm{Fr}^m_{Z}"]\arrow[rr, bend left=15mm,"F_Z^m"]\arrow[rd]&Z^{(m)}\arrow[r]\arrow[d]& Z\arrow[d]\\
&\mathrm{Spec}\,k\arrow[r,"F_{\mathrm{Spec}\,k}^m"]&\mathrm{Spec}\,k
\end{tikzcd}$$ where the square is Cartesian, $F_Z=(\mathrm{id}_Z, F^\natural_Z)$ is the absolute Frobenius morphism
defined by $F^\natural_Z(s)=s^p$, and 
$\mathrm{Fr}_{Z}$ is the relative Frobenius morphism. 
Let $T^*Z^{(m)}\to Z^{(m)}$ be the cotangent bundle of $Z^{(m)}$.

\begin{proposition}[Berthelot]\label{prop:Berthelot}
We have an isomorphism
$$\mathbb{S}\mathrm{pec} (\mathrm{Gr} (\c{D}_{Z}^{(m)}))_{\m{red}}\cong T^*Z^{(m)}\times_{Z^{(m)}} Z,$$
where $\mathbb{S}\mathrm{pec} (\mathrm{Gr} (\c{D}_{Z}^{(m)}))_{\m{red}}$ is the reduced scheme associated to the affine
$Z$-scheme defined by the quasi-coherent $\mathcal O_Z$-module $\mathrm{Gr} (\c{D}_{Z}^{(m)})$.
\end{proposition}

\begin{proof} By \cite[2.2.2]{Berthelot2}, $\mathrm{Fr}^m_{Z}$ induces a canonical morphism 
$$\Phi: \mathcal D_Z^{(m)}\to \mathrm{Fr}^{m*}_{Z}\mathcal D_{Z^{(m)}}^{(0)}.$$ Let's prove this morphism induces 
an isomorphism 
$$(\mathrm{Gr} (\c{D}_{Z}^{(m)}))_{\m{red}}\cong \mathrm{Fr}^{m*}_{Z}\mathrm{Gr}(\mathcal D_{Z^{(m)}}^{(0)}).$$
Combined with the fact $\mathrm{Gr}(\mathcal D_{Z^{(m)}}^{(0)})\cong \mathrm{Sym}_{\mathcal O_{Z^{(m)}}}\mathcal T_{Z^{(m)}}$, 
this implies the proposition. The problem is local, we may assume $Z$ has a coordinate chart
$(x_1, \ldots, x_n)$. By \cite[2.2.3]{Berthelot1}, $\c{D}_{Z}^{(m)}$ is 
a free $\c{O}_Z$-module with basis  $\partial_{x_1}^{\langle k_1\rangle_{(m)}}\cdots \partial_{x_n}^{\langle k_n\rangle_{(m)}}$. 
By \cite[2.2.4(iii)]{Berthelot1}, for any $0\leq j<m$, we have
\begin{align}\label{aln:calpower}
(\partial_{x_i}^{\langle p^j\rangle_{(m)}})^p&=
\Big\langle\begin{array}{c} 2p^j\\
p^j\end{array}\Big\rangle\Big\langle\begin{array}{c} 3p^j\\
2p^j\end{array}\Big\rangle\cdots \Big\langle\begin{array}{c} p p^j\\
(p-1)p^j\end{array}\Big\rangle
\partial_{x_i}^{\langle p^{j+1}\rangle_{(m)}}= \frac{p^{j+1}!}{(p^j!)^p}\partial_{x_i}^{\langle p^{j+1}\rangle_{(m)}},\\
\mathrm{ord}_p\Big(\frac{p^{j+1}!}{(p^j!)^p}\Big)&=\frac{p^{j+1}-1}{p-1}- \frac{p(p^j-1)}{p-1}=1.\nonumber
\end{align}
So $(\partial_{x_i}^{\langle p^j\rangle_{(m)}})^p=0$ for any $0\leq j<m$. 
By \cite[2.2.5(i)]{Berthelot1}, $\mathcal D_Z^{(m)}$ is generated by $\mathcal O_Z$ and $\partial_{x_i}^{\langle p^j\rangle_{(m)}}$
$(1\leq i\leq n, \, 1\leq j\leq m)$ as a sheaf of rings. 
It follows that $$(\mathrm{Gr} (\c{D}_{Z}^{(m)}))_{\m{red}}\cong 
\c{O}_Z[\xi_1^{\langle p^m\rangle_{(m)}}, \ldots, \xi_n^{\langle p^m\rangle_{(m)}}],$$
where $\xi_i^{\langle p^m\rangle_{(m)}}$ is the image of $\partial_{x_i}^{\langle p^m\rangle_{(m)}}$ 
in $(\mathrm{Gr} (\c{D}_{Z}^{(m)}))_{\m{red}}$ and it is homogeneous of degree
$p^m$.  Let $(x^{(m)}_1,\ldots, x^{(m)}_n)$ be the local coordinate chart for $Z^{(m)}$ obtained by base change from the coordinate chart
$(x_1, \ldots, x_n)$ of $Z$. Then we have 
$$\mathrm{Gr} (\c{D}_{Z^{(m)}}^{(0)})\cong 
\c{O}_{Z^{(m)}}[\xi^{(m)}_1, \ldots, \xi^{(m)}_n],$$
where $\xi^{(m)}_i$ is the image of $\partial_{x^{(m)}_i}$ 
in $\mathrm{Gr} (\c{D}_{Z^{(m)}}^{(0)})$ and it is homogeneous of degree $1$. 
By \cite[2.2.4]{Berthelot2}, we have 
$$\Phi(\partial_{x_i}^{\langle p^m\rangle_{(m)}})=1\otimes \partial_{x^{(m)}_i}.$$
So $\Phi$ induces an isomorphism 
\[(\mathrm{Gr} (\c{D}_{Z}^{(m)}))_{\m{red}}\stackrel\cong\to
\mathrm{Fr}^{m*}_{Z}\mathrm{Gr}(\mathcal D_{Z^{(m)}}^{(0)}),\quad \xi_i^{\langle p^m\rangle_{(m)}}\mapsto 1\otimes
\xi^{(m)}_i.\qedhere\]
\end{proof}

Let $V$ be a smooth $R$-scheme and let $V_k=V\otimes_Rk$. Denote $\mathcal D_V^{(m)}\otimes_R k$ by $\mathcal D_{V,k}^{(m)}$. 
By \cite[2.2.2]{Berthelot1}, we have 
$$\mathcal D_{V,k}^{(m)}\cong \mathcal D_{V_k}^{(m)}.$$

\begin{proposition}\label{def-L^<k>}
For any vector field $L$ on a smooth $R$-scheme
$V$ and any integer $1\leq r\leq p^m$, there exists a unique element $L^{\langle r\rangle_{(m)}}\in \mathrm{Gr}^r (\c D_{V}^{(m)})$ 
such that $r!L^{\langle r\rangle_{(m)}}=L^r$. 
The functions $L^{\left<{r}\right>_{(m)}}$ 
on $\mathrm{Spec}\, (\mathrm{Gr}(\c D_{V,k}^{(m)}))_{\m{red}}$ vanish for $1\leq r<p^m$, 
and the function $L^{\left<{p^m}\right>_{(m)}}$ on $\mathrm{Spec}\, (\mathrm{Gr}(\c D_{V,k}^{(m)}))_{\m{red}}$ 
can be identified with the composite
\begin{equation}\label{composite}
\mathrm{Spec}\, (\mathrm{Gr} (\mathcal D_{V,k}^{(m)}))_{\m{red}}\cong T^*V_{k}^{(m)}\times_{V_{k}^{(m)}} V_{k}
\to T^*V_{k}^{(m)}\xrightarrow{L^{(m)}}\b{A}^1_k,   
\end{equation}
where $L^{(m)}: T^*V_{k}^{(m)}\to\b{A}^1_k$ is the base change of $L: T^*V_{k}\to \b{A}_k^1$ by 
the Frobenius morphism $F_{\mathrm{Spec}\,k}^m: \mathrm{Spec}\, k\to \mathrm{Spec}\, k$.
\end{proposition}

\begin{proof}
The uniqueness of $L^{\langle r\rangle_{(m)}}$ follows from the fact that
$\mathrm{Gr}(\c D_{V}^{(m)})$ is a locally free $\c{O}_{V}$-module and $\c{O}_{V}$ is flat over $\b{Z}$.
The existence is a local problem. Assume $V$ has a coordinate chart $(x_1,\ldots, x_n)$. 
First consider the case $L=\partial_{x_i}$. We define $L^{\langle r\rangle_{(m)}}=\sigma_r(\partial_{x_i}^{\langle r\rangle_{(m)}})$. 
By \cite[2.2.4 (iii)]{Berthelot1}, we have $$\partial_{x_i}^r=r!\partial_{x_i}^{\langle r\rangle_{(m)}} \quad (1\leq r\leq p^m).$$
So we have $r!L^{\langle r\rangle_{(m)}}=L^r$ $(1\leq r\leq p^m)$. 
By the calculation in (\ref{aln:calpower}), $L^{\langle p^j\rangle_{(m)}}$ vanishes on $\mathrm{Spec}\, 
(\mathrm{Gr}(\c D_{V,k}^{(m)}))_{\m{red}}$ for all
$1\leq j<m$. By \cite[(2.2.5.1)]{Berthelot1}, $L^{\langle r\rangle_{(m)}}$ vanishes on $\mathrm{Spec}\, 
(\mathrm{Gr}(\c D_{V,k}^{(m)}))_{\m{red}}$ for all
$1\leq r<p^m$. By the proof of Proposition \ref{prop:Berthelot}, the image of $\partial_{x_i}^{\langle p^m\rangle_{(m)}}$ in 
$\mathrm{Gr}(\c D_{V,k}^{(m)}))_{\m{red}}$ corresponds to the image of $\partial_{x^{(m)}_i}$ in $\mathrm{Gr}(\c D_{V}^{(0)})$. So 
the function $L^{\langle p^m\rangle_{(m)}}$ corresponds to the base change of $L: T^*V_{k}\to \b{A}_k^1$ by the Frobenius morphism. 
This proves the existence of $L^{\langle r\rangle_{(m)}}$ for $L=\partial_{x_i}$.
In general, we may write $L=f_1 \partial_{x_1}+\cdots+ f_n\partial_{x_n}$
for some sections $f_i$ of $\mathcal O_V$. We define $L^{\langle r\rangle_{(m)}}$ $(1\leq r\leq p^m)$ 
in $\mathrm{Gr}(\c D_{V}^{(m)})$ by 
$$L^{\langle r\rangle_{(m)}}=\sum_{r_1+\cdots+r_n=r} f_1^{r_1}\cdots f_n^{r_n} 
\partial^{\langle r_1\rangle_{(m)}}_{x_1}\cdots\partial^{\langle r_n\rangle_{(m)}}_{x_n}.$$
Then in the sheaf $\mathrm{Gr}(\c D_{V}^{(m)})$ of commutative rings, we have 
$$r!L^{\langle r\rangle_{(m)}}=\sum_{r_1+\cdots+r_n=r}\frac{r!}{r_1!\cdots r_n!} f_1^{r_1}\cdots f_n^{r_n} 
\partial^{r_1}_{x_1}\cdots\partial^{ r_n}_{x_n}=L^r.$$
The functions on $\mathrm{Spec}\, (\mathrm{Gr}(\c D_{V,k}^{(m)}))_{\m{red}}$ defined by $L^{\langle r\rangle_{(m)}}$ $(1\leq r<p^m)$ vanish 
since the functions defined by $\partial^{\langle r_i\rangle_{(m)}}_{x_i}$ $(1\leq r_i< p^m)$ vanish. As functions on 
$\mathrm{Spec}\, (\mathrm{Gr}(\c D_{V,k}^{(m)}))_{\m{red}}$, we have
\[L^{\langle p^m\rangle_{(m)}}=f_1^{p^m}
\partial^{\langle p^m\rangle_{(m)}}_{x_1}+\cdots+f_n^{p^m}
\partial^{\langle p^m\rangle_{(m)}}_{x_n}.\qedhere\]
\end{proof}

\begin{corollary}\label{heighest-symbol-of-nabla} Notation as Definition \ref{defn:rightinvoper}.
For any $1\leq r\leq p^m$ and $1\leq i\leq s$,
we have 
\begin{align*}
&\sigma_r((-1)^rL_{\xi_i^{\langle r\rangle_{(m)}}})=(L_{\xi_i})^{\langle r\rangle_{(m)}}\hbox{ in }\mathrm{Gr}^r(\c D_{H}^{(m)}),\\ 
&\sigma_r((-1)^rL_{\xi_i^{\langle r\rangle_{(m)}}}^V)=(L_{\xi_i}^V)^{\langle r\rangle_{(m)}}\hbox{ in }\mathrm{Gr}^r(\c D_V^{(m)}), 
\end{align*}
\end{corollary}

\begin{proof} This follows from the uniqueness part of Proposition \ref{def-L^<k>}, Lemmas \ref{lm:pdpower} and \ref{description-alpha}.
\end{proof}

\begin{proposition}\label{beta0} 
Let $M$ be a left $R[H]$-comodule with coaction
$$a: M\to R[H]\otimes_RM.$$
Let $\beta$ be the $R$-linear map
$$\beta: \Gamma(H, {\mathcal D}_{H}^{(m)})\to \m{End}_R (M)$$ so that for any $P\in \Gamma(H, {\mathcal D}_{H}^{(m)})$, $\beta(P)$ is 
the composite
$$\beta(P):M\xrightarrow{a}  R[H]\otimes_RM\xrightarrow{P\otimes \m{id}_M} 
R[H]\otimes_RM \xrightarrow{\mathrm{ev}_1\otimes \m{id}_M}M.$$ 
For any $P\in\Gamma(H, {\mathcal D}_{H}^{(m)})$ and $\xi\in \mathfrak L(H)$, we have
$$\beta(PL_{\xi})=\beta(P) \beta( L_{\xi}).$$ In particular, for any $\zeta_1,\cdots,\zeta_n\in\f{L}(H)$, we have 
$$\beta(L_{\zeta_1}\cdots L_{\zeta_n})=\beta(L_{\zeta_1})\cdots \beta(L_{\zeta_n}).$$
\end{proposition}

\begin{proof}
Since $L_\xi$ is right $H$-invariant, we have a commutative diagram
$$\begin{tikzcd}
R[H]\arrow[d, "c"] \arrow[r, "L_\xi"]       & R[H]\arrow[d, "c"] \\
R[H]\otimes_R R[H] \arrow[r, " L_\xi \otimes_R \m{id}"] & R[H]\otimes_R R[H],
\end{tikzcd}$$ 
where $c:R[H]\to R[H]\otimes_R R[H]$ is the coproduct. 
So we have a commutative diagram
$$\begin{tikzcd}
M \arrow[r, "a"] \arrow[d, "a",swap]& {R[H] \otimes_{R} M} \arrow[d, "\m{id}\otimes a ",swap] 
\arrow[r, "L_\xi \otimes \m{id}"] & {R[H] \otimes_{R} M} \arrow[d, "\m{id}\otimes a ",swap] \arrow[r, "\m{ev_1}\otimes\m{id}"] & M 
\arrow[d, "a",swap] \\
{R[H] \otimes_{R} M} \arrow[r, "c\otimes \m{id}"] \arrow[rrd, "L_\xi \otimes \m{id}"'] & {R[H] \otimes_{R} R[H] \otimes_{R} M} 
\arrow[r, "L_\xi \otimes \m{id}\otimes\mathrm{id}"]& {R[H] \otimes_{R} R[H] \otimes_{R} M} \arrow[r, "\m{ev_1}\otimes
\m{id}\otimes\mathrm{id}"]& {R[H] \otimes_{R} M.}\\
&& {R[H] \otimes_{R} M} \arrow[u, "c\otimes \m{id}"] \arrow[ru, "\m{id}\otimes\mathrm{id}"']&&
\end{tikzcd}$$
This shows that $(L_\xi \otimes \m{id}_M)a(x)=a(\beta(L_{\xi})(x))$ for all $x\in M$. Hence
\begin{align*}
\beta(P L_{\xi})(x)&=(\m{ev_1}\otimes\m{id}_M)(PL_{\xi}\otimes\m{id}_M) a(x)
=(\m{ev_1}\otimes\m{id}_M)(P\otimes\m{id}_M) a(\beta(L_{\xi}) (x))\\
&=\beta (P)\beta( L_{\xi})(x).\qedhere
\end{align*}
\end{proof}

Choose a nonzero left invariant top differential form $\omega_0$ on $H$. We have an isomorphism 
$$\mathcal O_H\stackrel\cong\to  \omega_H, \quad f\mapsto f w_0.$$ 
Using this isomorphism, we transform the right $\mathcal D_H^{(m)}$ action on $\omega_H$ to $\mathcal O_H$, that is, 
for any section $f$ of $\mathcal O_H$ and any section $P$ of $\mathcal D_H^{(m)}$, we define $fP$ to be the section of 
$\mathcal O_H$ so that 
$$(fP)\omega_0=(f\omega_0)P.$$ 
For any $\xi\in\mathfrak L(H)$, the Lie derivative of $\omega_0$ with respect to $L_\xi$ vanishes 
by the left invariance of $\omega_0$. This implies that 
$$fL_\xi= -L_\xi(f).$$ 
In the next section, we use the following variant of Proposition \ref{beta0}.

\begin{proposition}\label{beta} 
Let $M$ be a left $R[H]$-comodule with coaction
$$a: M\to R[H]\otimes_RM.$$
Let $\beta'$ be the $R$-linear map
$$\beta': \Gamma(H, {\mathcal D}_{H}^{(m)})\to \m{End}_R (M)$$ so that for any $P\in \Gamma(H, {\mathcal D}_{H}^{(m)})$, $\beta'(P)$ is 
the composite
$$\beta'(P):M\xrightarrow{a}  R[H]\otimes_RM\xrightarrow{\cdot P\otimes\mathrm{id}_M} 
R[H]\otimes_RM  \xrightarrow{\m{ev_1}\otimes\mathrm{id}_M}M,$$ 
where $\cdot P$ is the right action of $P$ on $R[H]$ defined above. 
For any $P\in\Gamma(H, {\mathcal D}_{H}^{(m)})$ and $\xi\in \mathfrak L(H)$, we have
$$\beta'(L_{\xi}P)=\beta'(P)\beta'( L_{\xi}).$$ In particular, for any $\zeta_1,\cdots,\zeta_n\in\f{L}(H)$, we have 
$$\beta'(L_{\zeta_1}\cdots L_{\zeta_n})=\beta'(L_{\zeta_s})\cdots \beta'(L_{\zeta_1}).$$
\end{proposition}

\section{Calculation on the modified hypergeometric $\mathcal D$-module}

\subsection{Homogeneization}\label{Homogeneization}
Let $\mathbb G_m:=\mathrm{Spec}\,R[t, t^{-1}]$ be the multiplicative group $R$-scheme, $\rho'_j$ $(j=1,\ldots, N)$
the representations 
$$\rho'_j: \mathbb G_m\times G\to \mathrm{GL}(V_j), \quad (t, g)\mapsto t\rho_j(g),$$
$\rho'_0$ the representation
$$\rho'_0:  \mathbb G_m\times G\to \mathrm{GL}(R),\quad (t, g)\mapsto t,$$
$\mathbb V':=\mathbb A^1\times  \mathbb V=\prod_{j=0}^N \mathrm{End}(V_j)$, $\iota'$ (resp. $\iota$) the morphism
\begin{eqnarray*}
\iota': \mathbb G_m\times G\to\mathbb V', && (t, g)\mapsto (\rho'_0(t, g),\rho'_1(t, g),\ldots, \rho'_N(t, g)),\\
(\hbox{resp. } \iota:  G\to\mathbb V, && g\mapsto (\rho_1(g),\ldots, \rho_N(g)),
\end{eqnarray*}
and $X'$ (resp. $X$) the scheme theoretic image of $\iota'$ (resp. $\iota$).
We assume $\iota: G\to \mathbb V$ is quasi-finite. Then so is $\iota'$. Let $Y'$ (resp. $Y$) be the integral closure of $X'$
(resp. $X$) in the function field of $\mathbb G_m\times G$ (resp. $G$). Let $H'=(\mathbb G_m\times G)\times (\mathbb G_m\times G)$
(resp. $H=G\times G$). Then $H'$ (resp. $H$) acts on $X'$ and $Y'$ (resp. $X$ and $Y$). 

\begin{proposition}\label{fiberat1}  ${}$

\begin{enumerate}[(i)]
\item Suppose $Z$ is an $R$-scheme
equipped with a group $R$-scheme $G$-action and an equivariant morphism 
$\pi: Z\to G$. Let $1:\mathrm{Spec} \,R \to G$ be the unit section. 
Define $Z_1$ by the Cartesian diagram
$$\begin{tikzcd}
Z_1\arrow[r,"i_1"]\arrow[d]& Z\arrow[d,"\pi"]\\
\mathrm{Spec}\, R\arrow[r,"1"]&G.
\end{tikzcd}$$
Then we have an isomorphism 
$$G\times Z_1\stackrel\cong\to Z,\quad (g,z)\mapsto gz.$$ 
\item Let $i_1$ be the closed immersion 
$$i_1: \mathbb V\to \mathbb V',\quad v\mapsto (1, v).$$
We have Cartesian diagrams
$$\begin{tikzcd}
 Y\arrow[d,"i_{1,Y'}"]\arrow[r]&X\arrow[d,"i_{1,X'}"]\arrow[r]&\mathbb V\arrow[d,"i_1"]\\
Y'\arrow[r]& X'\arrow[r] &\mathbb V'.
\end{tikzcd}$$
\item Let $\mathbb P:=\mathbb P(\mathbb A^1\times \mathbb V)$. Embed 
$\mathbb V'=\mathbb A^1\times \mathbb V$ into $\mathbb P^1\times \mathbb P$ via 
$$(t, v)\mapsto ([1:t], [1:v]).$$
Let $\overline X'$ (resp. $\overline X$) be the scheme theoretic image of the composite
$$\mathbb G_m\times G\stackrel{\iota'}\to  \mathbb A^1\times \mathbb V\to \mathbb P^1\times \mathbb P\quad 
(\hbox{resp. } G\stackrel{\iota}\to \mathbb V\to \mathbb P),$$ 
let $\overline Y'$ (resp. $\overline Y$) be the integral closure of $\overline X'$
(resp. $\overline X$) in the function field of $\mathbb G_m\times G$ (resp. $G$), and let 
$i_1$ be the closed immersion 
$$i_1: \mathbb P\to \mathbb P^1\times\mathbb P,\quad [v]\mapsto ([1:1], [v]).$$
We have Cartesian diagrams
$$\begin{tikzcd}
\overline Y\arrow[d,"i_{1,\overline Y'}"]\arrow[r]&\overline X\arrow[d,"i_{1,\overline X'}"]\arrow[r]&\mathbb P\arrow[d, "i_1"]\\
\overline Y'\arrow[r]& \overline X'\arrow[r]& \mathbb P^1\times\mathbb P.
\end{tikzcd}$$
\end{enumerate}
\end{proposition} 

\begin{proof} (i) Consider the commutative diagram of Cartesian squares
$$\begin{tikzcd}
G\times Z_1\arrow[r,"\mathrm{id}_G\times i_1"]\arrow[d]&G\times Z\arrow[r,"\cong"]\arrow[d]&
G\times Z\arrow[r,"p_2"]\arrow[d]&Z\arrow[d]\\
G\times \mathrm{Spec}\,R\arrow[r,"\mathrm{id}\times 1"]&G\times G\arrow[r,"\cong"]&
G\times G\arrow[r,"p_2"]&G,
\end{tikzcd}$$
where the isomorphisms in the middle are given by 
$$(g, z)\mapsto (g, gz).$$ 
The composite of the morphisms in the bottom line is an isomorphism. So is
the composite of the morphisms on the top line. Thus $G\times Z_1\cong Z$. 

(ii) We have $\mathbb G_m$-equivariant morphisms
$$\mathbb G_m\times G\hookrightarrow Y'\to X'\to \mathbb V'=\mathbb A^1\times \mathbb V\stackrel{p_1}\to \mathbb A^1.$$
Regarding $\mathbb G_m$ as an open subscheme of $\mathbb A^1$ and let 
$\mathring X'$, $\mathring Y'$ be the inverse images of $\mathbb G_m$ in $X'$, $Y'$, respectively. Then we have
$\mathbb G_m$-equivariant morphisms
$$\mathbb G_m\times G\hookrightarrow \mathring Y'\to \mathring X'\to \mathbb G_m\times \mathbb V\stackrel{p_1}\to\mathbb G_m.$$
Moreover, $\mathring X'$ is the scheme theoretic image of $\mathbb G_m\times G$ in $\mathbb G_m\times \mathbb V$,
and $\mathring Y'$ is the integral closure of $\mathring X'$ in the function field of $\mathbb G_m\times G$. 
By (i), we have $$\mathring X'\cong\mathbb G_m\times X'_1, \quad \mathring Y'\cong\mathbb G_m\times Y'_1.$$
This implies that $X'_1$ is integral, $Y'_1$ is integral normal, $Y'_1\to X'_1$ is a finite morphism,
$G\to Y'_1$ is an open immersion, and $\Gamma(X'_1,\mathcal O_{X'_1})\to \Gamma(G,\mathcal O_G)$ is 
injective. So $X'_1$ is the scheme theoretic image of $G\to\mathbb V$ and 
$Y'_1$ is the integral closure of $X'_1$ in $G$. We thus have $X'_1\cong X$ and $Y'_1\cong Y$. 

(iii) The proof is similar to (ii). We have $\mathbb G_m$-equivariant morphisms
$$\mathbb G_m\times G\hookrightarrow \overline Y'\to \overline X'\to \mathbb P^1\times \mathbb P\stackrel{p_1}\to \mathbb P^1.$$
Take the inverse images of the open subscheme $\mathbb G_m$ of $\mathbb P^1$ and apply (i). 
\end{proof}

\begin{proposition}\label{ass2}  There exists a morphism
$\sigma': \tilde Y'\to Y'$ such that the following conditions hold:
\begin{enumerate}[(1)]
\item $\tilde Y'$ is a smooth $R$-scheme with $(\mathbb G_m\times H)$-action containing 
$\mathbb G_m\times G$ as an open subscheme, and $\tilde Y'\to \mathrm{Spec}\,R$ has geometrically
connected fibers.
\item $\sigma'$ is equivariant, proper, and induces identity on $\mathbb G_m\times G$.
\item Let $i_{1,Y'}: Y\to Y'$ be the closed immersion in Proposition \ref{fiberat1} (ii). 
There exists a morphism $\sigma: \tilde Y\to \overline Y$ with the properties 
in Proposition \ref{ass1} such that we have a Cartesian diagram 
$$\begin{tikzcd}
\sigma^{-1}(Y)\arrow[r]\arrow[d,"\sigma"]&\tilde Y'\arrow[d, "\sigma'"]\\
Y\arrow[r,"i_{1,Y'}"]&Y'.
\end{tikzcd}$$
\end{enumerate}
\end{proposition}

\begin{proof} Keep the notation in Proposition \ref{fiberat1} (ii). We have an action of the group 
scheme $H'=(\mathbb G_m\times G)\times
(\mathbb G_m\times G)$ on 
$$\mathbb P^1\times \mathbb P=\mathbb P^1\times \mathbb P\Big(\mathbb A^1\times \prod_{j=1}^N \mathrm{End}(V_j)\Big)$$
given by 
\begin{eqnarray*}
&&\Big((t_1, g_1, t_2,g_2), ([q:r], [s:A_1:\cdots: A_N])\Big)\in H'\times (\mathbb P^1\times\mathbb P)\\
&\mapsto& 
\Big([q:t_1rt_2^{-1}], \Big[s:t_1t_2^{-1}\rho_1(g_1)A_1\rho_1(g_2)^{-1}:\cdots: t_1t_2^{-1}\rho_1(g_N)A_N\rho_N(g_2)^{-1}\Big]\Big)
\in \mathbb P^1\times\mathbb P.
\end{eqnarray*}
This action induces actions of $H'$ on $\overline X'$ and on $\overline Y'$, and $\overline Y'$ is an equivariant embedding 
of the group $G':=\mathbb G_m\times G$.  
By \cite[6.2.5]{Frob-split}, we may choose a $G'_K$-equivariant proper morphism 
$\bar\sigma'_K:\widetilde{\overline Y'}_K\to \overline Y'_K$ such that
$\widetilde{\overline Y'}_K$ is a smooth toroidal equivariant embedding of $G'_K$ in the sense of \cite[6.2.2]{Frob-split}.
By \cite[6.2.3]{Frob-split}, there exists an open subscheme $\widetilde{\overline Y'}^\circ_K$ of $\widetilde{\overline Y'}_K$
stable under the action of $(\mathbb G_{m,K}\times_K B_K^+)\times_K (\mathbb G_{m,K}\times_K B_K^-)$, and 
a closed subvariety $S'_K$ of $\widetilde{\overline Y'}^\circ_K$
stable under the action of $(\mathbb G_{m,K}\times_K T_K)\times_K(\mathbb G_{m,K}\times_K T_K)$ 
such that $S'_K$ is a toric variety containing $\mathbb G_{m,K}\times_K T_K$ as an open dense torus,
$$H'_K \widetilde{\overline Y'}^\circ_K=\widetilde{\overline Y'}_K, \quad 
U^+_K(\mathbb G_{m,K}\times_K T)U^-_K=G'_K\cap \widetilde{\overline Y'}^\circ_K,$$ and we have an
isomorphism $$U^+_K\times_K 
U^-_K \times_K S'_K \stackrel\cong \to\widetilde{\overline Y'}^\circ_K, \quad (g,y)\mapsto gy.$$
By \cite[5.4]{Lishang} applied to the fan $\Sigma'$ defining $S'_K$, the above data can be defined over $R$. 
More precisely,  we have an $R$-scheme $\widetilde{\overline Y'}$ equipped 
with an action of $(\mathbb G_m\times G)\times (\mathbb G_m\times G)$ and containing 
$\mathbb G_m\times G$ as an open dense subscheme such that 
the generic geometric fiber of $\widetilde{\overline Y'}$ is $\widetilde{\overline Y'}_K$, and the special fiber 
$\widetilde{\overline Y'}_k$
is the toroidal embedding of $\mathbb G_{m,k}\times_k G_k$ associated to the fan $\Sigma'$ (\cite[6.2.4]{Frob-split}). 
Let $S'$ be the toric $R$-scheme with the fan $\Sigma'$. There exists an open subscheme 
$\widetilde{\overline Y'}^\circ$ of $\widetilde{\overline Y'}$ stable under the action of 
$(\mathbb G_m\times  B^+)\times (\mathbb G_{m}\times B^-)$ such that we have 
\begin{eqnarray}\label{eqn:cond}
\begin{array}{l}
U^+\times U^-\times S'\stackrel\cong\to \widetilde{\overline Y'}^\circ,\\
H'\widetilde{\overline Y'}^\circ=\widetilde{\overline Y'}, \quad U^+\times U^-\times 
(\mathbb G_m\times T)\cong (\mathbb G_m\times G)\cap \widetilde{\overline Y'}^\circ.
\end{array}
\end{eqnarray}
Since $\widetilde{\overline Y'}_K$ is smooth, $\Sigma'$ is a regular fan. So $S'$ is smooth over $R$. This 
implies that $\widetilde{\overline Y'}^\circ$ is smooth. Since $\widetilde{\overline Y'}$ is a translation of 
$\widetilde{\overline Y'}^\circ$ by $H'$,  $\widetilde{\overline Y'}$ is also smooth. Since $\overline Y'$ is normal, the morphism 
$\bar\sigma'_K: \widetilde{\overline Y'}_K\to \overline Y'_K$ 
can be uniquely extended to an $R$-morphism $\bar\sigma': \widetilde{\overline Y'}\to \overline Y'$, 
which is necessarily $H'$-equivariant. It is proper since both $\widetilde{\overline Y'}$ and $\overline Y'$ are proper over $R$. 
By definition, $X'$ is an open subscheme of $\overline X'$. So $Y'$ is an open subscheme of $\overline Y'$. 
Let $\tilde Y'=\bar\sigma'^{-1}(Y')$, and let $\sigma': \tilde Y'\to Y'$ be the morphism induced by $\bar \sigma'$. 
Then the conditions (1) and (2) hold. Define $\sigma: \tilde Y\to \overline Y$ to be the base change of $\bar\sigma': 
\widetilde{\overline Y'}\to \overline Y'$ by $i_{1, \overline Y'}: \overline Y\to \overline Y'$:
$$\begin{tikzcd}
\tilde Y\arrow[r]\arrow[d, "\sigma"]&\widetilde{\overline Y'}\arrow[d, "\bar\sigma'"]\\
\overline Y\arrow[r,"i_{1, \overline Y'}"]&\overline Y'.
\end{tikzcd}$$
By taking base change by $i_1$ of the condition (\ref{eqn:cond}), one can check 
$\sigma: \tilde Y\to \overline Y$ has the properties in Proposition \ref{ass1}. Over the open subset $Y'$ over 
$\overline Y'$, the above Cartesian diagram is that in (3). 
\end{proof}

\begin{proposition}\label{prop:omega}
Let $\omega_{Y'}=\sigma'_*\omega_{\tilde Y'}$, $\omega_{Y'_k}=\sigma'_{k*}\omega_{\tilde Y'_k}$,
$\omega_{Y}=\sigma_*\omega_{\tilde Y}|_{Y}$, and 
$\omega_{Y_k}=\sigma_{k*}\omega_{\tilde Y_k}|_{Y_k}$. Then we have
\begin{align}\label{aln:homomega}
i_{1,Y'}^*\omega_{Y'}\cong \omega_Y,\quad i_{1,Y'_k}^*\omega_{Y'_k}\cong \omega_{Y_k},
\end{align}
\end{proposition}

\begin{proof} We prove the first isomorphism. As in the proof of Proposition \ref{fiberat1}, we have a Cartesian diagram 
$$\begin{tikzcd}
\mathbb G_m\times \sigma^{-1}(Y)\arrow[r,hook]\arrow[d,"\mathrm{id}\times \sigma"]&\tilde Y'\arrow[d,"\sigma'"]\\
\mathbb G_m\times Y\arrow[r,hook,"A"]&Y',
\end{tikzcd}$$
where the horizontal arrows are given by $(t, y)\mapsto ty$ and are $(\mathbb G_m\times H)$-equivariant open immersions.
It follows that 
 $$A^*\omega_{Y'}= A^* \sigma'_*\omega_{\tilde Y'}\cong \omega_{\mathbb G_m}\boxtimes \sigma_*\omega_{\tilde Y}|_Y
 = \omega_{\mathbb G_m}\boxtimes \omega_Y.$$
 Restricting to $1\times Y$, we get $i_{1,Y'}^*\omega_{Y'}\cong \omega_Y$.
\end{proof}

Let $A$ be a $k$-algebra provided with a filtration 
$$A_0\subset A_1\subset\cdots$$
by subgroups such that 
\begin{eqnarray}\label{eqn:filteredalg}
k\subset A_0, \quad A_iA_j\subset A_{i+j},\quad A=\bigcup_i A_i.
\end{eqnarray}
Let 
$$C(A)=\bigoplus_{i=0}^\infty A_i t^i.$$ The property $A_iA_j\subset A_{i+j}$ implies that 
$C(A)$ is a graded $k[t]$-algebra. Multiplication by $t$ is homogeneous of degree $1$ and is injective. 
So $C(A)$ has no $t$-torsion. As a graded $k[t]$-algebra, this implies that $C(A)$ has no torsion and hence flat over $k[t]$.  
If $A=\bigoplus_{i=0}^\infty A^{(i)}$ is a graded $k$-algebra provided with the filtration defined by $$A_i=\sum_{j\leq i} A^{(i)},$$
then we have an isomorphism $C(A)\cong A[t]$ given by 
$$\bigoplus_{i=0}^\infty t^i A_i \stackrel\cong \to A[t], \quad t^i \sum_{j\leq i} f_j\mapsto \sum_{j\leq i} t^{i-j}f_j$$
for any $f_j\in A^{(j)}$. This is an isomorphism of graded $k$-algebras if $A[t]$ is provided with the grading
$$A[t]=\bigoplus_{d=0}^\infty \Big(\sum_{i+j=d} t^i A^{(j)}\Big).$$
If $A=k[x_1, \ldots, x_n]$ is the polynomial ring with the grading 
defined by the degree, then the isomorphism from $C(A)=\bigoplus x_0^i A_i$ to $A[x_0]=k[x_0,x_1, \ldots, x_n]$
is the usual map
$$f(x_1, \ldots, x_n)\mapsto x_0^i f\Big(\frac{x_1}{x_0}, \ldots, \frac{x_n}{x_0}\Big)$$
for any polynomial $f$ of degree $\leq i$.

For any $A$-module $M$ provided with a filtration 
$$\cdots\subset M_j\subset M_{j+1}\subset \cdots$$ by subgroups such that
\begin{eqnarray}\label{eqn:filtration}
M=\bigcup_j M_j, \quad A_iM_j\subset M_{i+j}.
\end{eqnarray}
Let
$$C(M)=\bigoplus_{j=-\infty}^{\infty}  t^jM_j.$$
Then $C(M)$ is a graded $C(A)$-module, and multiplication by $t$ on $C(M)$ is injective. 
This implies that $C(M)$ is flat over $k[t]$. 

\begin{proposition}\label{prop:filgr} ${}$

\begin{enumerate}[(i)]
\item The functor $A\mapsto C(A)$ is an 
equivalence from the category of filtered $k$-algebras satisfying the condition (\ref{eqn:filteredalg}) to the category 
of graded $k[t]$-algebras on which multiplication by $t$ is injective and homogeneous of degree $1$. We have 
$$A\cong C(A)/(t-1)C(A),\quad \mathrm{Gr}(A)\cong C(A)/tC(A).$$

\item The functor $M\mapsto C(M)$ is an 
equivalence from the category of filtered $A$-modules satisfying the condition (\ref{eqn:filtration}) to the category 
of graded $C(A)$-modules on which multiplication by $t$ is injective. We have 
$$M\cong C(M)/(t-1)C(M),\quad \mathrm{Gr}(M)\cong C(M)/tC(M).$$
If $C(M)$ is finitely generated over $C(A)$, then
$\mathrm{Gr}(M)$ is finitely generated over $\mathrm{Gr}(A)$.

\item Suppose $A=\bigoplus_{i=0}^\infty A^{(i)}$ is a graded $k$-algebra.
We have an equivalence of categories $M\mapsto C(M)$
from the category of filtered $A$-modules satisfying the condition (\ref{eqn:filtration}) to the category 
of graded $A[t]$-modules on which multiplication by $t$ is injective. We have 
$$M\cong C(M)/(t-1)C(M),\quad \mathrm{Gr}(M)\cong C(M)/tC(M).$$
If $C(M)$ is finitely generated over $A[t]$, then
$\mathrm{Gr}(M)$ is finitely generated over $\mathrm{Gr}(A)\cong A$.
\end{enumerate}
\end{proposition}

\begin{proof}  (i) 
Let $C=\bigoplus_{i=0}^\infty C_i$ be a graded $k[t]$-algebra such that multiplication by $t$ is injective of degree $1$. 
Multiplication by $t$ defines a direct system
$$C_0\stackrel t\hookrightarrow C_1\stackrel t\hookrightarrow \cdots.$$ 
Let $A=\varinjlim_i C_i$, and let $A_i$ be the image of $C_i$ in $A$. 
For any $f, g\in A$, choose nonnegative integers $i, j$ such that $f\in A_i$ and 
$g\in A_j$. Let $f'\in C_i$ (resp. $g'\in C_j$) be the (unique) preimage for $f$ (resp. $g$). 
We define $fg\in A$ to be the image of $f'g'\in C_{i+j}$ in $A$. It is independent of the choices of $i$ and $j$. 
We thus get a filtered $k$-algebra $A$ with the property $C(A)\cong C$. The maps
\begin{eqnarray*}
\bigoplus_{i=0}^{\infty} t^i A_i\to \bigoplus_{i=0}^{\infty} A_i/A_{i-1}, && \sum_i t^i f_i \mapsto (f_i + A_{i-1}),\\
\bigoplus_{i=0}^{\infty} t^i A_i\to A, && \sum_i t^i f_i\mapsto \sum_i f_i
\end{eqnarray*}
induce isomorphisms
$$C(A)/tC(A)\cong \mathrm{Gr}(A),\quad
C(A)/(t-1)C(A)\stackrel\cong\to A.$$
Indeed, if $\sum_i f_i=0$, then $\sum_i t^i f_i=(1-t)(\sum_i t^ig_i)$, where 
$g_i=\sum_{j=0}^i f_j$.

(ii) is proved in a similar way. 
(iii) follows from (ii).
\end{proof}

\subsection{}
Let $(x_1, \ldots, x_n)$ be a linear coordinate on $\mathbb V$, let $(x'_1,\ldots, x'_n)$ be the dual coordinate for the dual space
$\mathbb V^*$ of $\mathbb V$, and let
$$A^{(m)}=R[\partial_{x'_1}^{\langle p^j\rangle_{(m)}},\ldots,\partial_{x'_n}^{\langle p^j\rangle_{(m)}} ]_{0\leq j\leq m}.$$
By \cite[(2.2.5.1)]{Berthelot1}, $A^{(m)}$ can be regarded as an $R$-subalgebra of 
$D_{\mathbb V^*}^{(m)}:=\Gamma(\mathbb V^*,\mathcal D_{\b V^*}^{(m)})$. The action of $H$ on 
$D_{\mathbb V^*}^{(m)}$ induces an action of $H$ on $A^{(m)}$. We let $H$ act on the right 
of $A^{(m)}$ so that it corresponds to a left comodule structure on $A^{(m)}$ over $R[H]$. 
Recall that $\pi^{p-1}=-p$. Since 
\begin{align*}
&p^j! \partial_{x_i}^{\langle p^j\rangle_{(m)}}= \partial^{p^j}_{x_i}\quad (0\leq j\leq m),\\
&\mathrm{ord}_\pi(p^j!)=\frac{p^j-1}{p-1}\cdot \mathrm{ord}_\pi(p)=p^j-1, 
\end{align*}
we have 
$$A^{(m)}=R[\pi(\partial_{x'_1}/\pi)^{p^j},\ldots,\pi(\partial_{x'_n}/\pi)^{p^j} ]_{0\leq j\leq m}.$$
Let 
$$B^{(m)}=R[(\pi x_1)^{p^j}/p^j!,\ldots,(\pi x_n)^{p^j}/p^j! ]_{0\leq j\leq m}=R[\pi x_1 ^{p^j},\ldots,\pi x_n^{p^j} ]_{0\leq j\leq m}.$$
The Fourier transformation isomorphism (\ref{eqn:Fourieriso}) induces an isomorphism 
\begin{eqnarray}\label{AB}
A^{(m)}\stackrel\cong\to B^{(m)}, \quad \partial_{x'_i}\mapsto \pi x_i.
\end{eqnarray}
Let $H$ act on the right of $\mathbb V$ so that $R[x_1, \ldots, x_n]$ is equipped with a left $R[H]$-comodule structure, and
$B^{(m)}$ is a sub-comodule of $R[x_1, \ldots, x_n]$. 
We have $$B^{(m)}_\mathbb Q:=B^{(m)}\otimes_{\mathbb Z}\mathbb Q\cong K[x_1, \ldots, x_n].$$ 
We regard $\mathrm{Spec}\, B^{(m)}$ 
as an integral model of $\mathbb V_K$.

Let $f': Y'\to \mathbb V'$ be the composite $Y'\to X'\to \mathbb V'$, let 
$\omega_{Y'}=\sigma'_*\omega_{\tilde Y'}$, and let 
$$M'^{(m)}_{\mathbb Q}=\Gamma(\mathbb V'_K, f'_*\omega_{Y'}).$$ 
Identify $\mathbb V'_K$ with $\mathrm{Spec}\, (K[t]\otimes_KB^{(m)}_{\Q})$.
Then $M'^{(m)}_{\mathbb Q}$ is a finitely generated $(K[t]\otimes_KB^{(m)}_{\Q})$-module with 
a right action by $\mathbb G_{m, K}\times H_K$ so that for any section $\omega$ of 
$\Gamma(\mathbb V'_K, f'_*\omega_{Y'})$ and any $K$-point $g$ 
of $\mathbb G_{m, K}\times H_K$, $\omega g$ is the pulling back of $\omega$ by the automorphism 
of $\mathbb V'_K$ defined by the action of $g$. This action endows $M'^{(m)}_{\mathbb Q}$ with 
a left $(R[\mathbb G_m]\otimes_RR[H])$-comodule structure
$$M'^{(m)}_{\mathbb Q}\to(K[\mathbb G_{m,K}]\otimes_KK[H])\otimes_K  M'^{(m)}_{\mathbb Q}
\cong  (R[\mathbb G_m]\otimes_RR[H])\otimes_R  M'^{(m)}_{\mathbb Q}$$ by \cite[3.2 a)-b)]{Serre}. 
(In \cite[3.2 b)]{Serre}, to get a left comodule, the linear action of $G(A')$ on 
$A'\otimes E$ must be a right action.)
By \cite[Proposition 2]{Serre}, there exists a sub-comodule $I$ of $M'^{(m)}_{\Q}$ which is a finitely generated $R$-module and contains
a finite family of generators of $M'^{(m)}_{\mathbb Q}$ as a module over $K[t]\otimes_KB^{(m)}_{\Q}$. Let $M'^{(m)}$ be the 
$(R[t]\otimes_R B^{(m)})$-submodule of $M'^{(m)}_{\Q}$ generated by $I$. Then we have 
$$M'^{(m)}\otimes_{\b{Z}}{\Q}\cong M'^{(m)}_{\mathbb Q}.$$
We claim $M'^{(m)}$ is an $(R[\mathbb G_m]\otimes_RR[H])$-subcomodule of $M'^{(m)}_{\mathbb Q}$. Write 
$R[H']=R[\mathbb G_m]\otimes_RR[H]$, $S=R[t]\otimes_R B^{(m)}$, $c$ the coproducts on $S$, on $M'^{(m)}_{\mathbb Q}$ and on 
$I$, $m: R[H']\otimes_R R[H']\to R[H']$ the multiplication, and $\mu: S\otimes_RM'^{(m)}_{\mathbb Q}\to M'^{(m)}_{\mathbb Q}$ the 
scalar multiplication. Then $\mu$ is compatible with the the group action and hence 
with the comodule structure. So the square on the right of
the  following diagram commutes:
$$\begin{tikzcd}
S\otimes_R I\arrow[d,"c\otimes c"]\arrow[r,hook]& S\otimes_R M'^{(m)}_{\Q}\arrow[d,"c\otimes c"]\arrow[r, "\mu"]& 
M'^{(m)}_{\Q}\arrow[ddd,"c"]\\
(R[H']\otimes_R S)\otimes_R (R[H']\otimes_R I)
\arrow[d,"\cong"]\arrow[r, hook]&(R[H']\otimes_R S)\otimes_R (R[H']\otimes_R M'^{(m)}_{\mathbb Q})\arrow[d,"\cong"]&\\
(R[H']\otimes_RR[H'])\otimes_R (S\otimes_RI)\arrow[d,"{m\otimes \mathrm{id}}"]\arrow[r,hook] &
(R[H']\otimes_RR[H'])\otimes_R (S\otimes_R M'^{(m)}_{\mathbb Q})\arrow[d,"{m\otimes \mathrm{id}}"]&\\
R[H']\otimes_R (S\otimes_R I)\arrow[r,hook]& 
R[H']\otimes_R (S\otimes_R M'^{(m)}_{\mathbb Q})\arrow[r,"\mathrm{id}\otimes\mu"] &R[H']\otimes_RM'^{(m)}_{\Q}
\end{tikzcd}$$
The left part of the diagram clearly commutes. 
The commutativity of the outer loop shows that $c: M'^{(m)}_{\Q}\to R[H']\otimes_RM'^{(m)}_{\Q}$
maps $M'^{(m)}$ to $R[H']\otimes_RM'^{(m)}$. This proves our claim. 
Let $$M^{(m)}=M'^{(m)}\otimes_{R[t]}R[t]/(t-1).$$ Then $M^{(m)}$ is a left $R[H]$-comodule and a finitely generated $B^{(m)}$-module. 
For any abelian group $M$, set 
$$M_{\mathbb Q}:=M\otimes_{\mathbb Z}\mathbb Q.$$
By (\ref{aln:homomega}), we have 
$$M^{(m)}_{\mathbb Q}\cong \Gamma(\mathbb V_K, f_{K*}\omega_{Y_K}).$$ 

\subsection{}\label{affintmodel}
Denote the base change from $R$ to $k$ of an object over $R$ by the same notation with a subscript $k$. Then 
$M'^{(m)}_k:=M'^{(m)}\otimes_R k$ (resp. $k[t]\otimes_kB^{(m)}_k$) is a finite $(k[t]\otimes_kB^{(m)}_k)$-module 
(resp. a $k$-algebra) with a $(\mathbb G_{m,k}\times_k H_k)$-action. 
The $\mathbb G_{m,k}$-action endows $M'^{(m)}_k$ (resp. $k[t]\otimes_kB^{(m)}_k$) with 
a graded module (resp. a graded algebra) structure. Let $\bar M'^{(m)}_k$ be the quotient 
of $M'^{(m)}_k$ by the maximal $t$-torsion submodule. Then the canonical homomorphism
$M'^{(m)}_k\twoheadrightarrow \bar M'^{(m)}_k$ induces an isomorphism
$$M^{(m)}_k\cong M'^{(m)}_k\otimes_{k[t]} k[t]/(t-1)\stackrel\cong\to \bar M'^{(m)}_k\otimes_{k[t]} k[t]/(t-1).$$
The $(\mathbb G_{m,k}\times_k H_k)$-action on $ M'^{(m)}_k$ induces a $(\mathbb G_{m,k}\times_kH_k)$-action on 
$\bar M'^{(m)}_k$. By Proposition \ref{prop:filgr} (iii), the $H_k$-invariant graded module structure on $\bar M'^{(m)}_k$ 
defines an $H_k$-invariant good filtration on 
$M^{(m)}_k$ such that 
$$\mathrm{Gr}(M^{(m)}_k)\cong \bar M'^{(m)}_{k} \otimes_{k[t]}k[t]/(t).$$ 
In particular, $\mathrm{Gr}(M^{(m)}_k)$ is a quotient of $M'^{(m)}_{k} \otimes_{k[t]}k[t]/(t)$.

\subsection{}\label{S0=omega}
We construct an explicit $M'^{(m)}$ for the level $m=0$ case, which will be used later. We have 
$$A^{(0)}=R[\partial_{x'_1},\ldots, \partial_{x'_n}], \quad B^{(0)}=R[\pi x_1,\ldots, \pi x_n]$$
Choose $M'^{(0)}$ as follows. We have a homomorphism
$$\phi: R[t,x_1, \ldots, x_n ]\to R[t,\pi x_1,\ldots, \pi x_n],\quad t\mapsto  \pi t, \quad x_j\mapsto \pi x_j.$$
We define 
$$M'^{(0)}:=\Gamma(\mathbb A^1\times \b{V}, f'_*\omega_{Y'})\otimes_{R[t,x_1, \ldots, x_n ],\phi}
R[t,\pi x_1,\ldots, \pi x_n].$$
Note that $\mathbb G_m\times H$ acts on $M'^{(0)}$, and hence $M'^{(0)}$ is an $(R[t,t^{-1}]\otimes_R R[H])$-comodule.
On the generic fiber, $\phi_K$ is the automorphism of $\mathbb A^1\times_K\mathbb V_K$ induced by the action of 
$\pi\in\mathbb G_{m,K}(K)$. Since $f'_{*}\omega_{Y',K}$ is $\mathbb G_{m,K}$-equivariant, we have
$M'^{(0)}\otimes_{\mathbb Z}\mathbb Q\cong M'^{(0)}_K.$ By Corollary \ref{cor:usedin4} in the appendix, we have 
\begin{align*}
M'^{(0)}_k\otimes_k k[t]/(t)& \cong \Gamma(\mathbb A_k^1\times_k \b{V}_k, f'_{k*}\omega_{Y'_k})\otimes_{k[t,x_1, \ldots, x_n ],\phi_k}
k[t,\pi x_1,\ldots, \pi x_n]\otimes_k k[t]/(t)\\
& \cong \Gamma(\mathbb A_k^1\times_k \b{V}_k, f'_{k*}\omega_{Y'_k})\otimes_{k[t,x_1, \ldots, x_n ],\phi_k} k[\pi x_1,\ldots, \pi x_n],
\end{align*}
where in the last equation, $\phi_k$ is given by 
$$\phi_k: k[t,x_1, \ldots, x_n ] \to k[\pi x_1, \ldots, \pi x_n ], \quad t\mapsto 0, \quad x_i\mapsto \pi x_i.$$
Let $M^{(0)}=M'^{(0)}\otimes_{R[t]}R[t]/(t-1)$. Then $M^{(0)}_k$ is provided with a good filtration so that $\mathrm{Gr}(M^{(0)}_k)$ is 
a quotient of $\Gamma(\mathbb A_k^1\times \b{V}, f'_{k*}\omega_{Y'_k})\otimes_{k[t,x_1, \ldots, x_n ],\phi_k} k[\pi x_1,\ldots, \pi x_n]$.

\subsection{}
Let $\widehat A^{(m)}$ and $\widehat B^{(m)}$ be the $\mathfrak m$-adic completion of $A^{(m)}$ and $B^{(m)}$, respectively. 
We have
$$\widehat A^{(m)}=R\langle \pi(\partial_{x'_1}/\pi)^{p^j},\ldots,\pi(\partial_{x'_n}/\pi)^{p^j} \rangle_{0\leq j\leq m},
\quad \widehat B^{(m)}=R\langle \pi x_1 ^{p^j},\ldots,\pi x_n^{p^j} \rangle_{0\leq j\leq m},$$
where $R\langle t_1, \ldots, t_n\rangle$ denote the ring of power series 
$\sum_{i_1, \ldots, i_n\geq 0}a_{i_1\ldots i_n}t_1^{i_1}\cdots t_n^{i_n}$
such that $a_{i_1\ldots i_n}\in R$ and $a_{i_1\ldots i_n}\to 0$ as $i_1+\cdots+i_n\to \infty$.
We have an $R$-module isomorphism   
$$\widehat{D}_{\widehat{\mathbb V}^*}^{(m)}:=\Gamma(\widehat{\b{V}}^*,\widehat{\mathcal D}_{\widehat{\mathbb V}^*}^{(m)})\cong 
R\langle x'_1,\ldots, x'_n\rangle \widehat{\otimes}_{R} \widehat{A}^{(m)}.$$
The rigid analytic space $\mathrm{Spm}\, \widehat{B}^{(m)}_{\Q}$ is the closed polydisc 
defined by $|x_i|\leq |\pi|^{-\frac{1}{p^m}}$ $(i=1,\ldots, n)$ in the analytification 
$\mathbb V_K^{\mathrm{an}}$ of $\mathbb V_K$.

Let $\widehat M'^{(m)}$ be the $\mathfrak m$-adic completion of $M'^{(m)}$. It is a finite
$(R \langle t\rangle\widehat\otimes_R\widehat{B}^{(m)})$-module and an $(R[{\b{G}}_m\times H])^\wedge$-comodule. 
Let $\widehat{M}^{(m)}$ be the completion of $M^{(m)}$. It is a finite $\widehat{B}^{(m)}$-module and 
an $(R[\widehat H])^\wedge$-comodule.
We have $\widehat M^{(m)}_k\cong {M}^{(m)}_k$. So $\widehat{M}^{(m)}_k$ is equipped with an $H_k$-invariant good filtration.

\begin{proposition}\label{rel-M-omega} We have an isomorphism 
$$\widehat{M}^{(m)}_{\mathbb Q}\cong \Gamma(\mathrm{Spm}\, \widehat{B}^{(m)}_{\Q}, (f_{K*}\omega_{Y_K})^{\mathrm{an}})$$
compatible with the action of $\widehat H_K$, where $(f_{K*}\omega_{Y_K})^{\mathrm{an}}$ is the analytification
of the $\mathcal O_{\mathbb V_K}$-module $f_{K*}\omega_{Y_K}$. 
\end{proposition}

\begin{proof} Note that $B^{(m)}$ (resp. $M^{(m)}$) is an $R$-model of $B_{\mathbb Q}^{(m)}$ (resp. $M^{(m)}_{\mathbb Q}$),
and $\mathrm{Spm}\, \widehat B^{(m)}_{\mathbb Q}$ is an open subset of $(\mathrm{Spec}\, B_{\mathbb Q}^{(m)})^{\mathrm{an}}\cong 
\mathbb V_K^{\mathrm{an}}$.
Let $\widehat M^{(m),\sim}_{\mathbb Q}$ be 
the $\mathcal O_{\mathrm{Spm}\, \widehat B^{(m)}_{\mathbb Q}}$-module  
associated to the $\widehat B^{(m)}_{\mathbb Q}$-module
$\widehat M^{(m)}_{\mathbb Q}$, and let $M^{(m),\sim}_{\mathbb Q}$ be the  
$\mathcal O_{\mathrm{Spec}\, B^{(m)}_{\mathbb Q}}$-module associated to the $B^{(m)}_{\mathbb Q}$-module
$M^{(m)}_{\mathbb Q}$.
Then we have
$$\widehat M^{(m),\sim}_{\mathbb Q}\cong (M^{(m),\sim}_{\mathbb Q})^{\mathrm{an}}|_{\mathrm{Spm}\, \widehat B^{(m)}_{\mathbb Q}}.$$
So we have
$$\widehat{M}^{(m)}_{\mathbb Q}\cong \Gamma(\mathrm{Spm}\, \widehat{B}^{(m)}_{\Q}, (M^{(m),\sim}_{\mathbb Q})^{\mathrm{an}}).$$
Since $M^{(m)}_{\mathbb Q}\cong\Gamma(\mathbb V_K, f_{K*}\omega_{Y_K})$, we have 
$M^{(m),\sim}_{\mathbb Q}\cong f_{K*}\omega_{Y_K}$. Our assertion follows.
\end{proof}

\begin{lemma}\label{lm:finitelevelN} Let $D^\d_{\b{\widehat P}, \Q}(^\dagger\infty)=
\Gamma(\mathbb P_k, {\mathcal D}^\d_{\b{\widehat P}, \Q}(^\dagger\infty)),$ 
and let $\bar\iota$ be the morphism in Proposition \ref{prop:inv0inf}. We have
\begin{align*}
(\bar\iota_*\omega_{\tilde Y})^\wedge_{\mathbb Q}\otimes_{\c{O}_{\widehat{\b{P}}, \Q}}
\mathcal {D}^\d_{\b{\widehat P}, \Q}(^\dagger \infty)&\cong \varinjlim_m \widehat M^{(m)}_{\mathbb Q}
\otimes_{\widehat B^{(m)}_{\mathbb Q}} {\mathcal D}^\d_{\b{\widehat P}, \Q}(^\dagger\infty), \\
\Gamma\Big(\mathbb P_k, (\bar\iota_*\omega_{\tilde Y})^\wedge_{\mathbb Q}\otimes_{\c{O}_{\widehat{\b{P}}, \Q}}
\mathcal {D}^\d_{\b{\widehat P}, \Q}(^\dagger \infty)\Big)&\cong \varinjlim_m \widehat M^{(m)}_{\mathbb Q}
\otimes_{\widehat B^{(m)}_{\mathbb Q}} D^\d_{\b{\widehat P}, \Q}(^\dagger\infty).
\end{align*}
\end{lemma}

\begin{proof} Let $j_m: \mathrm{Spm}\,B^{(m)}_{\mathbb Q}
\hookrightarrow \widehat{\mathbb P}_K$ be the open immersion. We have 
\begin{align}\label{al:passtolevel}
&\quad\; (\bar\iota_*\omega_{\tilde Y})^\wedge_{\mathbb Q}\otimes_{\c{O}_{\widehat{\b{P}}, \Q}}
\mathcal {D}^\d_{\b{\widehat P}, \Q}(^\dagger \infty)
\cong (\bar\iota_*\omega_{\tilde Y})^\wedge_{\mathbb Q}\otimes_{\c{O}_{\widehat{\b{P}}, \Q}}\mathcal O_{\widehat{\mathbb P},\mathbb Q}
(^\dagger\infty)\otimes_{\mathcal O_{\widehat{\mathbb P},\mathbb Q}
(^\dagger\infty)}\mathcal {D}^\d_{\b{\widehat P}, \Q}(^\dagger\infty)\nonumber\\
&\cong (\bar\iota_*\omega_{\tilde Y})^\wedge_{\mathbb Q}\otimes_{\c{O}_{\widehat{\b{P}}, \Q}}
 (\varinjlim_m \mathrm{sp}_* j_{m*}j_m^*
\mathcal O_{\widehat{\mathbb P}_K})\otimes_{\mathcal O_{\widehat{\mathbb P},\mathbb Q}
(^\dagger\infty)}\mathcal {D}^\d_{\b{\widehat P}, \Q}(^\dagger\infty)
\quad(\hbox{\cite[4.3.2]{Berthelot1}})\nonumber\\
&\cong \Big(\varinjlim_m 
\mathrm{sp}_*j_{m*}j_m^*(\bar\iota_{K*}\omega_{\tilde Y_K})^\mathrm{an}\Big)
\otimes_{\mathcal O_{\widehat{\mathbb P},\mathbb Q}
(^\dagger\infty)}\mathcal {D}^\d_{\b{\widehat P}, \Q}(^\dagger\infty)\quad(\hbox{\cite[2.1.3 (ii)]{B3}})\nonumber\\
&\cong \varinjlim_m \Big(
\mathrm{sp}_*j_{m*}j_m^*(\bar\iota_{K*}\omega_{\tilde Y_K})^\mathrm{an}
\otimes_{\mathrm{sp}_* j_{m*}j_m^*
\mathcal O_{\widehat{\mathbb P}_K}}\mathcal {D}^\d_{\b{\widehat P}, \Q}(^\dagger\infty)\Big)
\end{align}
By Proposition \ref{rel-M-omega}, we have
$$\Gamma(\widehat{\mathbb P}_K,j_{m*}j_m^*(\bar\iota_{K*}\omega_{\tilde Y_K})^\mathrm{an})\cong \widehat M^{(m)}_{\mathbb Q},
\quad \Gamma(\widehat{\mathbb P}_K, j_{m*}j_m^*
\mathcal O_{\widehat{\mathbb P}_K})\cong \widehat B^{(m)}_{\mathbb Q}.$$
Choose a free resolution 
\begin{align}\label{al:freeres}
(\widehat B^{(m)}_{\mathbb Q})^{\oplus s}\to (\widehat B^{(m)}_{\mathbb Q})^{\oplus t}\to \widehat M^{(m)}_{\mathbb Q}\to 0
\end{align}
for the $\widehat B^{(m)}_{\mathbb Q}$-module $\widehat M^{(m)}_{\mathbb Q}$. It gives rise to an exact sequence of 
${D}^\d_{\b{\widehat P}, \Q}(^\dagger\infty)$-modules
\begin{align}\label{al:exactsq}
({D}^\d_{\b{\widehat P}, \Q}(^\dagger\infty))^{\oplus s}\to ({D}^\d_{\b{\widehat P}, \Q}(^\dagger\infty))^{\oplus t}
\to \widehat M^{(m)}_{\mathbb Q}\otimes_{\widehat B^{(m)}_{\mathbb Q}}{D}^\d_{\b{\widehat P}, \Q}(^\dagger\infty)\to 0.
\end{align}
It also gives rise to exact sequences of sheaves 
\begin{align}\label{al:twoexsqsh}
&(\mathrm{sp}_* j_{m*}j_m^*
\mathcal O_{\widehat{\mathbb P}_K})^{\oplus s} \to
(\mathrm{sp}_* j_{m*}j_m^*
\mathcal O_{\widehat{\mathbb P}_K})^{\oplus t}\to
\mathrm{sp}_*j_{m*}j_m^*(\bar\iota_{K*}\omega_{\tilde Y_K})^\mathrm{an}\to 0,\nonumber\\
&(\mathcal {D}^\d_{\b{\widehat P}, \Q}(^\dagger\infty))^{\oplus s} \to
(\mathcal {D}^\d_{\b{\widehat P}, \Q}(^\dagger\infty))^{\oplus t}\to
\mathrm{sp}_*j_{m*}j_m^*(\bar\iota_{K*}\omega_{\tilde Y_K})^\mathrm{an}
\otimes_{\mathrm{sp}_* j_{m*}j_m^*
\mathcal O_{\widehat{\mathbb P}_K}}\mathcal {D}^\d_{\b{\widehat P}, \Q}(^\dagger\infty)\to 0.
\end{align}
By Proposition \ref{prop:Huyghe}, $\Gamma(\mathbb P_k^n,\hbox{-})$ is an exact functor on the category of coherent 
$\mathcal D^\dagger_{\widehat{\mathbb P},\mathbb Q}(^\dagger\infty)$-modules. Taking the global section of the 
exact sequence (\ref{al:twoexsqsh}), we get an exact sequence
$$({D}^\d_{\b{\widehat P}, \Q}(^\dagger\infty))^{\oplus s}\to ({D}^\d_{\b{\widehat P}, \Q}(^\dagger\infty))^{\oplus t}
\to \Gamma\Big(\mathbb P_k, 
\mathrm{sp}_*j_{m*}j_m^*(\bar\iota_{K*}\omega_{\tilde Y_K})^\mathrm{an}
\otimes_{\mathrm{sp}_* j_{m*}j_m^*
\mathcal O_{\widehat{\mathbb P}_K}}\mathcal {D}^\d_{\b{\widehat P}, \Q}(^\dagger\infty)\Big)\to 0,$$
Comparing with the exact sequence (\ref{al:exactsq}), we get 
$$\Gamma\Big(\mathbb P_k, 
\mathrm{sp}_*j_{m*}j_m^*(\bar\iota_{K*}\omega_{\tilde Y_K})^\mathrm{an}
\otimes_{\mathrm{sp}_* j_{m*}j_m^*
\mathcal O_{\widehat{\mathbb P}_K}}\mathcal {D}^\d_{\b{\widehat P}, \Q}(^\dagger\infty)\Big)
\cong \widehat M^{(m)}_{\mathbb Q}\otimes_{\widehat B^{(m)}_{\mathbb Q}}{D}^\d_{\b{\widehat P}, \Q}(^\dagger\infty).$$
Combined with (\ref{al:passtolevel}), we get
\begin{align*}
\Gamma(\mathbb P_k, (\bar\iota_*\omega_{\tilde Y})^\wedge_{\mathbb Q}\otimes_{\c{O}_{\widehat{\b{P}}, \Q}}
\mathcal {D}^\d_{\b{\widehat P}, \Q}(^\dagger \infty))
&\cong \varinjlim_m \Gamma\Big(\mathbb P_k,  
\mathrm{sp}_*j_{m*}j_m^*(\bar\iota_{K*}\omega_{\tilde Y_K})^\mathrm{an}
\otimes_{\mathrm{sp}_* j_{m*}j_m^*
\mathcal O_{\widehat{\mathbb P}_K}}\mathcal {D}^\d_{\b{\widehat P}, \Q}(^\dagger\infty)\Big)\\
& \cong \varinjlim_m \widehat M^{(m)}_{\mathbb Q}\otimes_{\widehat B^{(m)}_{\mathbb Q}} {D}^\d_{\b{\widehat P}, \Q}(^\dagger\infty)
\end{align*}
The free resolution (\ref{al:freeres}) also gives rise to an exact sequence of sheaves
\begin{align*}
({\mathcal D}^\d_{\b{\widehat P}, \Q}(^\dagger\infty))^{\oplus s}\to ({\mathcal D}^\d_{\b{\widehat P}, \Q}(^\dagger\infty))^{\oplus t}
\to \widehat M^{(m)}_{\mathbb Q}\otimes_{\widehat B^{(m)}_{\mathbb Q}}{\mathcal D}^\d_{\b{\widehat P}, \Q}(^\dagger\infty)\to 0.
\end{align*}
Comparing with the exact sequence (\ref{al:twoexsqsh}), we get 
$$\mathrm{sp}_*j_{m*}j_m^*(\bar\iota_{K*}\omega_{\tilde Y_K})^\mathrm{an}
\otimes_{\mathrm{sp}_* j_{m*}j_m^*
\mathcal O_{\widehat{\mathbb P}_K}}\mathcal {D}^\d_{\b{\widehat P}, \Q}(^\dagger\infty)
\cong \widehat M^{(m)}_{\mathbb Q}\otimes_{\widehat B^{(m)}_{\mathbb Q}}{\mathcal D}^\d_{\b{\widehat P}, \Q}(^\dagger\infty).$$
Our assertion follows.
\end{proof}

\subsection{}\label{subsection:Fourier} We describe the modified hypergeometric $\mathcal D$-module $\mathcal M=\f{F}_\pi(\c{N})$ 
as a coherent $\c{D}^\d_{\b{\widehat P}^*, \Q}(^\dagger \infty)$-module. By Lemma \ref{lm:finitelevelN}, we have
\begin{align*}
\Gamma(\mathbb P_k,\c{N})&=\Gamma\Big(\mathbb P_k, 
(\bar\iota_*\omega_{\tilde Y})^\wedge_{\mathbb Q}\otimes_{\c{O}_{\widehat{\b{P}}, \Q}}\mathcal {D}^\d_{\b{\widehat P}, \Q}(^\dagger \infty)
\Big/\sum_{\xi\in \f{L}(H)}L_{\xi}(\bar\iota_*\omega_{\tilde Y})^\wedge_{\mathbb Q}\otimes_{\c{O}_{\widehat{\b{P}}, \Q}}
\mathcal {D}^\d_{\b{\widehat P}, \Q}(^\dagger \infty)\Big)\\
&=\varinjlim_m \widehat M^{(m)}_{\mathbb Q}\otimes_{\widehat B^{(m)}_{\mathbb Q}} {D}^\d_{\b{\widehat P}, \Q}(^\dagger\infty)
\Big/\sum_{\xi\in \f{L}(H)}L_{\xi}(\widehat 
M^{(m)}_{\mathbb Q}\otimes_{\widehat B^{(m)}_{\mathbb Q}} {D}^\d_{\b{\widehat P}, \Q}(^\dagger\infty)).
\end{align*} 
Taking the Fourier transform, we get
$$\Gamma(\mathbb P^*_k, \f{F}_\pi(\c{N}))\cong 
\varinjlim_m \widehat M^{(m)}_{\mathbb Q}\otimes_{\widehat A^{(m)}_{\mathbb Q}} {D}^\d_{\b{\widehat P^*}, \Q}(^\dagger\infty)
\Big/\sum_{\xi\in \f{L}(H)}L_{\xi}(\widehat 
M^{(m)}_{\mathbb Q}\otimes_{\widehat A^{(m)}_{\mathbb Q}} {D}^\d_{\b{\widehat P^*}, \Q}(^\dagger\infty)),$$
where $\widehat M^{(m)}_{\mathbb Q}$ is regarded as an $\widehat A^{(m)}_{\mathbb Q}$-module via the isomorphism (\ref{AB}), and 
for any $b\in \widehat M^{(m)}_\Q$, $Q\in D^\d_{\widehat{\mathbb P}^*,\Q}(^\dagger\infty)$ and $\xi\in \f{L}(H)$, we define
\begin{align}\label{aln:L_xi}
L_{\xi}(b\otimes Q)=bL_{\xi} \otimes Q-b\otimes F_\pi(L_{\xi})Q.
\end{align}
Here, as in Corollary \ref{toroidalDmodiso'}, $bL_{\xi}$ is the right multiplication of a differential form $b$ by $L_{\xi}$, and 
$F_\pi$ is the Fourier transform of differential operators
$$F_\pi:  D^\d_{\widehat{\b{P}},\Q}(^\dagger\infty)\to D^\d_{\widehat{\mathbb P}^*,\Q}(^\dagger\infty), \quad 
x_i\mapsto \partial_{x'_i}/\pi_i, \quad \partial_{x_i}\mapsto -\pi x'_i.$$ 
Let $\beta'(L_\xi)$ be defined as in Proposition \ref{beta}. We have 
$$bL_{\xi}=\beta'(L_{\xi})(b).$$ We thus have 
\begin{align*}
\mathfrak F_\pi(\mathcal N)&\cong
\varinjlim_m \widehat M^{(m)}_{\mathbb Q}\otimes_{\widehat A^{(m)}_{\mathbb Q}} 
\mathcal {D}^\d_{\widehat {\mathbb P}^*, \Q}(^\dagger\infty)
\Big/\sum_{\xi\in \f{L}(H)}L_{\xi}(\widehat 
M^{(m)}_{\mathbb Q}\otimes_{\widehat A^{(m)}_{\mathbb Q}} \mathcal {D}^\d_{\widehat {\mathbb P}^*, \Q}(^\dagger\infty)),\\
\mathfrak F_\pi(\mathcal N)|_{\mathbb V^*_k}&\cong \varinjlim_m 
\widehat M^{(m)}_{\mathbb Q}\otimes_{\widehat A^{(m)}_{\mathbb Q}} \mathcal {D}^\d_{\widehat{\mathbb V}^*, \Q}
\Big/\sum_{\xi\in \f{L}(H)}L_{\xi}(\widehat 
M^{(m)}_{\mathbb Q}\otimes_{\widehat A^{(m)}_{\mathbb Q}} \mathcal {D}^\d_{\widehat{\mathbb V}^*, \Q}).
\end{align*}
Define a right $\widehat{\mathcal D}_{\widehat{\mathbb V}^*,\Q}^{(m)}$-module
$$\mathfrak F_\pi(\mathcal N)^{(m)}_\Q|_{\mathbb V^*_k}:
=(\widehat M^{(m)}_\Q\otimes_{\widehat{A}^{(m)}_\Q}\widehat{\mathcal D}_{\widehat{\mathbb V}^*,\Q}^{(m)})
\Big/\sum_{\xi\in \f{L}(H)}L_{\xi}(\widehat M^{(m)}_\Q
\otimes_{\widehat{A}^{(m)}_\Q}\widehat{\mathcal D}_{\widehat{\mathbb V}^*,\Q}^{(m)}),$$
where $\widehat{\mathcal D}_{\widehat{\mathbb V}^*,\Q}^{(m)}$ acts via the second factor of 
$\widehat M^{(m)}_\Q\otimes_{\widehat{A}^{(m)}_\Q}\widehat{\mathcal D}_{\widehat{\mathbb V}^*,\Q}^{(m)}$ by the right multiplication. 
By our construction, for any $m\leq m'$, we have
$${A}^{(m)}_\Q\cong {A}^{(m')}_\Q,\quad {M}^{(m)}_\Q\cong {M}^{(m')}_\Q.$$
On the other hand, we have 
$$\widehat{M}^{(m)}_\Q\otimes_{\widehat{A}^{(m)}_\Q}\widehat{A}^{(m')}_\Q\cong 
{M}^{(m)}_\Q\otimes_{{A}^{(m)}_\Q}\widehat{A}^{(m)}_\Q\otimes_{\widehat{A}^{(m)}_\Q}\widehat{A}^{(m')}_\Q
\cong {M}^{(m)}_\Q \otimes_{{A}^{(m)}_\Q}\widehat{A}^{(m')}_\Q
\cong {M}^{(m')}_\Q \otimes_{{A}^{(m')}_\Q}\widehat{A}^{(m')}_\Q\cong \widehat{M}^{(m')}_\Q.$$
Since  $\widehat{\mathcal D}_{\widehat{\mathbb V}^*}^{(m')}$ and ${\mathcal D}_{\widehat{\mathbb V}^*,\Q}^{\d}$ are flat 
over $\widehat{\mathcal D}_{\widehat{\mathbb V}^*}^{(m)}$ by \cite[3.5.3 and 3.5.4]{Berthelot1}, we have
$$\mathfrak F_\pi({\mathcal N})^{(m)}_\Q|_{\mathbb V^*_k}\otimes _{\widehat{\c D}_{\widehat{\mathbb V}^*,\Q}^{(m)}}
\widehat{\c D}_{\widehat{\mathbb V}^*,\Q}^{(m')}
\cong \mathfrak F_\pi({\mathcal N})^{(m')}_\Q|_{\mathbb V^*_k}, \quad 
\mathfrak F_\pi({\mathcal N})^{(m)}_\Q|_{\mathbb V^*_k}\otimes _{\widehat{\c D}_{\widehat{\mathbb V}^*,\Q}^{(m)}}{\c D}_{\widehat{\mathbb V}^*,\Q}^{\d}
\cong \mathfrak F_\pi({\mathcal N})|_{\mathbb V^*_k}.$$

\subsection{} \label{selfdual}
Choose a basis for each $V_j$ and write elements in $\mathrm{End}(V_j)$ by matrices $(x^{(j)}_{k_jl_k})$. Then 
$$(x^{(j)}_{k_jl_j})_{j\in\{1,\ldots, N\} ,\, k_j,l_j\in \{1,\ldots, \mathrm{dim}\,V_j\}}$$ 
is a linear coordinate system on $\mathbb V=\prod_{j=1}^N\mathrm{End}(V_j)$. $\mathbb V$ is self dual via the pairing 
$$\mathbb V\times\mathbb V\to \mathbb A^1, \quad \Big(\Big((x^{(1)}_{k_1l_1}), \ldots, (x^{(N)}_{k_Nl_N})\Big), 
\Big((y^{(1)}_{k_1l_1}), \ldots, (y^{(N)}_{k_Nl_N})\Big)\Big)\mapsto \sum_{j, k_j, l_j}x^{(j)}_{k_jl_j}y^{(j)}_{l_jk_j}.$$
Via the above pairing, the dual coordinate 
system for $\mathbb V^*=\mathbb V$ is $$(x'^{(j)}_{k_jl_j})_{j\in\{1,\ldots, N\} ,\, k_j,l_j\in \{1,\ldots, \mathrm{dim}\,V_j\}}$$ 
with $x'^{(j)}_{k_jl_j}=x^{(j)}_{l_jk_j}$. The Fourier transform $F_\pi$ for differential operators is given by 
$$F_\pi(x^{(j)}_{k_jl_j})=\partial_{x^{(j)}_{l_jk_j}}/\pi, \quad F_\pi(\partial_{x^{(j)}_{k_jl_j}})=-\pi x^{(j)}_{l_jk_j}.$$
For any element $\xi=(\xi_1,\xi_2)\in \mathfrak L(H)=\mathfrak L(G)\times\mathfrak L(G)$, let $\tau(\xi)=(\xi_2, \xi_1)$, and 
let $(c^{(j)}_{1k_jl_j})$ and $(c^{(j)}_{2k_jl_j})$ be the matrices of $\mathfrak L(\rho_j)(\xi_1)$ and 
$\mathfrak L(\rho_j)(\xi_2)$ respectively, where $\mathfrak L(\rho_j): \mathfrak L(G)\to \mathfrak{gl}(V_j)$ is the Lie algebra
representation corresponding to $\rho_j$. 
Then 
\begin{align*}
L_\xi &=\sum_{j=1}^N \sum_{k_j, l_j,m_j=1}^{\mathrm{dim}\,V_j}  
\Big(c^{(j)}_{1k_j m_j} x^{(j)}_{m_jl_j}-c^{(j)}_{2m_j l_j} x^{(j)}_{k_jm_j}\Big)\partial_{x^{(j)}_{k_jl_j}},\\
F_\pi(L_\xi)&=\sum_{j=1}^N \sum_{k_j, l_j,m_j=1}^{\mathrm{dim}\,V_j}  
\Big(c^{(j)}_{1k_j m_j} \partial_{x^{(j)}_{l_jm_j}}-c^{(j)}_{2m_j l_j} \partial_{x^{(j)}_{m_jk_j}}\Big)(-{x^{(j)}_{l_jk_j}})\\
&=\sum_{j=1}^N \sum_{k_j, l_j,m_j=1}^{\mathrm{dim}\,V_j}  
\Big(-c^{(j)}_{1k_j m_j} x^{(j)}_{l_jk_j}\partial_{{x^{(j)}_{l_jm_j}}}+ c^{(j)}_{2m_j l_j} x^{(j)}_{l_jk_j} \partial_{x^{(j)}_{m_jk_j}}\Big)\\
&\quad +\sum_{j=1}^N\mathrm{dim}\,\mathbb V_j\Big(-\mathrm{Tr}(\mathfrak L(\rho_j)(\xi_1))+\mathrm{Tr}(\mathfrak L(\rho_j)(\xi_2))\Big)\\
&=L_{\tau(\xi)}+\sum_{j=1}^N
\mathrm{dim}\,V_j\Big(-\mathrm{Tr}(\mathfrak L(\rho_j)(\xi_1))+\mathrm{Tr}(\mathfrak L(\rho_j)(\xi_2))\Big).
\end{align*}
Let $\chi$ be the character 
$$\chi:H=G\times G\to\b{G}_m,\quad h=(g_1,g_2)\mapsto \prod_{j=1}^N\det(\rho_j(g_1g^{-1}_2))^{\mathrm{dim}\,V_j}.$$
For any $R$-algebra $A$, $\chi$ induces a group homomorphism 
$$\chi(A):H(A)=\mathrm{Hom}_R(\mathrm{Spec}\,A, H) \to \mathbb G_{m}(A)=A^*.$$ 
Taking $A=R[H]$, then $\mathrm{id}_H$ defines an $R[H]$-point of $H$ and we have 
$\chi(R[H])(\mathrm{id}_H)\in R[H]^*.$ By our construction, $M^{(m)}$ is provided with an $R[H]$-comodule structure 
$$a: M^{(m)}\to R[H]\otimes_R M^{(m)}.$$ 
Consider the twisted comodule structure 
$$a'= \chi(R[H])(\mathrm{id}_H)\cdot a: M^{(m)}\to R[H]\otimes_RM^{(m)}.$$
Define $\beta'_\chi: \Gamma(H, \mathcal D^{(m)}_H)\to \mathrm{End}_R(M^{(m)})$ as in Proposition \ref{beta} 
using the twisted coaction $a'$. Then we have     
$$\beta'_\chi(L_{\xi})=\beta'(L_{\xi})+\sum_{j=1}^N
\mathrm{dim}\,V_j\Big(\mathrm{Tr}(\mathfrak L(\rho_j)(\xi_1))-\mathrm{Tr}(\mathfrak L(\rho_j)(\xi_2))\Big).$$ 
We can write (\ref{aln:L_xi}) as
\begin{align}\label{aln:L_xiagain}
&L_{\xi}(b\otimes Q)=bL_{\xi}\otimes Q-b\otimes F_\pi(L_\xi)Q=\beta'(L_\xi)(b)\otimes Q-b\otimes F_\pi(L_\xi)Q\nonumber\\
=&\, \beta'_\chi(L_{\xi})(b)\otimes Q-b\otimes   L_{\tau(\xi)}Q.
\end{align}

Let $d$ be the dimension of $G$, and let $I_G$ (resp. $I_H$) be the ideal of $R[G]$ (resp. $R[H]\cong R[G]\otimes_R R[G]$)
for the closed immersion $1_G: \mathrm{Spec}\,R\to G$ (resp. $1_H: \mathrm{Spec}\,R\to H$).
Choose $t_1, \ldots, t_d\in I_G$ so that their images in $I_G/I_G^2$ form a basis. 
Then $$T_1=t_1\otimes 1, \ldots, T_d=t_d\otimes 1,\quad T_{d+1}=1\otimes t_1, \ldots, 
T_{2d}=1\otimes t_d$$ lie in $I_H$, and their images in 
$I_H/I_H^2$ form a basis. As in Definition \ref{defn:rightinvoper}, 
$T_1^{\{k_1\}_{(m)}}\cdots T_{2d}^{\{k_{2d}\}_{(m)}}$ $(k_i\geq 0, \; k_1+\cdots+k_{2d}\leq n)$ 
form a basis of $P^n_{(m)}(I_H)$. Let  $\xi_1^{\langle k_1\rangle_{(m)}}\cdots \xi_{2d}^{\langle k_{2d}\rangle_{(m)}}$
be the dual basis of $F^nU^{(m)}(H)$, and let $L_{\xi_1^{\langle k_1\rangle_{(m)}}\cdots \xi_{2d}^{\langle k_{2d}\rangle_{(m)}}}$
be the corresponding right invariant differential operator on $H$. A basis for $\mathfrak L(H)$ is $\xi_1, \ldots, \xi_{2d}$. We 
have $$\tau(\xi_i)=\begin{cases}
\xi_{i+d} &\hbox{if }1\leq i\leq d, \\
\xi_{i-d}& \hbox{if }d+1\leq i\leq 2d.
\end{cases}$$

\begin{definition} For any $1\leq i\leq 2d$ and $1\leq r\leq p^m$, denote also by $L_{\xi_i^{\langle r\rangle_{(m)}}}$
the right $\widehat{\mathcal D}_{\widehat{\b{V}}}^{(m)}$-module homomorphism
$$L_{\xi_i^{\langle r\rangle_{(m)}}}: \widehat M^{(m)}\widehat{\otimes}_{R}\widehat{\mathcal D}_{\widehat{\b{V}}}^{(m)}\to \widehat 
M^{(m)}\widehat{\otimes}_{\widehat{A}^{(m)}}\widehat{\mathcal D}_{\widehat{\b{V}}}^{(m)},
\quad b\otimes Q\mapsto \beta'_\chi(L_{\xi_i^{\langle r\rangle_{(m)}}})(b)\otimes Q
-b\otimes (L_{\tau(\xi_i)^{\langle r\rangle_{(m)}}})^{\widehat{\mathbb V}}Q.$$
Note that in the domain the tensor product is taken over $R$.
Define $\mathfrak F_\pi(\mathcal N)^{(m)}$ to be the coherent right $\widehat{\c D}_{\widehat{\b{V}}}^{(m)}$-module 
$$\mathfrak F_\pi(\mathcal N)^{(m)}= (\widehat M^{(m)}\widehat{\otimes}_{\widehat{A}^{(m)}}\widehat{\mathcal D}_{\widehat{\b{V}}}^{(m)})
\Big/\sum_{1\leq i\leq 2d,\,1\leq {r}\leq p^m}L_{\xi_{i}^{\langle r\rangle_{(m)}}}(\widehat M^{(m)}\widehat\otimes_R
\widehat{\mathcal D}_{\widehat{\b{V}}}^{(m)}).$$
\end{definition}

We identify $\mathbb V^*$ with $\mathbb V$ as in \ref{selfdual}.
The next proposition shows that $\mathfrak F_\pi(\mathcal N)^{(m)}$ is an integral model of 
$\mathfrak F_\pi(\mathcal N)^{(m)}_\Q|_{\mathbb V_k}$.

\begin{proposition} We have an isomorphism 
$$\mathfrak F_{\pi}({\mathcal N})^{(m)}\otimes_{\b Z} {\Q}\cong \mathfrak F_\pi(\mathcal N)^{(m)}_\Q|_{\mathbb V_k}=
(\widehat M^{(m)}_\Q\otimes_{\widehat{A}^{(m)}_\Q}\widehat{\c D}_{\widehat{\b{V}},\Q}^{(m)})/\sum_{\xi\in \f{L}(H)}L_{\xi}(\widehat M^{(m)}_\Q\otimes_{\widehat{A}^{(m)}_\Q}\widehat{\c D}_{\widehat{\b{V}},\Q}^{(m)}).$$
\end{proposition}

\begin{proof} By Proposition \ref{inv-generated-by-L}, over $K$, $L_{\xi_i^{\langle r\rangle_{(m)}}}$ is a linear combination of 
operators of the form $P=L_{\zeta_1}\cdots L_{\zeta_n}$ for some $\zeta_1,\ldots, \zeta_n\in \mathfrak L(H)$. 
By Propositions \ref{description-alpha} and 
\ref{beta} and (\ref{aln:L_xiagain}), we have
\begin{align*}
&\quad\, \beta'_\chi(L_{\zeta_1}\cdots L_{\zeta_n})(b)\otimes Q-b\otimes (L_{\tau(\zeta_1)}\cdots L_{\tau(\zeta_n)})^{\widehat{\mathbb V}}Q \\
&= \beta'_\chi(L_{\zeta_n})\cdots \beta'_\chi(L_{\zeta_1})(b)\otimes Q-b\otimes L^{\widehat{\mathbb V}}_{\tau(\zeta_1)}
\cdots L^{\widehat{\mathbb V}}_{\tau(\zeta_n)}Q \\
&=  \sum_{i=1}^{n} \Big(\beta'_\chi(L_{\zeta_i})\cdots\beta'_\chi(L_{\zeta_1})b\otimes L^{\widehat{\mathbb V}}_{\tau(\zeta_{i+1})}
\cdots L^{\widehat{\mathbb V}}_{\tau(\zeta_n)}Q \\
& \;\;\qquad-\beta'_\chi(L_{\zeta_{i-1}})\cdots\beta'_\chi(L_{\zeta_1})b\otimes L^{\widehat{\mathbb V}}_{\tau(\zeta_{i})}\cdots 
L^{\widehat{\mathbb V}}_{\tau(\zeta_n)}Q\Big)\\
&=  \sum_{i=1}^{n} L_{\zeta_i}\Big(\beta'_\chi( L_{\zeta_{i-1}})\cdots\beta'_\chi(L_{\zeta_1})b\otimes 
L^{\widehat{\mathbb V}}_{\tau(\zeta_{i+1})}\cdots L^{\widehat{\mathbb V}}_{\tau(\zeta_n)}Q\Big),
\end{align*} 
which vanishes in 
\[\mathfrak F_\pi(\mathcal N)^{(m)}_\Q|_{\mathbb V_k}=
(\widehat M^{(m)}_\Q\otimes_{\widehat{A}^{(m)}_\Q}\widehat{\c D}_{\widehat{\b{V}},\Q}^{(m)})
/\sum_{\xi\in \f{L}(H)}L_{\xi}(\widehat M^{(m)}_\Q\otimes_{\widehat{A}^{(m)}_\Q}\widehat{\c D}_{\widehat{\b{V}},\Q}^{(m)}).\qedhere\]
\end{proof}

\begin{proposition}\label{quo-gr} Let
$\mathfrak F_\pi(\mathcal N)^{(m)}_k:=\mathfrak F_\pi(\mathcal N)^{(m)}\otimes_R k$. 
The good filtration on $\widehat M^{(m)}_k\cong M^{(m)}_k$ defined in \ref{affintmodel} induces a good 
filtration on $\mathfrak F_{\pi}(\mathcal N)^{(m)}_k$, and $\mathrm{Gr}(\mathfrak F_{\pi}(\mathcal N)^{(m)}_k)$ is a quotient of
$$\mathrm{Gr}(M^{(m)}_k) \otimes_{\mathrm{Gr}(A_k^{(m)})}\Big(\mathrm{Gr}(\mathcal D_{\mathbb V,k}^{(m)})\Big/
\sum_{1\leq i\leq 2d,\, 1\leq r\leq p^m}L_{\xi_i}^{\langle {r}\rangle_{(m)}}\mathrm{Gr}({\c D}_{\b V, k}^{(m)})\Big),$$
where $L_{\xi_i}^{\langle {r}\rangle_{(m)}}\in \mathrm{Gr}^r(\mathcal D_{\b{V}}^{(m)})$ is defined in Proposition \ref{def-L^<k>}.
\end{proposition}
 
\begin{proof} Let $\mathfrak m$ be the maximal ideal of $R$. We have 
\begin{align*}
(\widehat M^{(m)}\widehat{\otimes}_{\widehat{A}^{(m)}}\widehat{\mathcal D}_{\widehat{\b{V}}}^{(m)})_k&\cong 
(\widehat M^{(m)}\widehat{\otimes}_{\widehat{A}^{(m)}}\widehat{\mathcal D}_{\widehat{\b{V}}}^{(m)})/\mathfrak m 
(\widehat M^{(m)}\widehat{\otimes}_{\widehat{A}^{(m)}}\widehat{\mathcal D}_{\widehat{\b{V}}}^{(m)})\\
&\cong \widehat M^{(m)}/\mathfrak m \widehat M^{(m)}
\otimes_{\widehat{A}^{(m)}/\mathfrak m \widehat{A}^{(m)}}\widehat{\mathcal D}_{\widehat{\b{V}}}^{(m)}/
\mathfrak m\widehat{\mathcal D}_{\widehat{\b{V}}}^{(m)} \cong M^{(m)}_k \otimes_k \mathcal O_{\mathbb V_k}.
\end{align*}
The filtration $F^s(M^{(m)}_k\otimes_{k}\mathcal O_{\mathbb V_k}):=F^sM^{(m)}_k\otimes_{k}\mathcal O_{\mathbb V_k}$
defines a good filtration on the right $\mathcal D_{\mathbb V, k}^{(m)}$-module
$(\widehat M^{(m)}\widehat{\otimes}_{\widehat{A}^{(m)}}\widehat{\mathcal D}_{\widehat{\b{V}}}^{(m)})_k$, and we have an isomorphism
$$
\mathrm{Gr}(\widehat M^{(m)}\widehat{\otimes}_{\widehat{A}^{(m)}}\widehat{\mathcal D}_{\widehat{\b{V}}}^{(m)})_k\cong
\mathrm{Gr}(M^{(m)}_k\otimes_{k}\mathcal O_{\mathbb V_k})\cong \mathrm{Gr} (M^{(m)}_k)\otimes_{k}\mathcal O_{\mathbb V_k}.$$
Regard $\mathfrak F_\pi(\mathcal N)^{(m)}_k$ as a quotient of 
$(\widehat M^{(m)}\widehat{\otimes}_{\widehat{A}^{(m)}}\widehat{\mathcal D}_{\widehat{\b{V}}}^{(m)})_k$, 
and put the quotient good filtration on 
$\mathfrak F_\pi(\mathcal N)^{(m)}_k$. We then have an epimorphism
$$\mathrm{Gr}(M^{(m)}_k)\otimes_k\mathcal O_{\mathbb V_k}\twoheadrightarrow \mathrm{Gr}
(\mathfrak F_\pi(\mathcal N)^{(m)}_k).$$
To prove the proposition, it suffices to show $L_{\xi_i}^{\langle {r}\rangle_{(m)}}$
lies in the annihilator of $\mathrm{Gr}(\mathfrak F_{\pi}({\mathcal N})^{(m)}_k)$. 
By Corollary \ref{heighest-symbol-of-nabla}, it suffices to show that for any
$b\in F^sM^{(m)}_k$ and $P\in F^t \mathcal D_{\b{V}_k}^{(m)}$, we have
$$b\otimes P L^{\mathbb V_k}_{\xi_i^{\langle r\rangle_{(m)}}} \in  F^{s+t+r-1} \mathfrak F_\pi(\mathcal N)^{(m)}_k.$$
In $\mathfrak F_\pi(\mathcal N)^{(m)}$, we have the relation
$$\beta'_\chi(L_{\tau(\xi_i)^{\langle r\rangle_{(m)}}})b\otimes P-b\otimes 
L^{\mathbb V}_{\xi_i^{\langle r\rangle_{(m)}}} P=0.$$
We have 
\begin{align*}
b\otimes P L^{\mathbb V}_{\xi_i^{\langle r\rangle_{(m)}}}
&\equiv b\otimes L^{\mathbb V}_{\xi_i^{\langle r\rangle_{(m)}}} P \mod F^{s+t+r-1} \\
&\equiv \beta'_\chi(L_{\tau(\xi_i)^{\langle r\rangle_{(m)}}})b\otimes P  \mod F^{s+t+r-1} \\
&\equiv 0  \mod F^{s+t+r-1}, 
\end{align*}
where the last equation follows from the fact that the filtration on $M_k^{(m)}$ is $H_k$-invariant and hence
$\beta'_\chi(L_{\tau(\xi_i)^{\langle r\rangle_{(m)}}})b\in F^sM_k^{(m)}$. This proves our assertion.
\end{proof}

Let $\mathrm{Fr}^m_{\mathbb V_k}: \mathbb V_k\to \mathbb V_k^{(m)}$ be the relative Frobenius morphism.
We identify $\mathbb V$ with $\mathbb V^*$ as in \ref{selfdual}. By the proof of Proposition \ref{prop:Berthelot}, 
$\partial_{x'_i}^{\langle p^j\rangle_{(m)}}$ $(0\leq j<m)$ are nilpotent in $\mathcal D^{(m)}_{\mathbb V, k}$, and 
we have an isomorphism
$$\mathrm{Gr}(\mathcal D^{(m)}_{\mathbb V, k})_{\m{red}}\stackrel\cong \to\mathrm{Fr}_{\mathbb V_k}^{m*} 
\mathcal O_{T^*\mathbb V_k^{(m)}}$$
mapping the image of $\partial_{x'_i}^{\langle p^m\rangle_{(m)}}$ in $\mathrm{Gr}(\mathcal D^{(m)}_{\mathbb V, k})_{\m{red}}$ to
the image of $\mathrm{Fr}_{\mathbb V_k}^{m*} (\partial_{x'^{(m)}_i})$ in 
$\mathrm{Fr}_{\mathbb V_k}^{m*} \mathcal O_{T^*\mathbb V_k^{(m)}}$, where 
$(x'^{(m)}_1,\ldots, x'^{(m)}_n)$ are the coordinate of $\mathbb V_k^{(m)}$ induced 
by base change from the coordinate $(x'_1,\ldots, x'_n)$
of $\mathbb V_k$. This isomorphism  
induces an isomorphism $$\mathbb S\mathrm{pec}\, \mathrm{Gr}(\mathcal D^{(m)}_{\mathbb V, k})_{\m{red}}
\cong T^*\b{V}_k^{(m)}\times_{\b{V}_k^{(m)}} \b{V}_k\cong \b{V}_k^{*(m)}\times \b{V}_k,$$
where $\mathbb V^{*(m)}_k$ is the base change of $\mathbb V^*_k$ by $F^m_{\mathrm{Spec}\,k}$. 
We have $$A^{(m)}=R[\partial_{x'_1}^{\langle p^j \rangle_{(m)}},\ldots,\partial_{x'_n}^{\langle p^j \rangle_{(m)}}]_{0\leq j\leq m}
\subset \mathcal D^{(m)}_{\mathbb V}.$$ The same proof as that of Proposition \ref{prop:Berthelot} shows that 
$\partial_{x'_i}^{\langle p^j\rangle_{(m)}}$ $(0\leq j<m)$ are nilpotent in $A^{(m)}_k$, and we have an isomorphism 
\begin{align}\label{al:GrA}
\mathrm{Gr}(A^{(m)}_k)_{\m{red}}\stackrel\cong \to \mathcal O_{\mathbb V_k^{*(m)}}(\mathbb V_k^{*(m)})
\end{align}
mapping the image of $\partial_{x'_i}^{\langle p^m\rangle_{(m)}}$ in $\mathrm{Gr}(A^{(m)}_k)_{\m{red}}$ to
the image of $\partial_{x'^{(m)}_i}$ in $\mathcal O_{\mathbb V_k^{*(m)}}(\mathbb V_k^{*(m)})$. Denote the image 
of $\partial_{x'^{(m)}_i}$ in $\mathcal O_{\mathbb V_k^{*(m)}}$ by $\xi'^{(m)}_i$. We have the following. 

\begin{proposition}\label{M-support} Let $\mathrm{Fr}^m_{\mathbb V^*_k}: \mathbb{V}^*_k\to \mathbb{V}^{*(m)}_k$ be 
the relative Frobenius morphism
for $\b{V}^*_k$, and let $i_0$ be the closed immersion 
$$i_0: \mathbb{V}^*_k\to  {\mathbb A^1_k}\times_k \mathbb V^*_k,\quad x\mapsto (0,x).$$ 
The support of $\mathrm{Gr}(M^{(m)}_k)$ is contained in $\mathrm{Fr}^m_{\mathbb V^*_k}(i_0^{-1}(X'_k))$. 
\end{proposition}

\begin{proof}  By the discussion in \ref{affintmodel}, $\mathrm{Gr}(M^{(m)}_k)$
is a quotient of $M'^{(m)}_k\otimes_{k[t]}k[t]/(t)$. So we have 
$$\mathrm{supp}\,\mathrm{Gr}(M^{(m)}_k)\subset i_0^{(m),-1}(\mathrm{supp}\,M'^{(m)}_k),$$ 
$i^{(m)}_0$ be the closed immersion 
$$i^{(m)}_0: \mathbb{V}^{*(m)}_k\to  {\mathbb A^1_k}\times_k \mathbb V^{*(m)}_k,\quad x\mapsto (0,x).$$
It suffices to show the support of $M'^{(m)}_k$
lies in the image of $\mathbb A^1_k\times_k i_0^{-1}(X'_k)$ under the morphism
$$\mathrm{id}\times \mathrm{Fr}^m_{\mathbb V^*_k}: 
{\mathbb A^1_k}\times_k \mathbb V^*_k\to {\mathbb A^1_k}\times_k \mathbb V^{*(m)}_k.$$
$\widehat M'^{(m)}$ defines a coherent sheaf $(\widehat M'^{(m)})^\sim$ on the formal scheme 
$\m{Spf}\, R\langle t \rangle\widehat\otimes_R\widehat{B}^{(m)}$, and $M'^{(m)}_k$ defines
a coherent sheaf on special fiber $\mathrm{Spec}\, (k[t]\otimes_kB_k^{(m)})$. 
Similar to Proposition \ref{rel-M-omega}, we have an isomorphism 
$$\widehat M '^{(m)}_{\mathbb Q}\cong \Gamma(\mathrm{Spm}\, 
K\langle t \rangle\widehat\otimes_K\widehat{B}^{(m)}_{\Q}, (f'_{K*}\omega_{Y'_K})^{\mathrm{an}}).$$
Let $$\m{sp}:\mathrm{Spm}(K\langle t \rangle\widehat\otimes_K\widehat{B}^{(m)}_{\Q})\to\mathrm{Spec} (k[t]\otimes_kB_k^{(m)})$$ be the 
specialization map. We claim that 
$$X'^{\m{an}}_K\cap\mathrm{Spm}(K\langle t \rangle\widehat\otimes_K\widehat{B}^{(m)}_{\Q})\subset \m{sp}^{-1} 
(\mathrm{id}\times \mathrm{Fr}^m_{\mathbb V^*_k})(\mathbb A^1_k\times_k X'_{k,0}).$$
The generic fiber of the formal scheme 
$$\m{Spf}(R\langle t \rangle\widehat\otimes_R\widehat{B}^{(m)}) 
-(\mathrm{id}\times \mathrm{Fr}^m_{\mathbb V^*_k})(\mathbb A^1_k\times_k X'_{k,0})$$ is the rigid analytic space
$$\mathrm{Spm}(K\langle t \rangle\widehat\otimes_K\widehat{B}^{(m)}_{\Q})-
\m{sp}^{-1}(\mathrm{id}\times \mathrm{Fr}^m_{\mathbb V^*_k})(\mathbb A^1_k\times_k X'_{k,0}),$$ and 
the support of $f'_{K*}\omega_{Y'_K}$ is contained in $X'_K$.
If the claim is true, then $$(f'_{K*}\omega_{Y'_K})^{\mathrm{an}}
|_{\mathrm{Spm}(K\langle t \rangle\widehat\otimes_K\widehat{B}^{(m)}_{\Q})-
\m{sp}^{-1}(\mathrm{id}\times \mathrm{Fr}^m_{\mathbb V^*_k})(\mathbb A^1_k\times_k X'_{k,0})}=0,$$ and hence
the generic fiber of $(\widehat M'^{(m)})^\sim|_{\m{Spf}(R\langle t \rangle\widehat\otimes_R\widehat{B}^{(m)}) 
-(\mathrm{id}\times \mathrm{Fr}^m_{\mathbb V^*_k})(\mathbb A^1_k\times_k X'_{k,0})}$ vanishes. 
Since $M'^{(m)}$ is a subspace of $M'^{(m)}_{\Q}$, it is flat over $R$. So is $\widehat M'^{(m)}$. This implies that $M'^{(m)}_{k}$ 
is supported in $(\mathrm{id}\times \mathrm{Fr}^m_{\mathbb V^*_k})(\mathbb A^1_k\times_k X'_{k,0})$.

To prove the claim, we may replace $R$ by a finite extension, and assume it contains a $p^m$-th root $\sigma$ of $\pi$. Recall that
$$\widehat{B}^{(m)}=R\langle\pi{x_1^{p^j}},\ldots, \pi{x_n^{p^j}}\rangle_{0\leq j\leq m}.$$
Consider the homomorphism
$$R\langle \xi'^{(m)}_1,\ldots,\xi'^{(m)}_n \rangle\to \widehat{B}^{(m)}, \quad \xi'^{(m)}_i\mapsto \pi x_i^{p^m}.$$
It induces a commutative diagram
$$\begin{tikzcd}
\mathrm{Spm}\, (R\langle t\rangle\widehat \otimes_R  \widehat{B}^{(m)})_{\mathbb Q}
\arrow[r, "\m{sp}"] \arrow[d, "g_{\Q}"] & \mathrm{Spec}\, (k[t]\otimes_k B_k^{(m)}) \arrow[d, "g_{k}"] \\
\mathrm{Spm}\, K\langle t, \xi'^{(m)}_1,\ldots, \xi'^{(m)}_n\rangle \arrow[r, "\m{sp}"]  & 
\mathrm{Spec}\, k[t,\xi'^{(m)}_1,\ldots, \xi'^{(m)}_n].
\end{tikzcd}$$
For any $0\leq j\leq m-1$, we have 
$$(\pi x_i^{p^j})^p=\pi^p x_i^{p^{j+1}}=-p \pi x_i^{p^{j+1}}.$$
It follows that $\pi x_i^{p^j}$ $(0\leq j\leq m-1)$ are nilpotent in $B_k^{(m)}$. So
$g_{k}$ induces an isomorphism $$\mathrm{Spec}\, (k[t] \otimes_k B_k^{(m)})_{\m{red}}\cong \mathbb A^1_k\times_k\b{V}^{*(m)}_k,$$
which can be identified with the isomorphism induced by (\ref{al:GrA}).
Assume $X'$ is defined by a family of homogeneous equations $f_i(t,\mathbf x)=0$ in $\b{A}^1\times\mathbb V^*$.
Then we may identify $g_{\Q}((X'_K)^{\mathrm{an}}
\cap \mathrm{Spm}\, (K\langle t\rangle\widehat \otimes_K  \widehat{B}^{(m)}_{\mathbb Q}))$ with
\begin{align*}
&\{(s, \pi \mathbf y^{p^m})\in \mathrm{Spm}\, K\langle t, \xi'^{(m)}_1,\ldots, \xi'^{(m)}_n\rangle: f_i(s, \mathbf y)=0\}\\
=&\, \{(t, \mathbf y^{p^m})\in \mathrm{Spm}\, K\langle t, \xi'^{(m)}_1,\ldots, \xi'^{(m)}_n\rangle: f_i(\sigma t,\mathbf y)=0\}.
\end{align*}
The image of $g_{\Q}((X'_K)^{\mathrm{an}}\cap \mathrm{Spm}\, (K\langle t\rangle\widehat \otimes_K  \widehat{B}^{(m)}_{\mathbb Q}))$
under $\m{sp}$ is contained in
$$\{(t, \mathbf y^{p^m})\in \mathrm{Spec}\, k[t,\xi'^{(m)}_1,\ldots, \xi'^{(m)}_n] : f_i(0,\mathbf y)=0\}.$$
The claim follows.
\end{proof}

The following corollary shows that the characteristic cycle of $\mathfrak F(\mathcal N)^{(m)}$ is the zero section 
if restricted to the open subset $\mathbb V^{\mathrm{gen}}_k$ of $\mathbb V_k$.

\begin{corollary}\label{nil-gr} Identify $\mathrm{Spec}(\mathrm{Gr}(\mathcal D^{(m)}_{\mathbb V, k}))_{\m{red}}$ with 
$\b{V}_k^{*(m)}\times_k \b{V}_k$. Restricting to $\b{V}_k^{*(m)}\times_k\mathbb V_k^{\mathrm{gen}}$, 
the coherent module
$$\mathrm{Gr}(M^{(m)}_k) \otimes_{\mathrm{Gr}(A_k^{(m)})}\Big(\mathrm{Gr}(\mathcal D_{\b{V}, k}^{(m)})\Big/
\sum_{1\leq i\leq 2d,\, 1\leq r\leq p^m}L_{\xi_i}^{\langle {r}\rangle_{(m)}}\mathrm{Gr}({\c D}_{\b{V}, k}^{(m)})\Big)$$ 
is supported in the zero section of $0\times_k\mathbb V_k^{\mathrm{gen}}$.
\end{corollary}

\begin{proof} 
By Proposition \ref{def-L^<k>}, the function $L_{\xi_i}^{\left<{r}\right>_{(m)}}$ vanishes on 
$\mathrm{Spec}(\mathrm{Gr}(\mathcal D_{\b{V},k}^{(m)}))_{\m{red}}$ 
for $1\leq r<p^m$, and we have a commutative diagram
$$\begin{tikzcd}
\mathbb V_k^*\times_k \mathbb V_k\arrow[r,"\mathrm{Fr}^m_{\mathbb V^*_k}\times{\mathrm{id}_{\mathbb V_k}}"]
\arrow[ddr,"F^m_{\mathbb V^*_k}\times F^m_{\mathbb V_k}",swap]&
\b{V}_k^{*(m)}\times_k \b{V}_k\arrow[d,"\mathrm{id}_{\mathbb V_k^{*(m)}}
\times\mathrm{Fr}^m_{\mathbb V_k}",swap]\arrow[dr, "L_{\xi_i}^{\langle{p^m}\rangle_{(m)}}"]&&\\
&\mathbb V^{*(m)}_{k}\times_{k}\mathbb V^{(m)}_{k}\arrow[d]
\arrow[r,"L_{\xi_i}^{(m)}"]&\mathbb A^{1,(m)}_k\cong
\mathbb A^1_k\arrow[r]\arrow[d]&\mathrm{Spec}\,k\arrow[d,"F^m_{\mathrm{Spec},k}"]\\
&T^*\mathbb V_k\cong \mathbb V^*_k\times_k\mathbb V_k\arrow[r,"L_{\xi_i}"]&\mathbb A^1_k\arrow[r]&\mathrm{Spec}\,k,
\end{tikzcd}$$
where $F^m$'s are absolute Frobenii, $\mathrm{Fr}^m$'s are relative Frobenii, and the squares in the diagram are Cartesian.
By Proposition \ref{M-support}, the support of $\mathrm{Gr}(M^{(m)}_k)$ is contained in the image of $i_0^{-1}(X'_k)$ under the morphism
$\m{Fr}^m_{\mathbb V^*_k}: \b{V}_k^{*}\to \b{V}_k^{*(m)}$.
This implies
\begin{align*}
&\m{Supp}\Big(\mathrm{Gr}(M^{(m)}_k) \otimes_{\mathrm{Gr}(A_k^{(m)})} \Big(\mathrm{Gr}(\mathcal D_{\b{V}, k}^{(m)})\Big/
\sum_{1\leq i\leq 2d,\, 1\leq r\leq p^m}L_{\xi_i}^{\langle {r}\rangle_{(m)}}\mathrm{Gr}({\c D}_{\b{V}, k}^{(m)})\Big)\Big)\\
=& \Big(\m{Supp}\, \mathrm{Gr}(M^{(m)}_k)\times \b{V}_k\Big)\cap Z\Big(L_{\xi_i}^{\langle {p^m}\rangle_{(m)}}, i=1,\ldots, 2d\Big)\\
\subset &(\m{Fr}^m_{\mathbb V^*_k}\times \m{id}_{\mathbb V_k})
\Big((i_0^{-1}(X'_k)\times \b{V}_k)\cap Z\Big(L_{\xi_i}^{\langle {p^m}\rangle_{(m)}}
\circ (\m{Fr}^m_{\mathbb V^*_k}\times \m{id}_{\mathbb V_k}), i=1,\ldots, 2d\Big) \Big)\\
=&(\m{Fr}^m_{\mathbb V^*_k}\times\m{id}_{\mathbb V_k})\Big((i_0^{-1}(X'_k)\times \b{V}_k)\cap Z\Big(L^{(m)}_{\xi_i}
\circ{(\mathrm{Fr}^m_{\mathbb V^*_k}\times\mathrm{Fr}^m_{\mathbb V_k})},  i=1,\ldots, 2d\Big) \Big),
\end{align*}
where $Z(...)$ denotes the zero set of a family of regular functions. Note that 
$$(F^m_{\mathbb V^*_k}\times F^m_{\mathbb V_k})Z\Big(L^{(m)}_{\xi_i}
\circ{(\mathrm{Fr}^m_{\mathbb V^*_k}\times\mathrm{Fr}^m_{\mathbb V_k})},  i=1,\ldots, 2d\Big)
\subset Z\Big(L_{\xi_i},  i=1,\ldots, 2d\Big).$$ But $F^m_{\mathbb V^*_k}\times F^m_{\mathbb V_k}$ is identity
on the underlying topological spaces. So we have 
$$Z\Big(L^{(m)}_{\xi_i}
\circ{(\mathrm{Fr}^m_{\mathbb V^*_k}\times\mathrm{Fr}^m_{\mathbb V_k})},  i=1,\ldots, 2d\Big)
\subset Z\Big(L_{\xi_i},  i=1,\ldots, 2d\Big).$$
Therefore 
\begin{align*}
&\m{Supp}\Big(\mathrm{Gr}(M^{(m)}_k) \otimes_{\mathrm{Gr}(A_k^{(m)})} \Big(\mathrm{Gr}(\mathcal D_{\b{V}, k}^{(m)})\Big/
\sum_{1\leq i\leq 2d,\, 1\leq r\leq p^m}L_{\xi_i}^{\langle {r}\rangle_{(m)}}\mathrm{Gr}({\c D}_{\b{V}, k}^{(m)})\Big)\Big)\\
\subset &(\m{Fr}^m_{\mathbb V^*_k}\times\m{id}_{\mathbb V_k})\Big((i_0^{-1}(X'_k)
\times \b{V}_k)\cap Z\Big(L_{\xi_i},  i=1,\ldots, 2d\Big) \Big),
\end{align*}
If $\xi_i=(\zeta,\eta)\in\mathfrak g\times\mathfrak g$ and $(B, A)$ is a point in $\mathbb V_k\times_k\mathbb V_k\cong
\mathbb V^*_k\times_k\mathbb V_k$, then we have 
$$L_\xi(B, A)=\frac{\mathrm d}{\mathrm dt}\Big|_{t=0}\Big(\sum_{j=1}^N\mathrm{Tr}(B_j \rho_j(e^{t\zeta})A_j\rho_j(e^{-t\eta}))\Big)
=\frac{\mathrm d}{\mathrm dt}\Big|_{t=0}\Big(\sum_{j=1}^N\mathrm{Tr}(\rho_j(e^{-t\eta})B_j \rho_j(e^{t\zeta})A_j)\Big).$$
So the intersection $(i_0^{-1}(X'_k)\times A)\cap Z\big(L_{\xi_i},  i=1,\ldots, 2d\big)$ 
consists of critical points of the restriction of the function $B\mapsto \sum_{j=1}^N\mathrm{Tr}(A_jB_j)$ to $(G_k\times_k G_k)$-orbits
in $i_0^{-1}(X'_k)$. By the orbit decomposition of $i_0^{-1}(X'_k)$
in Proposition \ref{prop:orbitdecomX} in the appendix and Definition \ref{defn:nondeg}, 
if $A$ lies in $\mathbb V^{\mathrm{gen}}_k$, 
then $$(i_0^{-1}(X'_k)\times A)\cap Z\Big(L_{\xi_i},  i=1,\ldots, 2d\Big)=(0, A).$$ 
So $$(i_0^{-1}(X'_k)\times_k \b{V}^{\mathrm{gen }}_k)\cap Z\Big(L_{\xi_i},  i=1,\ldots, 2d\Big)=0\times_k \b{V}^{\mathrm{gen }}_k.$$
Our assertion follows. 
\end{proof}

\begin{corollary}\label{coherent} $\f{F}_{\pi}({\mathcal N})^{(m)} |_{\b{V}^{\mathrm{gen}}_k}$ is coherent as an 
$\c{O}_{\widehat{\b{V}}}|_{\mathbb V_k^{\mathrm{gen}}}$-module.
\end{corollary}

\begin{proof} By Proposition \ref{quo-gr} and Corollary \ref{nil-gr}, $\mathrm{Gr}^{\geq 1}(\mathcal D_{\mathbb V,k}^{(m)})$ acts nilpotently 
on $\mathrm{Gr}(\mathfrak F_{\pi}({\mathcal N})_k^{(m)})|_{\b{V}^{\mathrm{gen}}_k}$. By the construction of the good filtration in 
Proposition \ref{quo-gr}, 
$\mathrm{Gr}(\mathfrak F_{\pi}({\mathcal N})_k^{(m)})$ is a finitely generated $\mathrm{Gr}(\mathcal D_{\mathbb V,k}^{(m)})$-module. 
So we have $$F^i(\mathfrak F_{\pi}({\mathcal N})_k^{(m)}) |_{\b{V}^{\mathrm{gen}}_k}
=\mathfrak F_{\pi}({\mathcal N})^{(m)}_k |_{\b{V}^{\mathrm{gen}}_k}$$ 
for sufficiently large $i$. Therefore, $\mathfrak F_\pi({\mathcal N})^{(m)}_k |_{\b{V}^{\mathrm{gen}}_k}$ is a coherent $\c{O}_{\b{V}_k}$-module. 
By \cite[3.2.2]{Berthelot1}, this implies that $\mathfrak {F}_{\pi}({\mathcal N})^{(m)} |_{\b{V}^{\mathrm{gen}}_k}$ is coherent as an 
$\c{O}_{\widehat{\b{V}}}$-module.
\end{proof}

By the proofs of Proposition \ref{M-support} and 
Corollary \ref{nil-gr} for the case $m=0$ and the explicit construction of $M'^{(0)}$ in \ref{S0=omega}, we have
the following. 

\begin{corollary}\label{nil-gr0} Identify $\mathrm{Spec}(\mathrm{Gr}(\mathcal D^{(0)}_{\mathbb V, k}))$ with 
$\b{V}_k^*\times_k \b{V}_k$. The coherent module 
$$\mathrm{Gr}(M^{(0)}_k) \otimes_{\mathcal O_{\mathbb V^*_k}(\mathbb V^*_k)}\Big(\mathcal O_{\b V_k^*\times_k \b V_k}\Big/
\sum_{1\leq i\leq 2d}L_{\xi_i}\mathcal O_{\b V_k^*\times_k \b V_k}\Big)$$ is a quotient 
$$\Gamma(\mathbb A_k^1\times_k \b{V}^*_k, f'_{k*}\omega_{Y'_k})
\otimes_{k[t,x'_1, \ldots, x'_n ],\phi_k} \mathcal O_{\mathbb V^*_k}(\mathbb V^*_k)
 \otimes_{\mathcal O_{\mathbb V^*_k}(\mathbb V^*_k)}\Big(\mathcal O_{\b V_k^*\times_k \b V_k}\Big/
\sum_{1\leq i\leq 2d}L_{\xi_i}\mathcal O_{\b V_k^*\times_k \b V_k}\Big),$$
where $\phi_k$ is the homomorphism 
$$\phi_k: k[t,x'_1, \ldots, x'_n ]\to \mathcal O_{\mathbb V^*_k}(\mathbb V^*_k),\quad t\mapsto 0, 
\quad x'_i\mapsto x'_i.$$ It is supported
in 
$$(i_0^{-1}(X'_k)\times_k \b{V}_k)\cap Z\Big(L_{\xi_i},  i=1,\ldots, 2d\Big).$$
Restricting to $\b{V}_k^*\times_k\mathbb V_k^{\mathrm{gen}}$, 
it is supported in the zero section $0\times_k\mathbb V_k^{\mathrm{gen}}$.
\end{corollary}

\begin{proposition}\label{rank} Locally 
$\f{F}_{\pi}({\mathcal N})^{(0)}|_{\b{V}^{\mathrm{gen}}_k}$ is a quotient of a free 
$\c{O}_{\widehat{\b{V}}}|_{\mathbb V_k^{\mathrm{gen}}}$-module of rank at most
$d!\int_{\Delta_{\infty}\cap\mathfrak C}\prod_{\alpha \in R^+} \frac{\lambda(H_\alpha)^2}{\rho(H_\alpha)^2}\mathrm d\lambda.$
\end{proposition}

\begin{proof}
By Corollary \ref{coherent}, $\f{F}_{\pi}({\mathcal N})^{(0)}|_{\b{V}^{\mathrm{gen}}_k}$ is a coherent $\c{O}_{\widehat{\b{V}}}$-module. 
By Nakayama's lemma, it suffices to show that for any $\bar k$-point $A$ of $\b{V}^{\mathrm{gen}}_k$, the dimension of the fiber of
$\f{F}_{\pi}({\mathcal N})^{(0)}_{k}$ at $A$ does not exceed $d!\int_{\Delta_{\infty}\cap\mathfrak C}
\prod_{\alpha \in R^+} \frac{\lambda(H_\alpha)^2}{\rho(H_\alpha)^2}\mathrm d\lambda$.
By Proposition \ref{quo-gr} and Corollary \ref{nil-gr0}, we have a surjection
$$\Gamma(\mathbb A_k^1\times \b{V}, f'_{k*}\omega_{Y'_k})\otimes_{k[t,x_1, \ldots, x_n ],\phi_k}
\Big(\mathcal O_{\b V_k^*\times_k \b V_k}\Big/
\sum_{1\leq i\leq 2d}L_{\xi_i}\mathcal O_{\b V_k^*\times_k \b V_k}\Big)
\twoheadrightarrow \mathrm{Gr}(\f F_{\pi}({\mathcal N})^{(0)}_k).$$
The left hand side is supported in 
$(i_0^{-1}(X'_k)\times_k \b{V}_k)\cap Z\big(L_{\xi_i},  i=1,\ldots, 2d\big),$
and $$(i_0^{-1}(X'_k)\times_k \b{V}^{\mathrm{gen}}_k)\cap Z\big(L_{\xi_i},  i=1,\ldots, 2d\big)\subset 0\times_k\b{V}^{\mathrm{gen}}_k.$$
Let $L_{\xi, A}=L_{\xi}|_{\mathbb V^*\times \{A\}}$. Then we have 
$$i_0^{-1}(X'_{\bar k})\cap Z\big(L_{\xi_i,A},  i=1,\ldots, 2d\big)= 0.$$
For any integer $0\leq r\leq d$, we claim that there 
exists a linear subspace $F_r\subset \f{L}(H)_{\bar k}$ of dimension $r$ such that 
$i_0^{-1}(X'_{\bar k})\cap Z\big(L_{\xi,A}, \xi\in F_r)$
has dimension $\leq d-r$. For $r=0$, this is trivial. Suppose $1\leq r\leq d$ and let $F_{r-1}$ be a linear subspace of dimension $r-1$ such that 
$i_0^{-1}(X'_{\bar k})\cap Z\big(L_{\xi,A}, \xi\in F_{r-1})$ has dimension $\leq d-r+1$. 
Let $\eta_1,\cdots, \eta_s$ be the generic points of 
those irreducible components $i_0^{-1}(X'_{\bar k})\cap Z\big(L_{\xi,A}, \xi\in F_{r-1})$ of dimension $d-r+1$.
None of them are closed points and hence $\eta_i\not=0$. But 
$$0=i_0^{-1}(X'_{\bar k})\cap Z\big(L_{\xi,A}, \xi\in \mathfrak L(H)_{\bar k}).$$ So there exists $\xi\in \mathfrak L(H)_{\bar k}$ such that 
$\eta_i\notin i_0^{-1}(X'_{\bar k})\cap Z(L_{\xi,A})$. The set
$$U_i:=\{ \xi\in \f{L}(H)_{\bar k}: \eta_i\notin i_0^{-1}(X'_{\bar k})\cap Z(L_{\xi,A})\}$$
is a nonempty Zariski open subset of $\f{L}(H)$. The intersection $\cap_i U_i$ is nonempty. Choose $\xi_0\in \cap_i U_i$ and 
let $F_r=\m{span}\{F_{r-1}, \xi_0\}$. Then $F_r$ is a linear subspace of dimension $r$ such that
$i_0^{-1}(X'_k)\cap Z(L_{\xi,A}, \xi\in F_r)$ has dimension $\leq d-r$.
Let $F$ be a linear subspace $\f{L}(H)_{\bar k}$ of dimension $d$ so that $i_0^{-1}(X'_k)\cap Z(L_{\xi,A}, \xi\in F)$ is of dimension $0$.
Since $i_0^{-1}(X'_k)\cap Z(L_{\xi}, \xi\in F)$ is conical, we must have 
$$i_0^{-1}(X'_k)\cap Z(L_{\xi,A}, \xi\in F)=0$$  or equivalently,
$$X'_{\bar k}\cap (0\times Z(L_{\xi,A}, \xi\in F))=0,$$ where the last intersection is taken in 
$\mathbb V'_{\bar k}=\mathbb A^1_{\bar k}\times_{\bar k}\mathbb V_{\bar k}$. 
Let $Z_F= 0\times Z(L_{\xi,A}, \xi\in F)$, which is a linear subspace of $\mathbb V'_{\bar k}$. We have 
$$X'_{k}\cap Z_F=0.$$
So $d+1=\dim X'_{k}\leq \m{codim}\, Z_F$. As 
$Z_F$ is a linear subspace of $\mathbb V'_k$ defined by $d+1$ equations, 
we must have $d+1= \m{codim}\,\overline Z_F$.
The fiber of $$\Gamma(\mathbb A_k^1\times \b{V}, f'_{k*}\omega_{Y'_k})\otimes_{k[t,x_1, \ldots, x_n ],\phi_k}
\Big(\mathcal O_{\b V_k^*\times_k \b V_k}\Big/
\sum_{1\leq i\leq 2d}L_{\xi_i}\mathcal O_{\b V_k^*\times_k \b V_k}\Big)$$ at $A$ is a quotient of 
$f'_{\bar k*}\omega_{Y'_{\bar k}}\otimes_{\mathcal O_{\mathbb V'_{\bar k}}}\mathcal O_{Z_F}$, which is supported at the origin.
To prove our assertion, it suffices to show 
\begin{eqnarray}\label{eq:goal}
\mathrm{dim}\, f'_{\bar k*}\omega_{Y'_{\bar k}}\otimes_{\mathcal O_{\mathbb V'_{\bar k}}}\mathcal O_{Z_F}
= d!\int_{\Delta_{\infty}\cap\mathfrak C}\prod_{\alpha \in R^+} \frac{\lambda(H_\alpha)^2}{\rho(H_\alpha)^2}\mathrm d\lambda.
\end{eqnarray}

Recall that for any two coherent $\mathcal O_{\mathbb V'_k}$-modules $\mathcal F$ and $\mathcal G$ on 
$\mathbb V'_k$ such that the intersection of their supports has dimension zero,  
their intersection multiplicity (\cite[V C Theorem 1]{Serre2}) is defined 
to be the alternating sum 
\[\chi(\mathcal F,\mathcal G)=\sum_{i}(-1)^i\mathrm{length}(\mathcal Tor_i^{\mathcal{O}_{\mathbb V'_k}}(\mathcal F,
\mathcal G)).\]
Let $r=[K(\mathbb G_{m,k}\times_k G_k): K(X'_k)]$ be the degree of the 
extension of the function fields of $Y'_k$ and $X'_k$, and let $\eta$ be the generic point of $X'_k$.
Then $$(f'_{k*}\omega_{Y'_k})_{\eta}\cong K(\mathbb G_{m,k}\times_k G_k)\cong \mathcal O_{X'_k,\eta}^r.$$
Since $X'_k$ is affine, we may find global section 
$s_1,\ldots, s_r\in \Gamma(X'_k, f'_{k*}\omega_{Y'_k})$ so that their germs at $\eta$ form a basis 
of $(f'_{k*}\omega_{Y'_k})_{\eta}$. Let $\phi:\mathcal O_{X'_k}^r\to f'_{k*}\omega_{Y'_k}$ be the 
morphism mapping the standard basis of $\mathcal O_{X'_k}^r$ to $\{s_1, \ldots, s_r\}$. Then 
the dimensions of the 
supports of the sheaves  
$\mathrm{ker}\,\phi$ and $\mathrm{coker}\,\phi$ 
are strictly smaller than $\mathrm{dim}\,X'_k=\mathrm{codim}\,Z_F$. By Serre's theorem \cite[V B Theorem 1 (3)]{Serre2}
on the vanishing of the intersection multiplicity, we have 
$$\chi(\mathrm{ker}\,\phi,\mathcal{O}_{Z_F})=\chi(\mathrm{coker}\,\phi,\mathcal{O}_{Z_F})=0.$$ So we have 
\begin{eqnarray}\label{deletequot}
\chi(f'_{k*}\omega_{Y'_k},\mathcal{O}_{Z_F})
=r\chi(\mathcal{O}_{X'_k},\mathcal{O}_{Z_F}).
\end{eqnarray} Here for convenience, we denote the direct images in $\mathbb V'_k$ of the sheaves $\mathrm{ker}\,\phi$,
$\mathrm{coker}\,\phi$, $\mathcal O_{Z_F}$ and $\mathcal{O}_{X'_k}$ by the same symbols. 

Let $\mathbb P'=\mathbb P(\mathbb A^1\times_R\mathbb V')$ be the projective space containing the affine space $\mathbb V'$, and let
$\overline X'_{\bar k}$ and $\overline Z_F$ be the closures of $X'_{\bar k}$ and $Z_F$ 
in $\mathbb P'_{\bar k}$, respectively. 
Since $X'_{\bar k}$ and $Z_F$ are conical, we have 
$$\overline X'_{k}\cap \overline Z_F=X'_{k}\cap Z_F=0.$$
So the intersection number of $\overline X'_k$ and $\overline Z_F$ is given by 
$\chi(\mathcal{O}_{X'_k},\mathcal{O}_{Z_F})$.
This intersection number is also $\mathrm{deg}(\overline X'_k)$. 
By Lemma \ref{degY} below, we have
\begin{eqnarray}\label{degree}
r\chi(\mathcal{O}_{X'_k},\mathcal{O}_{Z_F})=
r\,\mathrm{deg}(\overline X'_k)
=d!\int_{\Delta_{\infty}\cap\mathfrak C} \prod_{\alpha \in R^+} \frac{\lambda(H_\alpha)^2}{\rho(H_\alpha)^2}\mathrm d\lambda.
\end{eqnarray}
By Corollary \ref{cor:usedin4} in the appendix, the sheaf $f'_{k*}\omega_{Y'_k}$ is Cohen-Macaulay. 
Let $\{l_1,\dots, l_{d+1}\}$ be a basis of $(\mathbb V'_k/Z_F)^*$. 
Since $\mathrm{dim} (X'_k\cap Z_F)=0$, $\{l_1,\dots, l_{d+1}\}$ is a regular sequence for 
$f'_{k*}\omega_{Y'_k}$ by \cite[IV B Theorem 2]{Serre2}. The Koszul complex of this sequence computes 
$\mathcal Tor_i^{\mathcal{O}_{\mathbb V'_k}}(f'_{k*}\omega_{Y'_k},\mathcal{O}_{Z_F})$.
So by \cite[IV A Proposition 2]{Serre2}, we have $$\mathcal Tor_i^{\mathcal{O}_{\mathbb V'_k}}
(f'_{k*}\omega_{Y'_k},\mathcal{O}_{Z_F})=0 \quad 
(i\geq 1).$$ We conclude
\begin{eqnarray}\label{deletetor}
\dim_{\overline K}(f'_{k*}\omega_{Y'_k}\otimes_{\mathcal{O}_{\mathbb V'_k}} \mathcal{O}_{Z_F})
=\chi(\omega_{Y'_k},\mathcal{O}_{Z_F}).
\end{eqnarray}
Combining (\ref{deletequot}), (\ref{degree}) and (\ref{deletetor}) together, we get 
(\ref{eq:goal})
\end{proof}

\begin{lemma}\label{degY} Notation as above. We have 
$$[K(\mathbb G_{m,k}\times_k G_k): K(X'_k)]\mathrm{deg}(\overline X'_k)
=d!\int_{\Delta_{\infty}\cap\mathfrak C} \prod_{\alpha \in R^+} \frac{\lambda(H_\alpha)^2}{\rho(H_\alpha)^2}\mathrm d\lambda.$$
\end{lemma}

\begin{proof}  Since $\overline X'$ is flat over $R$, 
we have $\mathrm{deg}(\overline X'_k)=\mathrm{deg}(\overline X'_K)$ by \cite[III 9.9]{Hartshorne}. 
It suffices to show $$[K(\mathbb G_{m,K}\times_K G_K): K(X'_K)]\mathrm{deg}(\overline X'_K)
=d!\int_{\Delta_{\infty}\cap\mathfrak C} \prod_{\alpha \in R^+} \frac{\lambda(H_\alpha)^2}{\rho(H_\alpha)^2}\mathrm d\lambda.$$
Let $\tilde G'_K$ be the image of the homomorphism 
$$\mathbb G_{m,K}\times_K G_K\to \prod_{j=0}^N \mathrm{GL}(V_{j,K}),\quad g\mapsto (\rho'_{0,K}(g),\rho'_{1,K}(g)\ldots, 
\rho'_{N,K}(g)).$$
We may identify the weight
lattice $\tilde \Lambda'$ of $\tilde G'_K$ with a sub-lattice of the weight lattice $\Lambda':=\mathbb Z\times \Lambda$ 
of $G'_K:=\mathbb G_{m,K}\times_K G_K$ with index $$r:=[\Lambda':\tilde \Lambda']=[K(\mathbb G_{m,K}\times_k G_k): K(X'_K)].$$ 
We may identify 
$\tilde \Lambda'_{\mathbb R}:=\tilde \Lambda'\otimes_{\mathbb Z}\mathbb R$ with 
$\Lambda'_{\mathbb R}:=\Lambda'\otimes_{\mathbb Z}\mathbb R\cong \mathbb R\times \Lambda_{\mathbb R} $. 
The Lebesgue measure $\mathrm dt \mathrm d\tilde \lambda$ on $\Lambda'_{\mathbb R}$ such that 
$\mathrm{vol}(\Lambda'_{\mathbb R}/\tilde \Lambda')=1$ is identified with 
$\frac{1}{r}\mathrm dt \mathrm d\lambda$. 
The Newton polytope at $\infty$ for the family of representations $\rho'_0,\rho'_1,\ldots, \rho'_N$ of $\mathbb G_m\times G$ is 
$\Delta'_\infty=\bigcup_{t\in [0,1]}t(1\times \Delta_{\infty})$. Moreover, we have
$$d=2|R^+|+\mathrm{dim}\,\Lambda_{\mathbb R}.$$ By the Brion-Kazarnovskii formula (\cite[4.1 Th\'eor\`eme]{Brion}, 
\cite{Kazarnovskii}), we have 
\begin{align*}
\mathrm{deg}(\overline X'_K)&=(d+1)! \int_{\bigcup_{t\in [0,1]}t(1\times \Delta_{\infty}\cap\mathfrak C)}
\prod_{\alpha \in R^+} \frac{\lambda(H_\alpha)^2}{\rho(H_\alpha)^2}\frac{1}{r}\mathrm dt \mathrm d\lambda\\
&=
\frac{(d+1)!}{r}\int_0^1  t^d\mathrm d t \Big(\int_{\Delta_{\infty}\cap\mathfrak C} 
\prod_{\alpha \in R^+} \frac{\lambda(H_\alpha)^2}{\rho(H_\alpha)^2}\mathrm d\lambda\Big)\\
&= \frac{d!}{r}\int_{\Delta_{\infty}\cap\mathfrak C} 
\prod_{\alpha \in R^+} \frac{\lambda(H_\alpha)^2}{\rho(H_\alpha)^2}\mathrm d\lambda. \qedhere
\end{align*} 
\end{proof}

\begin{theorem}\label{2ndmainthm}
$\f{F}_{\pi}(\c{N})|_{\mathbb V_{k}^{\mathrm{gen}}}$ is coherent over 
$\c{O}_{\widehat{\mathbb V}, \Q}|_{\mathbb V_k^{\mathrm{gen}}}$. For any divisor $T$ 
of $\b{P}_k$ containing the complement of $\mathbb V_{k}^{\mathrm{gen}}$,
$\m{sp}^*(\f{F}_{\pi}(\c{N})|_{\b{P}_k-T})$ is an $F$-isocrystal on $\mathbb P_{k}-T$ overconvergent along $T$
with rank at most $d!\int_{\Delta_{\infty}\cap\mathfrak C}\prod_{\alpha \in R^+} \frac{\lambda(H_\alpha)^2}
{\rho(H_\alpha)^2}\mathrm d\lambda.$
\end{theorem}

\begin{proof} Let $U\subset \mathbb V_{k}^{\mathrm{gen}}$ be an affine open subset. By \cite[2.2.9]{Caro-L} and the construction in 
\ref{subsection:Fourier}, the canonical map 
\begin{eqnarray}\label{eqn:canonical0m}
\Gamma(U, \f{F}_\pi({\mathcal N})^{(0)}_{\Q})\to \Gamma(U, \f{F}_\pi({\mathcal N})^{(m)}_{\Q})
\end{eqnarray}
has a dense image for each $m$. Here the topology on $\Gamma(U, \f{F}_\pi({\mathcal N})^{(m)}_{\Q})$ is induced by 
its $\widehat{\mathcal D}_{\widehat{\b{V}},\Q}^{(m)}$-module structure (\cite[4.1.1]{Berthelot1}). By Corollary \ref{coherent} and 
\cite[4.1.2]{Berthelot1}, it is equivalent to the topology induced by its $\c{O}_{\widehat{\b{V}},\b{Q}}$-module structure.
By \cite[3.7.3.1]{NonArchimedean}, the map (\ref{eqn:canonical0m}) has closed image. Hence it is surjective. Taking 
$\varinjlim_m$, we get an epimorphism 
$$\Gamma(U, \f{F}_{\pi}({\mathcal N})^{(0)}_{\Q})\twoheadrightarrow \Gamma(U, \f{F}_{\pi}({\mathcal N})).$$
By \cite[2.2.13]{Caro-L}, $\f{F}_\pi({\mathcal N})|_{U}$ is coherent over $\c{O}_{\widehat{\b{V}},\b{Q}}|_U$.
So $\f{F}_{\pi}(\c{N})|_{\mathbb V_{k}^{\mathrm{gen}}}$ is coherent over 
$\c{O}_{\widehat{\mathbb V}, \Q}|_{\mathbb V_k^{\mathrm{gen}}}$.
By \cite[2.2.12]{Caro-L}, $\f{F}_{\pi}({\mathcal N})$ is an isocrystal on $\mathbb P_{k}-T$ overconvergent 
along $T$. By Proposition \ref{rank} and the fact that  
$\f{F}_\pi({\mathcal N})^{(0)}_{\Q}\to  \f{F}_\pi({\mathcal N})$ is surjective on $\mathbb V_{k}^{\mathrm{gen}}$, the rank of 
$\f{F}_{\pi}({\mathcal N})|_{\mathbb V_{k}^{\mathrm{gen}}}$ is at most
$d!\int_{\Delta_{\infty}\cap\mathfrak C}\prod_{\alpha \in R^+} \frac{\lambda(H_\alpha)^2}{\rho(H_\alpha)^2}\mathrm d\lambda.$
\end{proof}

\begin{proof}[Proof of Theorem \ref{thm:hyp}]
That $\mathrm{Hyp}_{\pi,!}$ is an overholonomic arithmetic $\mathcal D$-module is proved in Proposition
\ref{prop:Fourier}. By Proposition \ref{prop:mhyp},  $\mathrm{Hyp}_{\pi,!}$ is a direct factor of $\mathfrak F_\pi(\mathcal N)$. The 
other assertions about $\mathrm{Hyp}_{\pi,!}$ follow from Theorem \ref{2ndmainthm}. The assertions for $\mathrm{Hyp}_{\pi,+}$
follow by duality.
\end{proof}

\section{Appendix}

\subsection{} Let $k$ be a field, $G$ a reductive algebraic group split over $k$, $T$ a maximal torus of $G$ 
defined over $k$, $\Lambda=\mathrm{Hom}_k(T,\mathbb G_{m,k})$ the weight lattice, 
$\rho_j: G\to\mathrm{GL}(V_j)$ $(j=1, \ldots, N)$ a family of representations of $G$, and
$$V_j=\bigoplus_{\lambda\in \Lambda} V_j(\lambda)$$ the weight decomposition. Define the \emph{Newton polytope}
$\Delta$ to be the convex hull in 
$\Lambda_{\mathbb R}:=\Lambda\otimes_{\mathbb Z}\mathbb R$ of the weights appeared in 
$V_j$ $(j=1, \ldots, N)$. For any face $\tau$ of $\Delta$, let $e(\tau)=(e(\tau)_j)\in \prod_{j=1}^N \mathrm{End}(V_j)$ be defined by 
$$e(\tau)_j\Big|_{ V_j(\lambda)}=\begin{cases}
\mathrm{id}_{V_j(\lambda)}&\hbox{if }\lambda\in \tau, \\
0&\hbox{otherwise}.
\end{cases}$$
We have an action
\begin{eqnarray*}
(G\times_k G)\times_k \prod_{j=1}^N \mathrm{End}(V_j)&\to& \prod_{j=1}^N \mathrm{End}(V_j),\\
\Big((g, h), (A_1, \ldots, A_N)\Big)&\mapsto&
\Big(\rho_1(g)A_1\rho_1(h^{-1}), \quad \rho_N(g)A_N\rho_N(h^{-1})\Big).
\end{eqnarray*}
Let $X$ be the scheme theoretic image of the morphism 
$$\iota: G\to \prod_{j=1}^N \mathrm{End}(V_j).$$

\begin{proposition}\label{prop:orbitdecomgeneral} Notation as above. We have the orbit decomposition 
$$X=\bigsqcup_{\tau\prec \Delta} Ge(\tau) G.$$
\end{proposition}

\begin{proof} By base change to an algebraic closure of $k$, we may assume $k$ is algebraically closed. 
Note that $X$ is an algebraic monoid. 
By \cite[Lemma 3]{charporbit}, every $G$-orbit of $X$ 
contains an idempotent element lying in the closure $\overline{\iota(T)}$ of $\iota(T)$. 
It suffices to show that for any point $x$ in $\overline{\iota(T)}$, 
there exists a face $\tau$ of $\Delta$ 
such that $x\in Ge(\tau)G$. Let $\mu_1, \ldots, \mu_l$ be all the weights appeared in 
$V_j$ $(j=1, \ldots, N)$. Then $\overline{\iota(T)}$ can be identified with the scheme theoretic image of the morphism
$$\iota': T\to \mathbb A^l, \quad t \mapsto  (\mu_1(t), \ldots,\mu_l(t).)$$
So $\overline{\iota(T)}$ is isomorphic to the affine toric scheme $\mathrm{Spec}\, k[M]$, where 
$M$ is the submonoid of $\Lambda$ generated by $\mu_1, \ldots,\mu_l$. The torus action orbits of this affine toric scheme
are in one-to-one correspondence with the faces of $\Delta$. 
\end{proof}

Keep the notation in \ref{Homogeneization}. For the group $\mathbb G_m\times G$, the weight 
lattice is $\Lambda'\cong \mathbb Z\oplus \Lambda$, and for the representations $\rho'_0, \rho'_1, \ldots, \rho'_N$
of $\mathbb G_m\times G$, the Newton polytope is the convex set $1\times \Delta_\infty$
in $\Lambda'_{\mathbb R}\cong \mathbb R\times \Lambda_{\mathbb R}$. Any face  of 
$1\times \Delta_\infty$ is of the form $1\times \tau$ for a face $\tau$ of $\Delta_\infty$. 
Let $e(1\times \tau)=(e(1\times \tau)_j)\in \mathbb V':=\prod_{j=0}^N\mathrm{End}(V_j)$ be defined by 
\begin{eqnarray*}
e(1\times \tau)_0&=&\begin{cases}
1&\hbox{if }(1,0)\in 1\times \tau, \hbox{ that is, } 0\in \tau,\\
0&\hbox{otherwise},
\end{cases}\\
e(1\times \tau)_j\Big|_{ V_j(\lambda)}&=&\begin{cases}
\mathrm{id}_{V_j(\lambda)}&\hbox{if }(1, \lambda)\in 1\times \tau, \hbox{ that is, } \lambda\in \tau,\\
0&\hbox{otherwise}.
\end{cases}
\end{eqnarray*}

\begin{proposition}\label{prop:orbitdecomX} Keep the notation in \ref{Homogeneization}. 
Let $i_0$ be the closed immersion
$$i_0: \mathbb V\to \mathbb V', \quad v\mapsto (0, v).$$
We have the orbit decompositions
$$X'_{\bar k}=\bigsqcup_{\tau\prec \Delta_\infty} (\mathbb G_{m, \bar k}\times G_{\bar k})e(1\times \tau) 
(\mathbb G_{m, \bar k}\times G_{\bar k}),\quad 
i_0^{-1}(X'_{\bar k})=\bigsqcup_{0\not\in \tau\prec \Delta_\infty} G_{\bar k}e(\tau) G_{\bar k} .$$
\end{proposition} 

\begin{proof} By the condition (1) in Proposition \ref{ass2}, $\mathbb G_{m,\bar k}\times G_{\bar k}$ is open dense 
in $\tilde Y'_{\bar k}$. The morphism $\tilde Y' \to X'$ is proper dominant and hence surjective. So $X'_{\bar k}$ is the 
scheme theoretic image of $\iota'_{\bar k}$. The first assertion follows from Proposition \ref{prop:orbitdecomgeneral}. 
We have $e(1\times \tau) \in X'\cap (0\times\mathbb V)$ if and only if $0\not\in \tau$. The 
second assertion follows from the first one. 
\end{proof}   

\begin{lemma}\label{lm:basechangesigma} 
Let $k$ be a field, let $G$ be a reductive algebraic group split over $k$, and let $H=G\times G$ act on 
$G$ via the left and the right multiplication of $G$. Suppose $\sigma: \tilde Y\to Y$ is a $k$-morphism
satisfying the following conditions:
\begin{enumerate}[(1)]
\item Both $\tilde Y$ and $Y$ are $k$-schemes with $H$-action, and contain $G$ as a dense open 
subscheme.
\item $\sigma$ is equivariant proper and induces identity on $G$. 
\item $\tilde Y$ is smooth and geometrically connected.
\end{enumerate}
Then $R^i\sigma_*\omega_{\tilde Y}=0$ for all $i\geq 1$, and $\sigma_*\omega_{\tilde Y}$ is Cohen-Macaulay. 
\end{lemma}

\begin{proof} By base change to an algebraic closure of $k$, we may assume $k$ is algebraically closed. 
Let $p: \mathcal Y\to Y$ be the normalization of 
the reduced scheme associated to $Y$. Then $\sigma:\tilde Y\to Y$ factors through $p$ via a morphism
$g: \tilde Y\to \mathcal Y$. By \cite[6.2.5]{Frob-split}, we can choose a proper equivariant 
morphism $f:Z \to \tilde Y$ 
such that $Z$ is a toroidal equivariant embedding of $G$ in the sense of
\cite[6.2.2]{Frob-split}. 
$$\begin{tikzcd}
Z\arrow[r,"f"]&\tilde Y\arrow[dr,"\sigma"]\arrow[r,"g"]& \mathcal Y\arrow[d,"p"]\\
&&Y
\end{tikzcd}$$
By \cite[6.2.8]{Frob-split}, we have 
\begin{eqnarray}\label{van}
Rf_*\mathcal O_Z\cong  \mathcal O_{\tilde Y}.
\end{eqnarray}
By the Grothendieck duality theorem (\cite[Theorem 3.4.4]{Serredual}), we have 
\begin{eqnarray*}
&&Rf_*\omega_Z\cong Rf_*R\mathcal Hom_{\mathcal O_Z}(\mathcal O_Z,
f^!\omega_{\tilde Y})\cong R\mathcal Hom_{\mathcal O_{\tilde Y}}
(Rf_*\mathcal O_Z, \omega_{\tilde Y})\\
&\cong& R\mathcal Hom_{\mathcal O_{\tilde Y}}(\mathcal O_{\tilde Y}, \omega_{\tilde Y})
\cong \omega_{\tilde Y}.
\end{eqnarray*}
Here we have
$$f^!\omega_{\tilde Y}\cong \omega_Z$$
since $Z$ and $\tilde Y$ are smooth of the same dimension. 
So we have 
$$R\sigma_*\omega_{\tilde Y}\cong R\sigma_*Rf_*\omega_Z\cong p_*R(gf)_*\omega_Z.$$
By \cite[6.2.8]{Frob-split}, we have 
$$R^i(gf)_*\omega_Z=0$$
for all $i\geq 1$. So $R^i\sigma_*\omega_{\tilde Y}=0$ for all $i\geq 1$. 
Moreover, by \cite[6.2.8]{Frob-split} we have $R(gf)_*\mathcal O_Z\cong 
\mathcal O_{\mathcal Y}$. Combined with (\ref{van}), we get
$$R\sigma_*\mathcal O_{\tilde Y}\cong p_*Rg_*\mathcal O_{\tilde Y}\cong p_*Rg_*Rf_*\mathcal O_Z\cong p_*\mathcal O_{\mathcal Y}.$$
Thus $R^i\sigma_*\mathcal O_{\tilde Y}=0$ for all $i\geq 1$ and hence
$$R\sigma_*\mathcal O_{\tilde Y}\cong \sigma_*\mathcal O_{\tilde Y}.$$
For any affine open subset $U$ of $Y$, let us prove $\sigma_{U*}\omega_{\sigma^{-1}(U)}$ is Cohen-Macaulay,
where $\sigma_U:\sigma^{-1}(U)\to U$ is the morphism induced by $\sigma$. 
Choose a closed immersion $i: U\to S$ so that $S$ is smooth. By \cite[Proposition IV B 11]{Serre2}, it suffices to show 
$i_*\sigma_{U*}\omega_{\sigma^{-1}(U)}$ is Cohen-Macauley. By the Grothendieck duality theorem, we have 
\begin{align*}
i_*\sigma_{U*}\mathcal O_{\sigma^{-1}(U)}&\cong i_*R\sigma_{U*}\mathcal O_{\sigma^{-1}(U)}\cong
R(i\sigma_U)_*R\mathcal Hom_{\mathcal O_{\sigma^{-1}(U)}}
(\omega_{\sigma^{-1}(U)}, \omega_{\sigma^{-1}(U)})\\
&\cong R(i\sigma_U)_*R\mathcal Hom_{\mathcal O_{\sigma^{-1}(U)}}
(\omega_{\sigma^{-1}(U)}, (i\sigma_U)^!\omega_S)
[\mathrm{dim}\,S-\mathrm{dim}\,\tilde Y]\\
&\cong R\mathcal Hom_{\mathcal O_S}(i_*\sigma_{U*}\omega_{\sigma^{-1}(U)}, \omega_S)[\mathrm{dim}\,S-\mathrm{dim}\,\tilde Y].
\end{align*}
It follows that $R^i\mathcal Hom_{\mathcal O_S}(i_*\sigma_{U*}\omega_{\sigma^{-1}(U)}, \omega_S)=0$ for 
$i\not =\mathrm{dim}\,S-\mathrm{dim}\,\tilde Y$. By \cite[Tag 0B5A]{stacksproject}, $i_*\sigma_{U*}\omega_{\sigma^{-1}(U)}$
is Cohen-Macaulay. 
\end{proof}

\begin{proposition}\label{prop:cohoverDedekind} Let $R$ be a discrete valuation ring, let $G$ be a split reductive group $R$-scheme, and
let $H=G\times G$ act on $G$ via the left and the right multiplication of $G$. Suppose $\sigma: \tilde Y\to Y$ is an $R$-morphism 
satisfying the following conditions:
\begin{enumerate}[(1)]
\item Both $\tilde Y$ and $Y$ are $R$-schemes with $H$-action, and contain $G$ as an open subscheme.
\item $\sigma$ is equivariant proper dominant and induces identity on $G$.  
\item $\tilde Y\to\mathrm{Spec}\,R$ is smooth and has geometrically connected fibers. 
\end{enumerate}
Let $\mathfrak m$ be the maximal ideal of $R$,  
and let $\sigma_n: \tilde Y_n\to Y_n$ be the base change of $\sigma:\tilde Y\to Y$ by $R\to R/\mathfrak m^{n+1}$. 
\begin{enumerate}[(i)]
\item We have $R^i\sigma_{n*}\omega_{\tilde Y_n}=0$ for all $i\geq 1$ and 
$$\sigma_{n*}\omega_{\tilde Y_{n}}\otimes_{R/\mathfrak m^{n+1}} R/\mathfrak m^n
\cong \sigma_{n-1,*}\omega_{\tilde Y_{n-1}}.$$ 
\item We have $R^i\sigma_* \omega_{\tilde Y}=0$ for all $i\geq 1$, and 
$$\sigma_*\omega_{\tilde Y}\otimes_R R/\mathfrak m^{n+1} \cong \sigma_{n*}\omega_{\tilde Y_{n}}.$$
\end{enumerate}
\end{proposition}

\begin{proof} (i) 
We use induction on $n$.
When $n=0$, this follows from Lemma \ref{lm:basechangesigma}. Assume $n\geq 1$. Choose a generator $\varpi$ for $\mathfrak m$.
We have a long exact sequence
\begin{align*}
0& \to R^0\sigma_{0*}\omega_{\tilde{Y}_{0}}\xrightarrow{\varpi^{n}} R^0\sigma_{n*}\omega_{\tilde{Y}_{n}}
\to R^0\sigma_{n-1,*}\omega_{\tilde{Y}_{n-1}}\to \cdots \\
&\to  R^i\sigma_{0*}\omega_{\tilde{Y}_{0}}\to R^i\sigma_{n*}\omega_{\tilde{Y}_{n}}\to R^i\sigma_{n-1,*}\omega_{\tilde{Y}_{n-1}}\to \cdots
\end{align*} 
By the induction hypothesis we have 
$$R^i\sigma_{0*}\omega_{\tilde{Y}_{0}}=R^i\sigma_{n-1,*}\omega_{\tilde{Y}_{n-1}}=0 \quad(i\geq 1).$$
So we have $R^i\sigma_{n*}\omega_{\tilde{Y}_{n}}=0$ for all $i\geq 1$ and we have a short exact sequence
$$0\to \sigma_{0*}\omega_{\tilde{Y}_{0}}\xrightarrow{\varpi^{n}} \sigma_{n*}\omega_{\tilde{Y}_{n}}\to 
\sigma_{n-1,*}\omega_{\tilde{Y}_{n-1}}\to  0.$$
Taking the composite of the epimorphisms $\sigma_{m*}\omega_{\tilde{Y}_{m}}\twoheadrightarrow 
\sigma_{m-1,*}\omega_{\tilde{Y}_{m-1}}$ $(m\leq n)$, we get an epimorphism 
$$\sigma_{n*}\omega_{\tilde{Y}_{n}}\twoheadrightarrow \sigma_{0*}\omega_{\tilde{Y}_0}.$$
So we may identify the image of $\sigma_{0*}\omega_{\tilde{Y}_{0}}\xrightarrow{\varpi^{n}} \sigma_{n,*}\omega_{\tilde{Y}_{n}}$ 
with the image of $\sigma_{n,*}\omega_{\tilde{Y}_{n}}\xrightarrow{\varpi^{n}} \sigma_{n,*}\omega_{\tilde{Y}_{n}}$.
Hence $\sigma_{n*}\omega_{\tilde{Y}_{n}}\otimes_{R/\mathfrak m^{n+1}} R/\mathfrak m^n\cong \sigma_{n-1,*}\omega_{\tilde{Y}_{n-1}}$.

(ii) By \cite[4.1.5]{EGAIII}, we have 
$$\varprojlim_n R^i\sigma_{n*}\omega_{\tilde{Y}_{n}} \cong (R^i\sigma_*\omega_{\tilde{Y}})^\wedge.$$ Combined with (i), 
we get
$\sigma_*\omega_{\tilde{Y}}\otimes_R R/\mathfrak m^{n+1}\cong \sigma_{n*}\omega_{\tilde{Y}_{n}}$, and 
the restriction of $R^i\sigma_*\omega_{\tilde{Y}}$ to the special fiber of $Y\to \mathrm{Spec}\,R$ vanish for all $i\geq 1$. By 
Lemma \ref{lm:basechangesigma} and the flat base change theorem, the restriction of 
$R^i\sigma_*\omega_{\tilde{Y}}$ to the generic fiber of $Y\to \mathrm{Spec}\,R$ also vanish for all $i\geq 1$.
\end{proof}   

\begin{corollary}\label{cor:usedin2} Let $\bar\iota: \tilde Y\to \mathbb P$ be the morphism 
in Proposition \ref{prop:inv0inf}, and let $\bar\iota_n: \tilde Y_n\to \mathbb P_n$ be the base change of 
$\bar\iota:\tilde Y\to \mathbb P$ by $R\to R/\mathfrak m^{n+1}$. 
\begin{enumerate}[(i)]
\item We have $R^i\bar\iota_{n*}\omega_{\tilde Y_n}=0$ for all $i\geq 1$ and 
$$\bar\iota_{n*}\omega_{\tilde Y_{n}}\otimes_{R/\mathfrak m^{n+1}} R/\mathfrak m^n\cong \bar\iota_{n-1,*}\omega_{\tilde Y_{n-1}}.$$ 
\item We have $R^i\bar\iota_* \omega_{\tilde Y}=0$ for all $i\geq 1$, and 
$$\bar\iota_*\omega_{\tilde Y}\otimes_R R/\mathfrak m^{n+1} \cong \bar\iota_{n*}\omega_{\tilde Y_{n}}.$$
\end{enumerate}
\end{corollary}

\begin{proof} Note that $\bar\iota$ is a composite of $\sigma:\tilde Y\to\overline Y$ and 
a finite morphism $\overline Y\to \mathbb P$. We can apply Proposition \ref{prop:cohoverDedekind}
to $\sigma:\tilde Y\to\overline Y$ and the base change theorem to the finite morphism $\overline Y\to \mathbb P$. 
\end{proof}

\begin{corollary}\label{cor:usedin4}  Keep the notation in \ref{Homogeneization} and Proposition \ref{prop:omega}. 
Denote the composite $Y'\to X'\to \mathbb V'$
by $f'$. Then $f'_{k*}\omega_{Y'_k}$ is Cohen-Macaulay, and $(f'\omega_{Y'})\otimes_R k\cong f'_{k*}\omega_{Y'_k}$.
\end{corollary}

\begin{proof} Note that $f'$ is a finite morphism. The Cohen-Macaulay property follows from Lemma \ref{lm:basechangesigma} 
and \cite[Proposition IV B 11]{Serre2}. The second assertion follows from Proposition \ref{prop:cohoverDedekind} applied 
to $\sigma':\tilde Y'\to Y'$ and the base change theorem for the finite morphism $f'$. 
\end{proof}

\end{document}